\documentclass[a4paper,USenglish,cleveref, autoref]{lipics-v2019}

\usepackage{todonotes}
\usepackage{mathtools}
\usepackage[utf8]{inputenc}
\usepackage{amssymb,amsmath}
\usepackage{fancyhdr}
\usepackage{stmaryrd}
\usepackage{xspace}
\usepackage{tikz}
\usepackage{amsthm}
\usepackage{amsfonts}
\usepackage{amssymb}
\usepackage{amsmath}
\usepackage{mathdots}

\usepackage{multicol}

\makeatletter
\newcommand{\thickhline}{%
    \noalign {\ifnum 0=`}\fi \hrule height 1pt
    \futurelet \reserved@a \@xhline
}
\newcolumntype{"}{@{\hskip\tabcolsep\vrule width 1pt\hskip\tabcolsep}}
\makeatother

\newcommand{\ccinterval}[2]{[{#1}..{#2}]}
\newcommand{\cointerval}[2]{[{#1}..{#2}-1]}     
\newcommand{\set}[2]{\left\{#1\; \middle|\; #2\right\}}

\newcommand{\oneset}[1]{\left\{\mathinner{#1}\right\}}

\newcommand{\abs}[1]{\left|\mathinner{#1}\right|}

\newcommand{\floor}[1]{\left\lfloor\mathinner{#1} \right\rfloor}

\newcommand{\N}{\ensuremath{\mathbb{N}}}
\newcommand{\Z}{\ensuremath{\mathbb{Z}}}

 \newcommand{\DSPACE}{\ensuremath{\mathsf{DSPACE}}\xspace} 
 \newcommand{\PSPACE}{\ensuremath{\mathsf{PSPACE}}\xspace} 
 \newcommand{\EXPSPACE}{\ensuremath{\mathsf{EXPSPACE}}\xspace} 
 \newcommand{\NP}{\ensuremath{\mathsf{NP}}\xspace} %
 \newcommand{\LOGCFL}{\ensuremath{\mathsf{LOGCFL}}\xspace} %
 \newcommand{\LOGDCFL}{\ensuremath{\mathsf{LOGDCFL}}\xspace} %
 \newcommand{\DLOGTIME}{\ensuremath{\mathsf{DLOGTIME}}\xspace} %
 \newcommand{\DTIME}{\ensuremath{\mathsf{DTIME}}\xspace} %
  \newcommand{\ATIME}{\ensuremath{\mathsf{ATIME}}\xspace} %
 
 \newcommand{\DLINTIME}{\ensuremath{\mathsf{DLINTIME}}\xspace} %
 
  \newcommand{\ALOGTIME}{\ensuremath{\mathsf{ALOGTIME}}\xspace} %
 \renewcommand{\L}{\ensuremath{\mathsf{LOGSPACE}}\xspace} %
 \newcommand{\shortL}{\ensuremath{\mathsf{L}}\xspace} %
 \newcommand{\TC}{\ensuremath{\mathsf{TC}^0}\xspace}
 \newcommand{\ACC}{\ensuremath{\mathsf{ACC}^0}\xspace} %
 \newcommand{\Tc}[1]{\ensuremath{\mathsf{TC}^{#1}}\xspace}
 \newcommand{\Ac}[1]{\ensuremath{\mathsf{AC}^{#1}}\xspace}
 \newcommand{\Nc}[1]{\ensuremath{\mathsf{NC}^{#1}}\xspace}

 \newcommand{\NC}{\ensuremath{\mathsf{NC}}\xspace}
 
 \renewcommand{\P}{\ensuremath{\mathsf{P}}\xspace}
 \newcommand{\Ppoly}{\ensuremath{\mathsf{P\!/poly}}\xspace}

\newcommand{\Pad}{\mathsf{Pad}} 
\newcommand{\APTIME}{\mathsf{APTIME}} 
\newcommand{\coNP}{\mathsf{coNP}} 
\newcommand{\NSPACE}{\mathsf{NSPACE}} 
\newcommand{\polyL}{\mathsf{polyL}} 
\newcommand{\Ptime}{\mathsf{P}} 

\newcommand{\supp}{\operatorname{supp}} 

\newcommand{\val}{\mathrm{val}}
 \newcommand{\LogCFL}{\ensuremath{\mathsf{LogCFL}}}
\newcommand{\nand}{\ensuremath{\mathsf{nand}}}
\newcommand{\Mod}[1]{\ensuremath{\mathsf{Mod}_{#1}}}
\newcommand{\coMod}[1]{\ensuremath{\mathsf{coMod}_{#1}}}
\newcommand{\leaf}{\mathrm{leaf}}
\newcommand{\LEAF}{\mathsf{LEAF}}
\newcommand{\bLEAF}{\mathsf{bLEAF}}

\renewcommand{\phi}{\varphi}
\newcommand{\eps}{\varepsilon}

\newcommand{\Oh}{\mathcal{O}}

\newcommand{\cP}{\mathcal{P}}

\newcommand{\CompWP}{\textsc{CompressedWP}\xspace}
\newcommand{\WP}{\textsc{WP}\xspace}

\newcommand{\Sym}[1]{\mathrm{Sym}({#1})}

\newcommand{\Aut}{\mathrm{Aut}}

\newcommand{\sse}{\subseteq}

\newcommand\ie{i.\,e., }
\newcommand\Wlog{W.\,l.\,o.\,g.\ }

\newcommand\eg{e.\,g.\xspace}

\renewcommand{\theenumi}{\textup{(\roman{enumi})}}

\renewcommand{\theenumiii}{\textup{\arabic{enumiii}.}}

\newbox\pairbox
\def\pair<#1>{{\mathsurround=0pt
		\setbox\pairbox\hbox{$\left\langle#1\right\rangle$}
		\left\langle\kern-0.35\ht\pairbox
		\copy\pairbox\kern-0.35\ht\pairbox\right\rangle}}

\title{Groups with ALOGTIME-hard word problems and PSPACE-complete
compressed word problems}

\author{Laurent Bartholdi}{ENS Lyon, Unité de Mathématiques Pures et Appliquées, France \and Universität Göttingen, Mathematisches Institut, Germany}{laurent.bartholdi@gmail.com}{https://orcid.org/0000-0002-1243-6384}{}
\author{Michael Figelius}{Universit{\"a}t Siegen, Germany }{figelius@eti.uni-siegen.de}{https://orcid.org/0000-0002-0407-0597}{Funded by DFG project LO 748/12-1.}
\author{Markus Lohrey}{Universit{\"a}t Siegen, Germany }{lohrey@eti.uni-siegen.de}{http://orcid.org/0000-0002-4680-7198}{Funded by DFG project LO 748/12-1.}
\author{Armin Wei\ss}{%
	Universität Stuttgart,
	Institut für Formale Methoden der Informatik (FMI), Germany}{armin.weiss@fmi.uni-stuttgart.de}{https://orcid.org/0000-0002-7645-5867}{Funded by DFG project DI 435/7-1.}

\titlerunning{ALOGTIME-hard word problems and PSPACE-complete
compressed word problems}

\authorrunning{L. Bartholdi, M. Figelius, M. Lohrey and A. Wei\ss}

\Copyright{Laurent Bartholdi, Michael Figelius, Markus Lohrey and Armin Wei\ss}



\category{}

\supplement{}

\acknowledgements{The authors are grateful to Schloss Dagstuhl and the organizers of Seminar 19131 for the invitation, where this work began.}

\nolinenumbers 

\hideLIPIcs  

\EventEditors{}
\EventNoEds{0}
\EventLongTitle{}
\EventShortTitle{}
\EventAcronym{}
\EventYear{}
\EventDate{}
\EventLocation{}
\EventLogo{}
\SeriesVolume{}
\ArticleNo{}

\begin{document}

\maketitle

\begin{abstract}
  We give lower bounds on the complexity of the word problem of
  certain non-solvable groups: for a large class of non-solvable
  infinite groups, including in particular free groups, Grigorchuk's
  group and Thompson's groups, we prove that their word problem is
  $\Nc1$-hard. For some of these groups (including Grigorchuk's
  group and Thompson's groups) we prove that the compressed word problem (which is equivalent
  to the circuit evaluation problem) is \PSPACE-complete.
\end{abstract}

\section{Introduction}

The {\em word problem} of a finitely generated group $G$ is the most fundamental algorithmic problem in group theory: given a word over the generators of $G$, the question is whether this word represents the identity of $G$.
The original motivation for the word problem came from topology and group theory \cite{dehn11}, within Hilbert's ``Entscheidungsproblem''. Nevertheless, it also played a role in early computer science when Novikov and Boone constructed finitely presented groups with an undecidable word problem  \cite{boone59,nov55}. Still, in many classes of groups it is (efficiently) decidable, a prominent example being the class of linear groups: Lipton and Zalcstein \cite{LiZa77} (for linear groups over a field of characteristic zero)  and Simon \cite{Sim79}
(for linear groups over a field of prime characteristic) showed that their word problem is in \L.

The class \Nc{1} consists of those languages that are accepted by families of boolean circuits of logarithmic depth. When combined with certain uniformity conditions it yields the subclass \ALOGTIME which is is contained in \L~--- so it is a very small complexity class of problems efficiently solvable in parallel.
A striking connection between the word problem for groups and complexity theory was established by Barrington \cite{Barrington89}: for every finite non-solvable
group $G$, the word problem of $G$ is $\Nc1$-complete.
Moreover, the reduction is as simple as it could be: every output bit depends on only one input bit. Thus, one can say that \Nc1 is completely characterized via group theory. Moreover, this idea has been extended to characterize \ACC by solvable monoids \cite{bt88jacm}. 
On the other hand, the word problem of a finite $p$-group is in $\ACC[p]$, so Smolensky's lower bound \cite{Smolensky87} implies that it is strictly easier than the word problem of a finite non-solvable group.

Barrington's construction is based on the observation that an {\sf and}-gate can be simulated by a commutator. 
This explains the connection to non-solvability. In this light it seems natural that the word problem of finite $p$-groups is not \Nc{1}-hard: they are all nilpotent, so iterated commutators eventually become trivial.
For infinite groups, a construction similar to Barrington's was used by Robinson \cite{Robinson93phd} to show that the word problem of a non-abelian free group is \Nc1-hard. Since by \cite{LiZa77} the word problem of a free group is in \L, the complexity is narrowed down quite precisely (although no completeness results has been shown so far).

The first contribution of this paper is to identify the essence of Barrington's and
Robinson's constructions. For this we introduce a strengthened condition of
non-solvability, which we call \emph{SENS (strongly efficiently non-solvable)}; see~\cref{def:SENS}. In a SENS group there 
are balanced nested commutators of arbitrary depth and whose word length grows at most exponentially. We also 
introduce {\em uniformly SENS} groups, where these balanced commutators are efficiently computable in a certain sense.
We then follow Barrington's arguments and show that every for every (uniformly) SENS group the word problem is hard for (uniform)  \Nc{1}
(Theorems~\ref{thm:NC1hard} and \ref{thm:uniNC1hard}).
That means that for every non-solvable group $G$, the word problem for $G$ is \Nc{1}-hard, unless the word length of the $G$-elements witnessing the non-solvability grows very fast
(we also give in Example~\ref{ex:nonENS} a non-solvable group in which the latter happens).

\theoremstyle{plain}
\newtheorem{corollaryA}{Corollary}
\renewcommand*{\thecorollaryA}{\Alph{corollaryA}}
\newtheorem{theoremA}[corollaryA]{Theorem}

Finite non-solvable groups and non-abelian free groups are easily seen to be uniformly SENS.
We go beyond these classes and present a general criterion that implies the uniform SENS-condition.
Using this criterion we show that {\em Thompson's groups} \cite{CaFlPa96} and {\em weakly branched self-similar groups} \cite{BartholdiGrigorchukSunic03,Nekrashevych05}
are uniformly SENS. As a corollary we get:


\begin{corollaryA} \label{coro:A}
	The word problems for the following groups are hard for $\ALOGTIME$:\begin{itemize}
		\item the three Thompson's groups $F$, $T$, and $V$, 
		\item weakly branched self-similar groups with a finitely generated branching subgroup.
	\end{itemize} 
\end{corollaryA}

Thompson's groups $F < T < V$ (introduced in 1965) belong due to their unusual properties to the most intensively studied
infinite groups. From a computational perspective it is interesting to note that all three Thompson's groups are co-context-free (i.e., the set of all non-trivial words
over any set of generators is a context-free language) \cite{LehSchw07}. This implies that the word problems for Thompson's groups are in \LOGCFL. To the best of our knowledge
no better upper complexity bound is known.  Weakly branched groups form an important subclass of the self-similar groups \cite{Nekrashevych05}, containing several celebrated groups like the Grigorchuk group (the first example of a group with intermediate word growth) and the Gupta-Sidki groups.
 We also show that the word problem for contracting
self-similar groups is in \L. This result is well-known, but to the
best of our knowledge no proof appears in the literature. 
The Grigorchuk group as well as the Gupta-Sidki groups are contracting and have finitely generated branching subgroups.

Another corollary of \cref{thm:uniNC1hard} is the following dichotomy result for finitely generated linear groups:
for every finitely generated linear group the word problem is in \DLOGTIME-uniform \TC or \ALOGTIME-hard (\cref{thm-linear-groups}).
To prove this we use Tits alternative (every finitely generated linear group either contains a free group of rank two or is virtually solvable) \cite{Tits72}
together with a result from \cite{KonigL15} stating that the word problem for a finitely generated solvable linear group is in 
\DLOGTIME-uniform \TC.

In the second part of the paper we study the {\em compressed word problem} \cite{Lohrey14compressed}. This is a succinct version of the word problem, where 
the input word is represented by a so-called straight-line program. A straight-line program is a context-free grammar that produces
exactly one string. The length of this string can be exponentially larger than the size of the straight-line program. 
The compressed word problem for a  finitely generated group $G$  is equivalent to the {\em circuit evaluation problem} for $G$. In the latter the 
input is a circuit where the input gates are labelled with generators of $G$ and the internal gates compute the product of their inputs. 
There is a distinguished output gate, and the question is whether this output gate evaluates to the group identity. 
For finite groups (and also monoids), the circuit evaluation problem has been studied in \cite{BeMcPeTh97}.
The circuit viewpoint also links the compressed word problem to the famous polynomial identity testing problem (the question whether
an algebraic circuit over a polynomial ring evaluates to the zero-polynomial); see \cite{ShpilkaY10} for a survey: it is shown in \cite{Lohrey14compressed} that
the compressed word problem for the group $\mathrm{SL}_3(\Z)$ is equivalent to  polynomial identity testing problem with respect 
to polynomial time reductions \cite[Theorem~4.16]{Lohrey14compressed}.

From a group theoretic viewpoint, the compressed word problem is interesting not only because group elements are naturally represented as straight line programs, but also because several classical (uncompressed) word problems
reduce to compressed word problems. For instance, the word problem for a finitely generated subgroup of $\Aut(G)$ reduces to the compressed
word problem for $G$ \cite[Theorem~4.6]{Lohrey14compressed}. Similar statements hold for certain group extensions \cite[Theorems~4.8 and~4.9]{Lohrey14compressed}. This motivates the search for groups in which the compressed word
problem can be solved efficiently.
For the following groups, the compressed word problem can be solved in polynomial time: finitely generated nilpotent groups \cite{KonigL15}
(for which the compressed word problem can be even solved in $\Nc{2}$), hyperbolic groups \cite{HoltLS19}
(even groups that are hyperbolic relative to a collection of free abelian subgroups \cite{HoltRees20})
 and virtually special groups \cite{Lohrey14compressed}. The latter are defined
as finite extensions of subgroups of right-angled Artin groups and form a very rich 
class of groups containing for instance Coxeter groups \cite{HagWi10}, fully residually free groups \cite{Wis09}
and fundamental groups of hyperbolic 3-manifolds \cite{Agol12}. 
Moreover, for finitely generated linear groups the compressed
word problem belongs to $\mathsf{coRP}$ (complement of randomized polynomial time).

In this paper, we are mainly interested in groups in which the compressed word problem is hard or intractable.
Indeed, it is known that the compressed word problem
for non-solvable finite groups and non-abelian free groups is \P-complete \cite{BeMcPeTh97,Loh06siam}. The proofs for these results use again the above mentioned constructions of Barrington and Robinson. Starting from this observation we introduce a variant of the uniform SENS-condition and show that every group satisfying this condition has a \P-hard compressed word problem. However, we go even further:
Recently, W\"achter and the fourth author  constructed an automaton group (a finitely generated group of tree automorphism, where the action of
generators is defined by a  Mealy automaton) with a \PSPACE-complete word problem and
\EXPSPACE-complete compressed word problem \cite{WaeWei19}~-- thus, the compressed word problem is provably more difficult than the word problem. The group arises from a quite technical construction; in particular one cannot 
call this group natural. Here, we exhibit several natural groups (that were intensively studied in other parts of mathematics) with a \PSPACE-complete
compressed word problem and a word problem in \L:  
\begin{corollaryA}
	The compressed word problem for the following groups is \PSPACE-complete: wreath products $G \wr \Z$ where $G$ is finite non-solvable or 
	free of rank at least two,  Thompson's groups, the Grigorchuk group, and all Gupta-Sidki groups.
\end{corollaryA}  
The group theoretic essence in order to get \PSPACE-hardness is a certain self-embedding property: we need a group 
$G$ such that a wreath product $G \wr A$ embeds into $G$ for some $A \neq 1$. Thompson's group $F$ has
this property for $A = \Z$ \cite{GubaSapir99}. For a weakly branched group $G$ that satisfies an additional technical condition (the branching subgroup 
$K$ of $G$ is finitely generated and has elements of finite order) we show that one can take $A =\Z/p$ for some $p \geq 2$. The above self-embedding property
allows us to carry out a subtle reduction from the leaf language class defined by the group $G$ to the compressed word problem for $G$.

\section{General notations}\label{sec:prelims}

For $a,b \in \mathbb{Z}$ we write $\ccinterval ab$ for the interval
$\{ z \in \mathbb{Z} \mid a \leq z \leq b\}$. 
We use common notations from formal language theory. In particular, we use $\Sigma^*$ to denote the set of words over an alphabet $\Sigma$ including the {\em empty word} $\eps$.
Let $w = a_0 \cdots a_{n-1} \in \Sigma^*$ be a word over $\Sigma$
($n\geq 0$, $a_0, \ldots, a_{n-1} \in \Sigma$).
The \emph{length} of $w$ is $|w| = n$. We write $\Sigma^{\leq d}$ for $\{ w \in \Sigma^* \mid \abs{w} \leq d\}$ and 
$\Sigma^{< d}$ for $\{ w \in \Sigma^* \mid \abs{w} < d\}$.
For a letter $a \in \Sigma$ let $|w|_a = |\{ i \mid a = a_i \}|$
be the number of occurrences of $a$ in $w$.
For $0 \leq i < n$ let $w[i] = a_i$ and
for $0 \leq i \leq j < n$ let $w[i:j] = a_i a_{i+1} \cdots a_j$.
Moreover $w[:i] = w[0:i]$. Note that in the notations $w[i]$ and $w[i:j]$ we take $0$ as the first position in $w$.
This will be convenient later.

The lexicographic order on $\mathbb{N}^*$ is defined as follows:
a word $u \in \mathbb{N}^*$ is lexicographically smaller than a word $v \in \mathbb{N}^*$ if  
either $u$ is a prefix of $v$ or there exist $w,x,y \in
\mathbb{N}^*$ and $i,j \in \mathbb{N}$ such that $u = wix$, $v = wjy$, and $i < j$.

A {\em finite order tree} is a finite set $T \subseteq \mathbb{N}^*$ such that for all $w \in \mathbb{N}^*$, $i \in
\mathbb{N}$: if $wi \in T$, then $w,wj  \in T$ for every $0 \leq j < i$.
The set of {\em children} of $u \in T$ is $u \mathbb{N} \cap T$.
A node $u \in T$ is a leaf of $T$ if it has no children.
A \emph{complete binary tree} is a subset
$T \subseteq \{0,1\}^*$ such that $T = \{ s \in \{0,1\}^* \mid |s| \leq k \}$ for some $k \geq 0$ where $k$ is called the \emph{depth} of $T$.

The boolean function $\nand : \{0,1\}^2 \to \{0,1\}$ (negated {\sf and}) is defined by
$\mathsf{nand}(0,0) =  \mathsf{nand}(0,1) = \mathsf{nand}(1,0) = 1$ and   $\mathsf{nand}(1,1)=0$.
Note that the standard boolean functions {\sf not} and binary {\sf and} and {\sf or} can be expressed in terms of \nand.

\section{Groups}

We assume that the reader is familiar with the basics of group theory, see e.g. \cite{HRR2017,Rot95} for more details.
Let $G$ be a group. We always write $1$ for the group identity element.
The group $G$ is called {\em finitely generated} if there exist a finite
set $S$ and a surjective homomorphism of the free group over $S$ onto
$G$. In this situation, the set $\Sigma = S \cup S^{-1}\cup\oneset{1}$ is
our preferred generating set for $G$ and we have a surjective monoid
homomorphism $\pi\colon \Sigma^* \to G$. The symbol $1$ is useful for
padding. We call the generating set $\Sigma$ \emph{standard}. We have a natural involution on words over $\Sigma$ 
defined by $(a_1\cdots a_n)^{-1} = a_n^{-1}\cdots a_1^{-1}$ for $a_i \in \Sigma$ (which is the same as forming inverses in the group).
For words $u,v \in \Sigma^*$ we usually say that $u = v$ in $G$ or
$u =_G v$ in case $\pi(u) = \pi(v)$. 
For group elements $g,h \in G$ or words $g,h\in \Sigma^*$ 
 we write $g^h$ for the {\em conjugate} $h^{-1} g h$ and $[h, g]$ for the {\em commutator} $h^{-1} g^{-1} hg$. We call $g$ a \emph{$d$-fold nested commutator}, if $d= 0$ or $g = [h_1, h_2]$ for $(d-1)$-fold nested commutators $h_1,h_2$.

A \emph{subquotient} of $G$ is a quotient of a subgroup of $G$. The {\em center} of $G$, $Z(G)$ for short, is the set of all elements $g \in G$ that
commute with every element from $G$. The center of $G$ is a normal subgroup of $G$.

 The {\em word problem} for the
finitely generated group $G$, $\WP(G)$ for short, is defined as
follows:
\begin{description}
	\item[Input:]   a word $w \in \Sigma^*$.
	\item [Question:] does $w=_G 1$ hold? 
\end{description}
We will also write $\WP(G,\Sigma)$ for the set $\{ w \in \Sigma^* \mid w=_G 1 \}$.

The word problem may be stated for any group whose elements may
be written as words over a finite alphabet. This applies to
subquotients $H/K$ of $G$ (also if $H$ is not finitely generated): given a word $w\in\Sigma^*$ with the
guarantee that it belongs to $H$, does it actually belong to $K$? Note
that the decidability of this problem depends on the actual choice of
$H$ and $K$, not just on the isomorphism type of $H/K$.

 We will consider groups $G$ that act on a set $X$
 on the left or right.
For $g \in G$ and $x \in X$ we write $x^g \in X$ (resp., ${}^g\!x$) for the result of a right (resp., left) action.
A particularly important case arises when $G = \Sym X$ is the symmetric
group on a set $X$, which acts on $X$ on the right.

\subsection{Wreath products} \label{sec-wreath}

A fundamental group construction that we shall use is the \emph{wreath product}: 
given groups $G$ and $H$ acting on the right on sets $X$ and $Y$ respectively, their
\emph{wreath product} $G\wr H$ is a group acting on $X\times Y$. 
We start with the restricted direct product $G^{(Y)}$ (the base group) of all mappings $f : Y \to G$ having finite support
$\supp(f) = \{ y \mid f(y) \neq 1 \}$ with the operation of pointwise multiplication. 
The group $H$ has a natural left action on
$G^{(Y)}$: for $f \in G^{(Y)}$ and $h \in H$, we define ${}^h\!f \in G^{(Y)}$ by $({}^h\!f)(y) = f(y^h)$. 
The corresponding semidirect product $G^{(Y)} \rtimes H$ is the \emph{wreath product} $G \wr H$.  In other words:
\begin{itemize}
\item
Elements of $G \wr H$ are pairs $(f,h) \in G^{(Y)} \times H$ and we simply write $fh$ for this pair.
\item
The multiplication in $G \wr H$ is defined as follows:
Let $f_1h_1, f_2h_2 \in G \wr H$. Then
$f_1h_1f_2h_2 = f_1 {}^{h_1}\!\!f_2 h_1h_2$, where
the product $f_1 {}^{h_1}\!\!f_2\colon y \mapsto f_1(y) f_2(y^{h_1})$ is the pointwise product.
\end{itemize}
The wreath product $G \wr H$ acts on $X \times Y$ by $(x,y)^{fh} = (x^{f(y)},y^h)$.
The wreath product defined above is also called the {\em (restricted) permutational wreath product}. 
There is also the variant where $G=X$ and $H=Y$ and both groups act on themselves by right-multiplication, which is called
the {\em (restricted) regular wreath product} (or {\em standard wreath product}). A subtle point is that the permutational wreath product is an associative operation whereas
the regular wreath product is in general not. The term ``restricted'' refers to the fact that the base group is $G^{(Y)}$, i.e., only finitely supported mappings are taken into account.
If $G^{(Y)}$ is replaced by $G^Y$ (i.e., the set of all mappings from $Y$ to $G$ with pointwise multiplication), then one speaks of an unrestricted wreath product.
For $Y$ finite this makes of course no difference.
There will be only two situations (Examples~\ref{ex:nonENS} and~\ref{ex:nonENS2}) where we need an unrestricted wreath product.
The action of $G$ on $X$ in the permutational wreath product is usually not important for us, but it is nice to have an associative
operation. For the right group $H$, we will only make use of the following cases: 
\begin{itemize}
\item $H = \Sym Y$ acting on $Y$,
\item $H$ a (finite or infinite) cyclic group acting on itself.
\end{itemize}
Thus, if $H$ is cyclic, the permutational wreath product and the regular wreath product (both denoted by $G \wr H$) coincide. Nevertheless, be aware that $G \wr (H \wr H) = (G \wr H) \wr H$ holds only for the permutational wreath product even if $H$ is cyclic. 
Note that if $G$ is generated by $\Sigma$ and $H$ is generated by $\Gamma$ then $G \wr H$ is 
generated by $\Sigma \cup \Gamma$.

\subsection{Richard Thompson's groups}

In 1965 Richard Thompson introduced three finitely presented groups $F < T < V$ acting  on the
unit-interval, the unit-circle and the Cantor set, respectively.
Of these three groups, $F$ received most attention (the reader should not confuse $F$ with a free group). 
This is mainly due to the still open conjecture that
$F$ is not amenable, which would imply that $F$ is another counterexample to a famous conjecture 
of von Neumann (a counterexample was found by  Ol'shanskii). A standard reference of 
Thompson's groups is \cite{CaFlPa96}. The group $F$ consists of all homeomorphisms of the unit interval that
are piecewise affine, with slopes a power of $2$ and dyadic
breakpoints. Famously, $F$ is generated by two elements $x_0,x_1$ defined by
\[x_0(t)=\begin{cases}2t & \text{ if }0\le t\le\frac14,\\
    t+\frac14 & \text{ if }\frac14\le t\le\frac12,\\
    \frac t2+\frac12 & \text{ if }\frac12\le t\le1,
  \end{cases}\qquad x_1(t)=\begin{cases} t & \text{ if }0\le t\le\frac12,\\
    \frac12+\frac{x_0(2t-1)}2 & \text{ if }\frac12\le t\le 1.\end{cases}
\]
The pattern repeats with $x_{n+1}$ acting trivially on the left
subinterval and as $x_n$ on the right subinterval. We have
$x_{k+1}=x_k^{x_i}$ for all $i<k$. In fact,
\begin{equation} \label{eq:F}
F = \langle x_0, x_1, x_2, \ldots \mid x_k^{x_i}  = x_{k+1} (i<k) \rangle =  \langle x_0, x_1 \mid [x_0 x_1^{-1} \!\!\:,\,  x_0^{-1} x_1 x_0],  [x_0 x_1^{-1} \!\!\:,\, x_0^{-2} x_1 x^2_0]   \rangle .
\end{equation}
The group $F$ is orderable (so in particular torsion-free), its derived subgroup $[F,F]$ is simple and the center of $F$ is trivial.
Important for us is the following fact:

\begin{lemma}[{\cite[Lemma~20]{GubaSapir99}}]\label{lem:GS}
  The group $F$ contains a subgroup isomorphic to $F\wr\Z$.
\end{lemma}
\begin{proof}
  The copy of $\Z$ is generated by $x_0$, and the copies of $F$ in $F^{(\Z)}$ are 
  the conjugates of
  $\langle x_1 x_2 x_1^{-2},x_1^2x_2 x_1^{-3}\rangle$ under powers of
  $x_0$.
\end{proof}
It follows, by iteration, that $F$ contains arbitrarily iterated
wreath products $\Z\wr\cdots\wr\Z$, as well as the limit
$((\cdots\wr\Z)\wr\Z)\wr\Z$.

\subsection{Weakly branched groups}\label{sec:weaklybranched}

We continue our list of examples with an important class of groups
acting on rooted trees. For more details, the
monographs~\cite{BartholdiGrigorchukSunic03,Nekrashevych05}
serve as good references.

Let $X$ be a finite set.\footnote{There will be one occasion (\cref{prop:wreathSENS}), where we will allow an infinite $X$.} 
The free monoid $X^*$ serves as the vertex
set of a regular rooted tree with an edge between $v$ and $v x$ for
all $v\in X^*$ and all $x\in X$. The group $W$ of automorphisms
of this tree naturally acts on the set $X$ of level-$1$ vertices, and permutes the
subtrees hanging from them. Exploiting the bijection
$X^+=X^* \times X$, we thus have an isomorphism
\begin{equation} \label{eq-phi}
\phi\colon W\to W\wr\Sym X = W^{X} \rtimes \Sym X,
\end{equation}
mapping $g\in W$ to elements $f \in W^{X}$ and $\pi\in\Sym X$ as
follows: $\pi$ is the restriction of $g$ to $X\subseteq X^*$, and 
$f$ is uniquely defined by $(x v)^g= x^\pi v^{f(x)}$. We always
write $g@x$ for $f(x)$ and call it the \emph{state (or coordinate) of $g$ at $x$}.
If $X = \ccinterval{0}{k}$, we write $g = \pair< g@0, \dots, g@k >\pi$.

\begin{definition} \label{def:self-sim}
  A subgroup $G\le W$ is \emph{self-similar} if
  $\phi(G)\le G \wr \Sym X$. In other words: the actions on
  subtrees $x X^*$ are given by elements of $G$ itself.
  A self-similar group $G$ is \emph{weakly branched} if there exists a
  non-trivial subgroup $K\le G$ with $\phi(K)\ge K^X$. In other words:
  for every $k\in K$ and every $x\in X$ the element acting as $k$
  on the subtree $x X^*$ and trivially elsewhere belongs to $K$. A
  subgroup $K$ as above is called a \emph{branching subgroup}.
\end{definition}
Note that we are weakening the usual definition of ``weakly
branched'': indeed it is usually additionally required that $G$ act
transitively on $X^n$ for all $n\in\N$. This extra property is not
necessary for our purposes, so we elect to simply ignore it. In fact,
all the results concerning branched groups that we shall use will be
proven directly from \cref{def:self-sim}.

Note also that the join $\langle K_1 \cup K_2 \rangle$ of two
branching subgroups $K_1$ and $K_2$ is again a branching
subgroup. Hence, there exists a maximal branching subgroup.  It
immediately follows from the definition that, if $G$ is weakly
branched, then for every $v\in X^*$ there is in $G$ a copy of its
branching subgroup $K$ whose action is concentrated on the subtree
$v X^*$. We denote this copy with $v * K$. With $v * k$ ($k \in K$)
we denote the element of $K$ acting as $k$
on the subtree $v X^*$ and trivially elsewhere.

Our main focus is on finitely generated groups. We first note
that the group $W$ itself is weakly branched. 
Here are countable weakly branched subgroups of $W$: For a subgroup $\Pi$ of
$\Sym X$, define $\Pi_\infty\le W$ as follows: set $\Pi_0=1\le W$ (the
trivial subgroup) and $\Pi_{n+1}=\phi^{-1}(\Pi_n\wr \Pi)$. We clearly have
$\Pi_n\le \Pi_{n+1}$, and we set $\Pi_\infty=\bigcup_{n\ge0}\Pi_n$. In words,
$\Pi_n$ consists of permutations of $X^*$ that may only modify the first
$n$ symbols of strings, and $\Pi_\infty$ consists of permutations that
may only modify a bounded-length prefix of strings. Clearly $\Pi_\infty$
is countable and $\phi(\Pi_\infty)=\Pi_\infty\wr \Pi$. 

Numerous properties are known to follow from the fact that a group is
weakly branched. For example, it satisfies no group
identity~\cite{Abert05}. In fact, if $G$ is a weakly branched
self-similar group and its branching subgroup $K$ contains an element of
order $p$, then it $K$ contains a copy of $(\Z/p)_\infty$,
see~\cite[Theorem~6.9]{BartholdiGrigorchukSunic03}.

There exist important examples of finitely generated self-similar
weakly branched groups, notably the {\em Grigorchuk group} $G$,
see~\cite{Grigorchuk80}. It may be described as a self-similar group
in the following manner: it is a group generated by $\oneset{a,b,c,d}$, and
acts on the rooted tree $X^*$ for $X=\oneset{0,1}$. The action, and
therefore the whole group, are defined by the restriction of $\phi$ to
$G$'s generators:
\[\phi(a)=(0,1),\quad\phi(b)=\pair<a,c>,\quad\phi(c)=\pair<a,d>,\quad\phi(d)=\pair<1,b>,\]
where we use the notation $(0,1)$ for the non-trivial element of
$\Sym X$ (that permutes $0$ and $1$) and $\pair<w_0,w_1>$ for a tuple in $G^{\{0,1\}} \cong G \times G$. We record some
classical facts:
\begin{lemma}
  The Grigorchuk group $G$ is infinite, torsion, weakly branched, and all its finite
  subquotients are $2$-groups (so in particular nilpotent). It has a branching subgroup $K$ of finite index, which
  is therefore finitely generated.
\end{lemma}
(Recall that every weakly branched group is infinite and
non-solvable, since it satisfies no identity. There are also easy
direct proofs of these facts.)
\begin{proof}
  That $G$ is an infinite torsion group is one of the \emph{raison d'\^etre} of $G$, see~\cite{Grigorchuk80}.
  Let $K \leq G$ be the normal closure of $[b,a]$ in $G$.
  It is easy to see that it has index $16$, and $\phi(\bigr[[b,a],d\bigl])=\pair<1,[b,a]>$ so
  $\phi(K)\geq K\times K$ and $G$ is weakly branched; see also \cite{BartholdiGrigorchukSunic03} for details.
  It is known that every element of $G$ has order
  a power of $2$~\cite{Grigorchuk80}, so the same holds for every subquotient of $G$.
\end{proof}
Other examples of finitely generated self-similar
weakly branched groups with a f.g.~branching subgroup include the Gupta-Sidki groups~\cite{GuptaSidki83}, the
Hanoi tower groups~\cite{GrigorchukSunik06}, and all iterated
monodromy groups of degree-$2$ complex 
polynomials~\cite{BartholdiNekrashevych08} except $z^2$ and $z^2-2$.

\subsection{Contracting self-similar groups}

Recall the notation $g@x$ for the coordinates of $\phi(g)$. We
iteratively define $g@v=g@x_1\cdots@x_n$ for any word
$v=x_1\cdots x_n\in X^*$.
\begin{definition}[{\cite[Definition~2.11.1]{Nekrashevych05}}]
  A self-similar group $G$ is called \emph{contracting} if there is a
  finite subset $N\subseteq G$ such that, for all $g\in G$, we have
  $g@v\in N$ whenever $v$ is long enough (depending on $g$).
\end{definition}
If $G$ is a finitely generated contracting group with word norm $\|\cdot\|$ (i.e., for $g \in G$, $\|g\|$ is the length of a shortest
word over a fixed generating set of $G$ that represents $g$), then a more
quantitative property holds: there are constants $0 < \lambda<1$, $h \geq 1$ and $k \geq 0$
such that for all $g\in G$ we have
\[\|g@v\|\le\lambda\|g\|+k\text{ for all }v\in X^h,
\]
see e.g. \cite[Proposition~9.3.11]{HRR2017}.
Then, for $c=-h/\log\lambda$ and a possibly larger $k$ we  have
$g@v\in N$ whenever $\abs v\ge c\log\|g\|+k$.
One of the cornerstones of Nekrashevych's theory of iterated
monodromy groups is the construction of a contracting self-similar
group that encodes a given expanding self-covering of a compact metric space.
 It is well-known and easy to check that the Grigorchuk group, the Gupta-Sidki groups and the Hanoi tower group for three pegs
are contracting.
The following result has been quoted numerous times, but has never
appeared in print. A proof for the Grigorchuk group may be found in~\cite{GaZa91}:
\begin{proposition}\label{prop:contractingL}
  Let $G$ be a finitely generated contracting self-similar group. Then $\WP(G)$ can be solved in \L (deterministic logarithmic space).
\end{proposition}
\begin{proof}
  Fix a finite generating set $\Sigma$ for $G$ and assume that $G$ is contracting with $0 <\lambda<1$, $h \geq 1$ and $k \geq 0$
  as above. We can assume that $k \geq 1$. Let $N$ be the nucleus of $G$.
   By replacing the tree alphabet $X$ by $X^h$ we 
  get $\|g@x\|\le\lambda\|g\|+k$ for all $x\in X$. Hence, if $\|g\| \leq k/(1-\lambda)$ then also
  $\|g@x\|\le k/(1-\lambda)$ for all $x\in X$. We now replace $\Sigma$ by the set of all $g \in G$ with 
  $\|g\| \leq k/(1-\lambda)$ (note that $k/(1-\lambda) \geq 1$) and get  $\phi(\Sigma)\subseteq \Sigma^X\times\Sym X$.
  Furthermore, there exists $m$ such that every non-trivial element of
  $N$ acts non-trivially on $X^m$. Recall that for $c=-1/\log\lambda$ and a possibly larger $k$ we have
  $g@v\in N$ whenever $\abs v\ge c\log\|g\|+k$. Hence, if $g$ is non-trivial then there must exist a $v \in X^*$ with
  $\abs v = c\log\|g\|+k+m$ such that $g$ does not fix $v$.
  
  The following algorithm solves $\WP(G)$: given $g\in\Sigma^*$,
  enumerate all vertices in $X^{d}$ for $d=c\log|g|+k+m$,
  and return ``true'' precisely when they are all fixed by $g$. The
  algorithm is correct by the previous remarks, and it remains to show
  that it requires logarithmic space.
  The vertices in $X^d$ are traversed by lexicographically enumerating them.
  They can be stored explicitly since their length is bounded by 
  $\Oh(\log \abs g)$.
  Now given a vertex $v \in X^d$, we apply the letters of $g$ to it one after
  the other. Again, this is done by a simple loop requiring
  $\Oh(\log\abs g)$ bits.
  Finally, to apply a generator to $v$, we use the property that all
  its states are generators ($\phi(\Sigma)\subseteq \Sigma^X\times\Sym X$), and traverse $v$ by performing $\abs v$
  lookups in the table storing $(\phi(a))_{a \in \Sigma}$.
\end{proof}

\section{Complexity theory}

We assume that the reader is familiar with the complexity classes \L (deterministic logarithmic space), \P (deterministic polynomial time), and \PSPACE (polynomial space); see e.g.\ \cite{AroBar09} for details. 
With $\polyL$ we denote that union of all classes $\NSPACE(\log^c n)$ for a constant $c$.
Since we also deal with sublinear time complexity classes, we use Turing machines with {\em random access} (this has no influence on the definition of the above classes).
Such a machine has an additional index tape and some special query states. Whenever the Turing machine enters a query state, the following transition depends on the input symbol at the position which is currently written on the index tape in binary notation. 

 We use the abbreviations DTM (deterministic Turing machine), NTM (non-deterministic Turing machine) and ATM (alternating Turing machine).
 An ATM is an NTM together with a partition of the state set into existential and universal states. A configuration is called existential (resp., universal) if the current state
in the configuration is existential (resp., universal). 
An existential configuration is accepting if there exists an
accepting successor configuration, whereas a universal configuration is accepting if all successor 
configurations are accepting. Note that a universal configuration which does not have a successor configuration is accepting,
whereas an existential configuration which does not have a successor configuration is non-accepting.
Finally, an input word is accepted if the corresponding initial configuration is accepted.
An ATM is in \emph{input normal form} if its input alphabet is $\{0,1\}$ and
on any computation path it queries at most one input bit and halts immediately after returning the value of the input bit or its negation (depending on the current state of the Turing machine).
We define the following complexity classes:
\begin{itemize}
\item $\DLINTIME$: the class of languages that can be accepted 
by a DTM in linear time.
\item $\DLOGTIME$: the class of languages that can be accepted 
by a DTM in logarithmic time.
\item $\ALOGTIME$: the class of languages that can be accepted 
by an ATM in logarithmic time. 
\item $\APTIME$: the class of languages that can be accepted 
by an ATM in polynomial time. 
\end{itemize}
If $X$ is one of the above classes, we speak of an $X$-machine with the obvious meaning.
It is well known that $\APTIME = \PSPACE$. Moreover, every language in \ALOGTIME can be recognized by an 
\ALOGTIME-machine in input normal form  \cite[Lemma 2.41]{Vollmer99}.

A $\mathsf{nand}$-machine 
is an NTM in which each configuration has either zero or two successor configurations
and configurations are declared to be accepting, respectively non-accepting,
according to the following rules, where $c$ is a configuration:
\begin{itemize}
\item If $c$ has no successor configurations and the state of $c$ is final (resp., non-final), then $c$ is accepting (resp., non-accepting).
\item If $c$ has two successor configurations and both of them are accepting, then $c$ is not accepting.
\item If $c$ has two successor configurations and at least one them is non-accepting, then $c$ is accepting.
\end{itemize}
Since the boolean functions {\sf and} and {\sf or} can be obtained with {\sf nand}, it follows easily that $\PSPACE$ (resp., \ALOGTIME) coincides with the class
of all languages that can be accepted by a polynomially (resp., logarithmically) time-bounded $\mathsf{nand}$-machine.

For a complexity class $\mathsf{C}$ we denote by $\forall \mathsf{C}$ 
the class of all languages $L$ such that there exists a polynomial $p(n)$
and a language $K \in \mathsf{C}$ such that 
$L = \{ u \mid \forall v \in \{0,1\}^{p(|u|)} : u \# v \in K \}$. We have for instance
$\forall \Ptime = \coNP$ and $\forall \PSPACE = \PSPACE$.
Likewise we define the class $\Mod{m} \mathsf{C}$ by $L \in \Mod{m} \mathsf{C}$ if there exists a polynomial $p(n)$
and a language $K \in \mathsf{C}$ such that 
$L = \big\{ u \mid |\{ v \in \{0,1\}^{p(|u|)} : u \# v \in K\}| \not\equiv 0 \mod m \big\}$.

\subsection{Efficiently computable functions}

A function $f\colon\Gamma^* \to \Sigma^*$ is \DLOGTIME-computable if there is some polynomial $p$ with $\abs{f(x)} \leq p(\abs{x})$ for all $x \in \Gamma^*$ and the set 
$L_f = \{ (x,a,i) \mid x \in \Gamma^* \text{ and the $i$-th letter of $f(x)$ is $a$}\}$
 belongs to $\DLOGTIME$. Here $i$ is a binary coded integer. 
Note that a \DLOGTIME-machine for $L_f$ can first (using binary search) compute the binary coding of $\abs x$ in time $\Oh(\log |x|)$.
Assume that the length of this binary coding is $\ell$.
If $i$ has more than $\ell$ bits, the machine can reject immediately. 
As a consequence of this (and since $\abs{\Sigma}$ is a constant), the running time of a
\DLOGTIME-machine for $L_f$ on input $(x,a,i)$ can be bounded by $\Oh(\log |x|)$ (independently of the actual bit length of $i$).
 We can also assume that the \DLOGTIME-machine outputs the letter $a$ on input of $x$ and $i$. In case
 $i > |x|$ we can assume that the machine outputs a distinguished letter.
A \DLOGTIME-reduction is a \DLOGTIME-computable many-one reduction.
We say that a \DLOGTIME-machine \emph{strongly} computes a function $f\colon\Sigma^* \to \Gamma^*$ with $\abs{f(x)} \leq C \log(\abs{x})$ for all $x \in \Sigma^*$ and for some constant $C$ if it computes the function value by writing it sequentially on a separate output tape (be aware of the subtle difference and that strong \DLOGTIME-computability is not a standard terminology, but is coincides with $\mathsf{FDLOGTIME}$ in \cite{CaussinusMTV98}.)

A $\PSPACE$-transducer is a deterministic Turing-machine with a 
read-only input tape, a write-only output tape 
and a work tape, whose length is polynomially bounded 
in the input length $n$. The output is written sequentially on the output tape.
Moreover, we assume that the transducer terminates for every
input. This implies that a $\PSPACE$-transducer computes a 
mapping $f : \Sigma^* \to \Gamma^*$, where $|f(x)|$ is bounded
by $2^{|x|^{\Oh(1)}}$. We call this mapping \PSPACE-computable.
We need the following simple lemma, see \cite{LohreyM13}:

\begin{lemma} \label{lemmaPSPACE}
Assume that the mapping $f : \Sigma^* \to \Gamma^*$ is \PSPACE-computable and let
$L \subseteq \Gamma^*$ be a language in $\polyL$. Then $f^{-1}(L)$ belongs to $\PSPACE$.
\end{lemma}

\subsection{Leaf languages}

In the following, we introduce basic concepts
related to leaf languages, more details can be found in  \cite{BoCrSi,Her97,HLSVW93,HVW96,JeMcTh96}.
An NTM $M$ with input alphabet $\Gamma$ is \emph{adequate}, if 
(i) for every input $x \in \Gamma^*$, $M$ does not have an infinite 
computation on input $x$, (ii) the finite set of transition tuples
of $M$ is linearly ordered, and (iii) when terminating $M$ prints a symbol $\alpha(q)$ from a finite alphabet $\Sigma$, where $q$ is 
 the current state of $M$.
For an input $x \in \Gamma^*$, we define the computation tree
by unfolding the configuration
graph of $M$ from the initial configuration. 
By condition (i) and (ii), the computation
tree can be identified with a  
finite ordered tree $T(x) \subseteq \mathbb{N}^*$.
For $u \in T(x)$ let $q(u)$ be the $M$-state of 
the configuration that is associated with the 
tree node $u$.
Then, the leaf string $\leaf(M,x)$
is the string $\alpha(q(v_1)) \cdots \alpha(q(v_k)) \in \Sigma^+$, where 
$v_1,\ldots,v_k$ are all leaves of $T(x)$ listed in lexicographic order.

An adequate NTM $M$ is called  \emph{balanced}, if for every input $x \in \Gamma^*$,
$T(x)$ is a complete binary tree.
With a language $K \subseteq\Sigma^*$ we associate 
the language 
$$
\LEAF(M,K) = \{ x \in  \Gamma^* \mid \leaf(M,x) \in K \}.
$$
Finally, we associate two complexity classes with $K \subseteq\Sigma^*$:
\begin{eqnarray*}
\LEAF(K) &=& \{ \LEAF(M,K) \mid \text{$M$ is an adequate polynomial time
    NTM}\}\\
\bLEAF(K) &=& \{ \LEAF(M,K) \mid \text{$M$ is a balanced polynomial time NTM}\}
\end{eqnarray*}
These classes are closed under polynomial time reductions.
We clearly have $\bLEAF(K)  \subseteq \LEAF(K)$.
The following result was shown in \cite{JeMcTh96} by padding computation trees to complete binary trees.

\begin{lemma}
Assume that $K \subseteq \Sigma^*$ is a language such that $\Sigma$ contains a symbol $1$ with the following property:
if $uv \in K$ for $u,v \in \Sigma^*$ then $u1v \in K$. Then $\LEAF(K)  = \bLEAF(K)$. 
\end{lemma}
In particular, we obtain the following lemma:

\begin{lemma} \label{lemma-JMT}
Let $G$ be a finitely generated group and $\Sigma$ a finite standard generating set for $G$. 
Then $\LEAF(\WP(G,\Sigma))  = \bLEAF(\WP(G,\Sigma))$. 
\end{lemma}

Moreover, we have:
\begin{lemma} \label{lemma-f.g.set-not-important}
Let $G$ be finitely generated group and $\Sigma$, $\Gamma$ finite standard generating sets for $G$. 
Then $\LEAF(\WP(G,\Sigma))  = \LEAF(\WP(G,\Gamma))$. 
\end{lemma}

\begin{proof}
Consider a language $L \in \LEAF(\WP(G,\Sigma))$. Thus, there exists an adequate polynomial time
 NTM $M$ such that $L = \LEAF(M,\WP(G,\Sigma))$. We modify $M$ as follows: If $M$ terminates and prints
 the symbol $a \in \Sigma$, it enters a small nondeterministic subcomputation that produces the leaf string $w_a$,
 where $w_a \in \Gamma^*$ is a word that evaluates to the same group element as $a$. Let $M'$ be the resulting
 adequate polynomial time NTM.
 It follows that $\LEAF(M,\WP(G,\Sigma)) =  \LEAF(M',\WP(G,\Gamma))$.
\end{proof}
Lemma~\ref{lemma-f.g.set-not-important} allows to omit the standard generating set $\Sigma$ in the notations
$\LEAF(\WP(G,\Sigma))$ and $\bLEAF(\WP(G,\Sigma))$. We will always do that.
In \cite{HLSVW93} it was shown that $\PSPACE = \LEAF(\WP(G))$ for every finite non-solvable group.

\subsection{Circuit complexity}

We define a \emph{polynomial length projection} (or just {\em projection}) as a function $f\colon \oneset{0,1}^* \to \oneset{0,1}^*$ such 
that there is a function $d(n) \in \Oh(\log n)$ with $\abs{f(x)} = \abs{f(y)} = 2^{d(n)}$ for all $x, y$ with $\abs{x} = \abs{y} = n$ and such that each output bit depends on at most one input bit in the following sense: 
For every $n \in \N$, there is a mapping $q_n \colon \{0,1\}^{d(n)} \to \set{\langle j,a,b\rangle}{j \in \ccinterval{1}{n}, a,b,\in \oneset{0,1}}$, where $q_n(i) = \langle j,a,b\rangle$ means that 
for all $x \in \{0,1\}^n$  the $i$-th bit of $f(x)$ is $a$ if the $j$-th bit of $x$ is $1$ and $b$ if it is $0$. Here, we identify $i \in  \{0,1\}^{d(n)}$ with a binary coded number from $\cointerval{0}{2^{d(n)}}$
(so the first position in the output is zero).
We also assume that the input position $j \in \ccinterval{1}{n}$ is coded in binary, i.e., by a bit string of length $\Oh(\log n)$.
Note that the output length $2^{d(n)}$ is polynomial in $n$. Restricting the output length to a power of two (instead of an arbitrary polynomial) is convenient for our purpose but in no way crucial.
Our definition of a projection is the same as in \cite{CaussinusMTV98} except for our restriction on the output length. Moreover,
in  \cite{CaussinusMTV98} projections were defined for arbitrary alphabets.

Let $q \colon \{1\}^* \times \oneset{0,1}^* \to \{0,1\}^* \times \{0,1\} \times \{0,1\}$ with $q(1^n,v) = q_n(v)$. 
We assume that $q(1^n,v)$ is a special dummy symbol if $|v| \neq d(n)$.
We call $q$ the {\em query mapping} associated with the projection $f$.
The projection $f$ is called \emph{uniform} if (i) $1^{d(n)}$  is strongly computable in \DLOGTIME from the string $1^n$, and (ii) $q$ is strongly \DLOGTIME-computable. 
Notice that if a language $K$ is reducible to $L$ via a uniform projection, then $K$ is also \DLOGTIME-reducible to $L$.

We are mainly interested in the circuit complexity class \Nc1.
A language $L \subseteq \oneset{0,1}^*$  is in \Nc1 if it can be recognized by a family of logarithmic depth boolean circuits of bounded fan-in. More precisely, $L \subseteq \oneset{0,1}^*$ belongs to  \Nc1 if there exists a family $(C_n)_{n \geq 0}$ of boolean circuits which, apart from the input gates $x_1, \ldots, x_n$, are built up from {\sf not}-, {\sf and}- and {\sf or}-gates. In the following we also use \nand-gates.
All gates must have bounded fan-in, where the fan-in of a gate is the number of incoming edges of the gate. Without loss of generality, we assume that all {\sf and}-, {\sf or}- and {\sf nand}-gates have 
fan-in two. The circuit $C_n$ must accept exactly the words from $L \cap \oneset{0,1}^n$, i.e., if each input gate $x_i$ receives the input $a_i \in \oneset{0,1}$, then a
distinguished output gate evaluates to $1$ if and only if $a_1 a_2 \cdots a_n \in L$. Finally, the depth (maximal length of a path from an input to the distinguished output) of $C_n$ must grow logarithmically in $n$. In the following, we also consider \DLOGTIME-uniform \Nc1, which is well-known to coincide with \ALOGTIME (see \eg \cite[Corollary 2.52]{Vollmer99}).
\DLOGTIME-uniform means that there is a \DLOGTIME-machine which decides on input of two gate numbers $i$ and $j$ in $C_n$ (given in binary), a binary string $w$, and the string $1^n$ whether, when starting at gate $i$ in $C_n$ and following the path labelled by $w$, we reach gate $j$. Here, following the path labelled by $w$ means that we go to the left (right) input of $i$ if $w$ starts with a $0$ ($1$) 
and so on. Moreover, we require that on input of $i$ in binary and the string $1^n$, the type of the gate $i$ in $C_n$ is computable in \DLOGTIME.
For more details on these definitions we refer to \cite{Vollmer99} (but we will not need the above definition of \DLOGTIME-uniformity).
For a language $L$  over a non-binary alphabet $\Sigma$, one first has to fix a binary 
encoding of the symbols in $\Sigma$. For  membership in \Nc1 the concrete encoding is irrelevant. However, we still assume that all letters of $\Sigma$ are encoded using the same number of bits.

The class \Ac{0} is defined as the class of languages (respectively functions) accepted (respectively computed) by circuits of constant depth and polynomial size with {\sf not}-gates and
unbounded fan-in {\sf and}- and {\sf or}-gates.

We will also work with a very restricted class of circuit families, where every circuit is a complete binary tree of {\sf nand}-gates.
For such a circuit, all the information is given by the labelling function for the input gates.

 \begin{definition} \label{nand-tree-circuit}
	A  \emph{family of balanced {\sf nand}-tree-circuits of logarithmic depth} $(C_n)_{n\in \N}$ is given by a mapping $d(n) \in \Oh(\log n)$ and a query mapping 
	$q \colon \{1\}^* \times \oneset{0,1}^* \to \{0,1\}^* \times \{0,1\} \times \{0,1\}$, which defines a projection $f$ mapping bit strings of length $n$ to 
	bit strings of length $2^{d(n)}$. The corresponding circuit $C_n$ for input length $n$ is then obtained by taking $\{0,1\}^{\le d(n)}$ as the set of gates.
	Every gate $v \in \{0,1\}^{< d(n)}$ computes the \nand\ of $v0$ and $v1$.
 	If $x \in \{0,1\}^n$ is the input string for $C_n$ and $f(x) = a_1 a_2 \cdots a_{2^{d(n)}}$,
	 then the $i$-th leaf $v \in \{0,1\}^{d(n)}$ (in lexicographic order) is set to $a_i$.
\end{definition}

\begin{lemma}\label{lem:layeredNAND}
 	For every $L$ in (non-uniform) \Nc1 there is a (non-uniform) family of balanced {\sf nand}-tree-circuits of logarithmic depth. 
 \end{lemma}

 \begin{proof}
 The proof is  straightforward: clearly, {\sf or}, {\sf and}, and {\sf not} gates can be simulated by {\sf nand} gates. Now take the circuit $C_n$ for input length $n$. We first unfold $C_n$ into a tree by duplicating gates with multiple outputs.
 Since $C_n$ has constant fan-in and logarithmic depth, the resulting tree has still polynomial size (and logarithmic depth).
 To transform this tree into a complete binary tree, we replace leafs by complete binary subtrees. If we replace a leaf labelled with $x_i$ by a subtree
 of even (resp. odd) height, then we label all leafs of the subtree with $\langle i, 1, 0 \rangle$ (resp., $\langle i, 0, 1 \rangle$). This labelling defines the query mapping $q$ in the natural way.
 \end{proof}

 \begin{lemma}\label{lem:layeredNANDuniform}
 	For every $L$ in \ALOGTIME  there is a 
 	family $\mathcal{C} = (C_n)_{n \ge 0}$ of balanced {\sf nand}-tree-circuits of logarithmic depth such that the mapping $1^n \mapsto 1^{d(n)}$ 
	and the query mapping $q$ from \cref{nand-tree-circuit} can be strongly computed in \DLOGTIME.
 \end{lemma}

\begin{proof}[Proof sketch]
We start with an \ALOGTIME-machine $M$ for $L$ and construct a circuit family with the required properties. 
We can assume that $M$ works in two stages: first it computes the binary coding of the input length in \DLOGTIME (using binary search). The second stage performs the actual computation. We can assume that the second stage is in input normal form \cite[Lemma 2.41]{Vollmer99} meaning that each computation path queries exactly one input position $i$ and halts immediately after querying that position
(returning a bit that is determined by the $i$-th bit of the input). Furthermore, we can assume that the computation tree of the second stage of $M$ is a complete binary tree. For this we enforce all computation paths to be of the same length. Note that
the running time of the second stage of $M$ can be bounded by $c \cdot |u|$, where $c$ is a fixed constant and $u$ is the binary coding of the input length which has been computed before.
Hence, the second stage of the machine makes in parallel to the actual computation $c$ runs over $u$. 
Finally, we also assume that there is an alternation in every step (this can be ensured as in the transformation of an arbitrary \Nc1-circuit into a balanced {\sf nand}-tree-circuit) and that the initial state is existential. The computation tree gives a tree-shaped circuit in a natural way (for details see \cite[Theorem 2.48]{Vollmer99}).  
The depth of this tree is $d \coloneqq c \cdot |u|$ (whose unary encoding is strongly computable in \DLOGTIME by the above arguments). 
Since we start with an existential state and there is an alternation in every step, the resulting circuit uses only {\sf nand}-gates 
(recall that $x$ {\sf nand} $y$ = ({\sf not} $x$) {\sf or} ({\sf not} $y$)).  The fact that every computation path queries only one input position yields the query function $q$ from Definition~\ref{nand-tree-circuit}.
More precisely, let $v \in \{0,1\}^d$ be an input gate of the balanced {\sf nand}-tree-circuit. 
Then $v$ determines a unique computation path of $M$. We simulate
$M$ in \DLOGTIME along this path and output the triple $\langle i, a,b\rangle$ if $M$ queries the $i$-th position of the input string (note that the binary coding of $i$
must be on the query tape of $M$) and outputs $a$ (resp., $b$) if the $i$-th input bit is $1$ (resp., $0$). 
\end{proof}

\subsubsection{\boldmath $G$-programs}

For infinite groups we have to adapt Barrington's notion of a $G$-program slightly. Our notation follows \cite{Vollmer99}.
\begin{definition}
	Let $G$ be a group with the finite standard generating set $\Sigma$. Recall our assumption that $1 \in \Sigma$. A $(G,\Sigma)$-program $P$ of length $m$ and input length $n$
	is a sequence of \emph{instructions} $\langle i_j, b_j, c_j \rangle$ for $0 \leq j \leq m-1$ where $i_j \in \ccinterval{1}{n}$ and $b_j,c_j \in \Sigma$.
	On input of a word $x = a_1 \cdots a_n\in\oneset{0,1}^*$, an instruction $\langle i_j, b_j, c_j \rangle$ evaluates to $b_j$ if $a_{i_j} = 1$ and to $c_j$ otherwise. The evaluation of a $(G,\Sigma)$-program is the product (in the specified order) of the evaluations of its instructions, and is denoted with $P[x] \in \Sigma^*$.

	A family $\cP = (P_n)_{n\in \N}$ of $(G,\Sigma)$-programs, where $P_n$ has input length $n$, defines a function $f_\cP\colon \oneset{0,1}^* \to G$: $f_\cP(x)$ is the group element
	represented by $P_{|x|}[x]$.
	 The language $L$ \emph{accepted} by the family of $(G,\Sigma)$-programs is the set of words $x \in \oneset{0,1}^*$ such that $f_\cP(x) = 1$ in $G$.
	For brevity, we also speak of a \emph{family of $G$-programs} instead of $(G,\Sigma)$-programs with the understanding that there is some finite standard generating set $\Sigma$ which is shared by all programs of the family.
\end{definition}
Notice two differences compared with the original definition: firstly, we fix the finite alphabet $\Sigma$, and secondly, for the accepted language we only take the preimage of $1$ instead of a finite set of final states.  The latter is more restrictive, but for the purpose of $\Nc1$-hardness causes no difference.

A family $\cP = (P_n)_{n\in \N}$ of $(G,\Sigma)$-programs is called \emph{uniform} if the length of $P_n$ is $2^{d(n)}$ for some function $d(n) \in \Oh(\log n)$, the 
mapping $1^n \mapsto 1^{d(n)}$  is strongly computable in \DLOGTIME, and the mapping that assigns to 
$1^n$ and $j \in \{0,1\}^{d(n)}$ (the latter is interpreted
as a binary coded number) the instruction $\langle i_j, b_j, c_j \rangle$ of the $n$-input program $P_n$ is 
strongly computable in \DLOGTIME. Notice that $i_j$ requires $\log n$ bits and $b_j$, $c_j$ require only 
a constant number of bits~--- thus, the tuple $\langle i_j, b_j, c_j \rangle$ can be written down in \DLOGTIME. 
Be aware that here we slightly differ from~\cite[Definition~4.42]{Vollmer99}.

\begin{remark}\label{rem:Gprogramprojection}
If a language $L$ is accepted by a family of polynomially length-bounded $(G,\Sigma)$-programs (by padding one can enforce the length to be of the form $2^{d(n)}$),
then $L$ is reducible via projections to $\WP(G)$~-- and, thus, also via \Ac0-many-one reductions.
This can be seen as follows: encode every letter in $\Sigma$ by a word over $\oneset{0,1}$ of some  fixed constant length.
Then the map assigning the evaluation of the $(G,\Sigma)$-program to an input word is a uniform projection since the output at every position depends on only one input bit.
A similar statement holds in the uniform case (uniformity follows immediately from the definition): if $L$ is accepted by a uniform family of $(G,\Sigma)$-programs, then $L$ is reducible via uniform projections to $\WP(G)$.
\end{remark}
 
\section{Efficiently non-solvable groups}\label{sec:SENS}

We now define the central group theoretic property that allows us to carry out a Barrington style construction:
\begin{definition}\label{def:SENS}		
  We call a group $G$ with the finite standard generating set $\Sigma$
  \emph{strongly efficiently non-solvable} \emph{(SENS)} if for every
  $d \in \N$ there is a collection of $2^{d+1}-1$ elements
  $g_{d,v} \in \Sigma^*$ for $v \in \oneset{0,1}^{\leq d}$ such that
  \begin{enumerate}[(a)]
  \item\label{SENSa} there is some constant $\mu \in \N$ with $\abs{g_{d,v}} = 2^{\mu d}$ for all $v \in \oneset{0,1}^{d}$,
  \item\label{SENSb} $g_{d,v} =	\bigl[ g_{d,v0},\,  g_{d,v1} \bigr]$ for all $v \in \oneset{0,1}^{< d}$ (here we take the commutator of words),
  \item\label{SENSc} $g_{d,\eps} \neq 1$ in $G$.
  \end{enumerate}
   The group $G$ is called \emph{uniformly strongly efficiently
    non-solvable} if, moreover,
  \begin{enumerate}[(a)]
    \setcounter{enumi}{3}
  \item\label{SENSu} given $v\in \oneset{0,1}^d$, a number $i$ encoded in binary with $\mu d$ bits, and $a \in \Sigma$ one can decide in $\DLINTIME$ whether the
  	$i$-th letter of $g_{d,v}$ is $a$.
  \end{enumerate} 
  If $Q=H/K$ is a subquotient of $G$, we call $Q$ \emph{SENS in $G$}
  if $G$ satisfies the conditions of a SENS group, all
  $g_{d,v}$ evaluate to elements of $H$, and $g_{d,\eps}\not\in K$. This definition is already interesting for $K=1$.
\end{definition}
Here are some simple observations:
\begin{itemize}
\item A strongly efficiently non-solvable group clearly cannot
be solvable, so the above terminology makes sense.
\item If one can find suitable $g_{d,v}$ of length at most $2^{\mu d}$,  then these words can always be padded to length $2^{\mu d}$ thanks to the padding letter $1$.
\item  It suffices to specify $g_{d,v}$ for $v \in \oneset{0,1}^{d}$;
  the other $g_{d,v}$ are then defined by Condition~(\ref{SENSb}).
 \item We have $\abs{g_{d,v}} = 2^{\mu d+ 2(d-\abs{v})}$ for all $v \in \oneset{0,1}^{\leq d}$. Thus, all $g_{d,v}$ have length $2^{\Oh(d)}$.
 \item Equivalently to Condition~(\ref{SENSu}), we can require that given $v\in \oneset{0,1}^d$ and a binary encoded number $i$ with $\mu d$ bits, one can compute the $i$-th letter of $g_{d,v}$ in $\DLINTIME$.
\end{itemize}
Henceforth, whenever $d$ is clear, we simply write $g_v$ instead of $g_{d,v}$.

We can formulate a weaker condition than being strongly efficiently
non-solvable which is sufficient for our purposes, but  slightly more
complicated to state:
\begin{definition}\label{def:ENS}
  We say $G$ is \emph{efficiently non-solvable} \emph{(ENS)} if there
  is an even constant $l$ such that for every $d \in \N$, there is a
  collection of elements
  $(g_{d,v})_{v\in \ccinterval{1}{l}^{\le
      d}}$ such that
  \begin{enumerate}[(a)]
  \item $\abs{g_{d,v}} \in 2^{\Oh(d)}$ when $\abs v=d$,
  \item $g_{d,v}=[ g_{d,v1},\,  g_{d,v2} ] \cdots [ g_{d,v (l-1)},\,  g_{d,v l} ]$ when $\abs v<d$,
  \item $g_{d,\eps} \neq 1$ in $G$.
  \end{enumerate}
  Analogously to \cref{def:SENS}, we
  define a group $G$ to be \emph{uniformly efficiently non-solvable}
  if the letters of $g_{d,v}$ for $\abs v =d$ can be computed in
  $\DLINTIME$, and a subquotient $Q=H/K$ of $G$ to be (uniformly) efficiently non-solvable in $G$ if the $g_{d,v}$ evaluate to elements of $H$ with $g_{d,\eps}\not\in K$.
\end{definition}
In the non-uniform setting, we can show that the ENS property is equivalent to the SENS property.

\begin{lemma}\label{lem:ENSimpliesSENS}
	If $G$ is ENS, then $G$ is SENS.
\end{lemma}
Notice that we are not aware whether an analogous statement for the uniform setting holds as well. 
\begin{proof}
Let $(g_{d,v})_{v\in \ccinterval{1}{l})^{\le
		d}}$ be as in \cref{def:ENS}, meaning that, in particular, $g_{d,\eps} \neq 1$ in $G$. We can think of $g_{d,\eps}$ as a word over the $g_{d,v}$ for $\abs{v} = d$.
	Using the identity $[x, zy] = [x,y] [x,z]^y$, we can rewrite $g_{d,\eps}$ as a product of balanced nested commutators (where no product appears inside any commutator). Since $[x,z]^y = [x^y,z^y]$ we can pull all conjugations inside the commutators.
	 Since the product is non-trivial, at least one of its factors is non-trivial. This factor is a balanced nested commutator of depth $d$, hence, witnessing that $G$ is SENS.
\end{proof}
We continue with some simple observations. In all cases, we only provide proofs for the SENS class. As observed
above, this is no restriction in the non-uniform setting. In the uniform setting one can use analogous arguments for the ENS class.

\begin{lemma}\label{lem:indepgens}
  The property of being (uniformly) (S)ENS is independent of the
  choice of the standard generating set.
\end{lemma}

\begin{proof}
  Let $\Sigma'$ be another standard generating set. Then, for some constant
  integer $k$, every element of $\Sigma$ may be written (thanks to the
  padding letter $1$)  as a word of length $2^k$ in $\Sigma'$. In particular, if
  $g_{d,v}$ has length $2^{\mu d}$ with respect to $\Sigma$, then it has length $2^{k + \mu d}$ with respect to $\Sigma'$. 
  There is also a simple \DLINTIME-algorithm for computing the $i$-th letter of $g_{d,v}\in(\Sigma')^*$: 
   given $v$, and $i$, it
  runs the \DLINTIME-algorithm for $\Sigma$ on input $v$ and $\floor{i/ 2^k}$, obtaining a
  letter $\sigma\in\Sigma$. Then, it looks up the length-$2^k$
  representation of $\sigma$ over $\Sigma'$, and extracts the
  $(i\bmod 2^k)$-th letter of that representation.
\end{proof}
Later (Example~\ref{ex:nonENS2}) we will give an example of a f.g.~non-ENS group $H$ which is 
 uniformly SENS in a group $G$. 

\begin{lemma}\label{lem:subquotient}
  If $Q = H/K$ is a finitely generated subquotient of a finitely generated group $G$ and $Q$ is
  (uniformly) (S)ENS, then $G$ is also (uniformly) (S)ENS. 
\end{lemma}

\begin{proof}
  Let $\Gamma$ be a standard  generating set of $Q$ and fix for every $a \in \Gamma$ an element $h_a \in H \leq G$ 
  such that $h_a$ is mapped to $a$ under the canonical projection $\pi : H \to Q = H/K$.
  By \cref{lem:indepgens} we can assume that all elements $h_a$ belong to the generating set of $G$.
  Let $h_{d,v} \in \Gamma^*$ be the words witnessing the fact that $Q$ is (uniformly) (S)ENS (in \cref{def:SENS} they are denoted with $g_{d,v}$). 
  We then define words $g_{d,v}$ by replacing every letter $a$ in $h_{d,v}$ by the letter $h_a$.
  Clearly,  $\pi(g_{d,v}) = h_{d,v}$  holds. In particular, $g_{d,\eps}$ is non-trivial, since $h_{d,\eps}$ 
  is non-trivial.
\end{proof}

\begin{lemma}\label{lem:commutatorSENS}
  If $G$ is (uniformly) (S)ENS, then the commutator subgroup $G'$ is
  (uniformly) (S)ENS in $G$. 
\end{lemma}

\begin{proof}
  Given $d\in\N$, produce the words $g_{d+1,v}$ 
  with $\abs v\le d+1$ witnessing the property for $G$, and consider the same
  words with $\abs v\le d$. They witness the same property for $G'$. In
  effect, we are truncating the leaves of a tree of commutators in
  $G$.
\end{proof}
The following is a stronger version of \cref{lem:commutatorSENS}:
\begin{lemma}\label{lem:SENStoSubgroup}
	If $G$ is (uniformly) (S)ENS and $N$ a normal subgroup such that $G/N$ is solvable, then  $N$ is
	(uniformly) (S)ENS in $G$. 
\end{lemma}

\begin{proof}
        Assume that $G/N$ is solvable of derived length $\delta$. 
        Hence, any $\delta$-fold nested commutator of elements in $G$ is contained in $N$.
        We only prove the theorem for the case that $G$ is (uniformly) SENS; the same argument
        applies if $G$ is (uniformly) ENS.
        Let $h_{d,v}$ be the elements witnessing that $G$ is (uniformly) SENS.
        Given $d$ and $v \in \oneset{0,1}^{\leq d}$ define 
	$g_{d,v} = h_{d + \delta, v}$. Then all these elements are $\delta$-fold nested commutators and, hence, contained in $N$.
	Thus, the elements $g_{d,v}$ witness that $N$ is (uniformly) SENS in $G$.	
\end{proof}

\begin{lemma}\label{lem:SENStoQuotient}
	If $G$ is (S)ENS and $N$ a solvable normal subgroup of $G$, then $G/N$ is
	(S)ENS. 
\end{lemma} 
Be aware that we do not know whether there is a variant of \cref{lem:SENStoQuotient} for uniformly (S)ENS. The problem is to compute the word $u$ in the proof below.

\begin{proof}
    Again, we only prove the statement for the case that $G$ is SENS.
	As in the proof of \cref{lem:SENStoSubgroup}, let $h_{d,v}$ for $d \in \N$ and $v \in \oneset{0,1}^{\leq d}$ denote the elements witnessing that $G$ is SENS. Let $\delta $ denote the derived length of $N$. 
	Assume for contradiction that all the elements $h_{d + \delta, v}$ for $v \in \oneset{0,1}^\delta$ are in $N$. 
	Then, $h_{d + \delta, \eps}$ would be trivial because it is a $\delta$-fold nested commutator of the $h_{d + \delta, v}$ for $v \in \oneset{0,1}^\delta$ and the derived length of $N$ is $\delta$. 
	Thus, there exists some $u \in \oneset{0,1}^\delta$ such that $h_{d + \delta, u} \not \in N$. 
	We fix this $u$ and  set $g_{d,v} = h_{d + \delta, uv}$ for $v \in \oneset{0,1}^{\leq d}$. Since  $g_{d,\eps} = h_{d + \delta,  u} \not \in N$, this shows that $G/N$ is SENS.	
\end{proof}

\begin{lemma}\label{lem:centerSENS}
	If $G$ is (uniformly) (S)ENS, then $G/Z(G)$ is
	(uniformly) (S)ENS. 
\end{lemma}
\begin{proof}
	As before, let $h_{d,v}$ for $d \in \N$ and $v \in \oneset{0,1}^{\leq d}$ denote the elements witnessing that $G$ is (uniformly) SENS. 
	We set $g_{d,v} = h_{d + 1, 0v}$ for $v \in \oneset{0,1}^{\leq d}$. Then $g_{d,\eps} = h_{d + 1, 0}$ cannot be in $Z(G)$ for otherwise $h_{d+1, \eps} = [g_{d,\eps}, h_{d + 1, 1}]$ would be trivial. This shows that $G/Z(G)$ is (uniformly) SENS.
\end{proof}
The following result is, for $G=A_5$, the heart of Barrington's
argument:
\begin{lemma}\label{lem:finiteSENS}
  If $G$ is a finite non-solvable group, then $G$ is uniformly SENS.
\end{lemma}
\begin{proof}
Let us first show the statement for a non-abelian finite simple group $G$.
  By the proof of Ore's conjecture~\cite{LiebeckOST10}, every element
  of $G$ is a commutator. This means that we may choose
  $g_\eps\neq1$ at will, and given $g_v$ we define $g_{v0},g_{v1}$
  by table lookup, having chosen once and for all for each element of
  $G$ a representation of it as a commutator. Computing $g_v$ requires
  $\abs v$ steps and bounded memory.
  
  If $G$ is finite non-solvable, then any composition series of $G$ contains
  a non-abelian simple composition factor $G_i/G_{i+1}$. Hence, we can apply
  \cref{lem:subquotient}.
\end{proof}
Notice that at the time of Barrington's original
proof~\cite{Barrington89}, Ore's conjecture was not known to
hold. This explains that he used only what we defined as
ENS in order to establish his result on
$\Nc{1}$-hardness. Nevertheless, a direct proof of \cref{lem:finiteSENS} is also possible: 

First of all, it is clear that every non-abelian finite simple $G$ group is ENS. Thus, \cref{lem:ENSimpliesSENS} tells us that $G$ is SENS. Let $g_{d,v}$ be the words from \cref{def:SENS}. It remains to show Condition~\eqref{SENSu}
from \cref{def:SENS}. 
Since $G$ is finite, we can find a subset $S$ of $G$ such that for each $g \in S$ there are $h_1, h_2 \in S$ with $g = [h_1,h_2]$. In order to find such $S$, take $d_0 = \abs{G}+1$. Then on any path from the root $\eps$ to a leaf $v \in \oneset{0,1}^{d_0}$ in the complete binary tree of depth $d_0$, there must be vertices $u_v,t_v$ with $\abs{u_v}< \abs{t_v}$ (\ie $u_v$ is a prefix of $t_v$ and the latter is a prefix of $v$) and $g_{d_0,u_v} = g_{d_0,t_v}$ in $G$. Let $S$ be just the union over all $g_{d_0,u}$ with $u$ being a prefix of some $t_v$.
Using this $S$ (which can be hard-wired in our algorithm for computing the $g_{d,v}$), the uniformity condition can be seen as follows: for each $g \in S$ fix $h_1, h_2 \in S$ with $g = [h_1,h_2]$ (also hard-wired in the algorithm). We can define new $g_{d,v}$ for the 
uniform SENS condition using this recursion by starting with some arbitrary fixed $g_{\eps} = g_{d,\eps} \in S$ for all $d$. Now it is clear that $g_{d,v}$ can be computed in $\DLINTIME$ from $v$~-- actually by a finite state automaton  (note that it is independent of $d$): the states are just the elements of $S$ with initial state $g_\eps$. If the current state is $g = [h_1,h_2]$ and the next input bit from $v$ is 0, then the next state is $h_1$, otherwise, it is $h_2$.

By \cref{lem:subquotient} and \cref{lem:finiteSENS}, every group
having a subgroup with a finite, non-solvable quotient is uniformly
SENS. Since every free group projects to a finite simple group, we get:

\begin{corollary}\label{cor:freeSENS}
  If $F_n$ is a finitely generated free group of rank $n \geq 2$, then $F_n$ is uniformly SENS.
\end{corollary}
This result was essentially shown by Robinson \cite{Robinson93phd},
who showed that the word problem of a
  free group of rank two is $\Nc{1}$-hard. He used a similar commutator
  approach as Barrington. One can prove \cref{cor:freeSENS} also directly by exhibiting a
  free subgroup of infinite rank whose generators are easily
  computable. For example, in $F_2 = \langle x_0,x_1\rangle$ take
  $g_v = x_0^{-v}x_1 x_0^{v}$ for $v \in \oneset{0,1}^d$ viewing the
  string $v$ as a binary encoded number (the other $g_v$ for
  $v \in \oneset{0,1}^{<d}$ are then defined by the commutator
  identity in \cref{def:SENS}), and appropriately padding with
  $1$'s. It is even possible to take the $g_v$ of constant length: consider a free
  group $F=\langle x_0,x_1,x_2\rangle$, and the elements
  $g_v=x_{v\bmod 3}$ with $v$ read as the binary representation of an
  integer. It is easy to see that the nested commutator $g_\eps$ is
  non-trivial.

\begin{example}\label{ex:nonENS} 
  Here is a finitely generated group that is not solvable, has decidable word problem, but
  is not ENS. The construction is inspired from~\cite{Wilson80}.

  Start with the trivial group $H_0=1$ and set
  $H_{n+1}=H_n\wr\Z$. We have a natural embedding $H_0\le H_1$, which
  induces for all $n$ an embedding $H_n\le H_{n+1}$. We set
  $H=\bigcup_{n\ge0}H_n$, and denote by $x_0,x_1,\dots$ the generators
  of $H$, starting with $\Z=\langle x_0\rangle$. In particular,
  $H_d \coloneqq\langle x_0,\dots,x_d\rangle$ is solvable of class precisely $d$ whereas $H$ is non-solvable.

  For an injective function $\tau\colon\N\to\N$ to be specified later,
  consider in the {\em unrestricted} wreath product $H^\Z\rtimes\Z$ the
  subgroup $G$ generated by the following two elements:
  \begin{itemize}
  \item the generator $t$ of $\Z$ and
  \item the function $f\colon\Z\to H$ defined by $f(\tau(n))=x_n$ and all
    other values being $1$.
  \end{itemize}
  We make the assumption that $\tau$ has the following property: For every integer 
  $z \in \mathbb{Z} \setminus \{0\}$ there is at most one 
  pair $(m,i)  \in \N\times\N$ with $z = \tau(m)-\tau(i)$.
  For instance, the mapping $\tau(n) = 2^n$ has this property.
  
  Let us define the conjugated mapping $f_i = t^{\tau(i)} f t^{-\tau(i)} \in G$.  We have 
  $f_i(0) = x_i$ and more generally
 $f_i(\tau(m)-\tau(i)) = x_m$ (and $f_i^{-1}(\tau(m)-\tau(i)) = x_m^{-1}$) for all $m$.
  Consider now a product $g = f_{i_1}^{\alpha_1} \cdots f_{i_k}^{\alpha_k}$
  ($\alpha_1, \ldots, \alpha_k \in \{-1,1\}$). We get
  $g(0) = x_{i_1}^{\alpha_1} \cdots x_{i_k}^{\alpha_k}$. 
  For a position $z \in \mathbb{Z} \setminus \{0\}$ which is not
  a difference of two different $\tau$-values we have $g(z)=1$.
  For all other  non-zero positions $z$ there is a unique
  pair $(m,i)$ such that $z = \tau(m)-\tau(i)$, which yields $g(z) = x_m^e$, where $e$ is the sum of those $\alpha_{j}$
  such that $i_j = i$. Hence, the commutator $[g,h]$ of two mappings 
  $g = f_{i_1}^{\alpha_1} \cdots f_{i_k}^{\alpha_k}$ and  $h = f_{j_1}^{\beta_1} \cdots f_{j_l}^{\beta_l}$
  satisfies $[g,h](0) = [x_{i_1}^{\alpha_1} \cdots x_{i_k}^{\alpha_k}, x_{j_1}^{\beta_1} \cdots x_{j_l}^{\beta_l}]$
  and $[g,h](z)=0$ for all $z \in \mathbb{Z} \setminus \{0\}$.
  Hence, $G$ contains the restricted wreath product $[H,H]\wr\Z$, so in particular is
  infinite and non-solvable; and $G'$ contains the restricted direct product $[H,H]^{(\Z)}$.
  
  We now assume that $\tau$ grows superexponentially (take for instance $\tau(n) = 2^{n^2}$).
  Note that if $k\in\Z$ is not of the form $\tau(i)-\tau(j)$ for some $i, j \in \N$, then
  $t^k f t^{-k}$ and $f$ commute. It follows that the intersection of
  $G''$ with the ball of radius $R$ in $G$ is contained in
  $[H_d,H_d]^\Z$  for $d$ growing
  sublogarithmically in $R$ (more precisely as $\Oh(\sqrt{\log R})$), and in particular does not contain a
  nested non-trivial commutator of depth $\Omega(\log R)$. This implies that $G$ is not SENS (and in fact
  not ENS).

  Furthermore, if $\tau$ is computable, then $\WP(G)$ is
  decidable: given a word $w\in\oneset{t^{\pm1},f^{\pm1}}^*$, compute its
  exponent sum in the letters $t^{\pm1}$ and $f^{\pm1}$ (which must
  both vanish if $w=_G1$) and the coordinates $-|w|,\dots,|w|$ of its
  image in $H^\Z$. Each of these coordinates belongs to a finitely
  iterated wreath product $\Z\wr\cdots\wr\Z$, in which the word
  problem is decidable (again by counting exponents and computing
  coordinates).
\end{example}

\begin{example}\label{ex:nonENS2} 
Here is an example of a f.g.~non-ENS group which is 
uniformly SENS in a larger group. We continue on the notation of Example~\ref{ex:nonENS}.

Consider the non-ENS group $G=\langle t,f\rangle$ from
Example~\ref{ex:nonENS}. The reason that $G$ fails to be uniformly SENS
is the following: there are elements $y_i \in G$ ($i \ge 0$) such
that a non-trivial depth-$d$ nested commutator may uniformly be constructed using
$y_0,\dots,y_{d-1}$, but the $y_i$ have length growing superexponentially
in $i$.  

Essentially by the same construction as in Example~\ref{ex:nonENS} one
can embed $G$ as a heavily distorted subgroup in a finitely generated
subgroup $\widetilde G\coloneqq\langle t,f,\tilde t, \tilde f\rangle$
of the unrestricted wreath product $G^\Z\rtimes\Z$, thereby bringing the
$y_i$ back to exponential length: the elements $t,f$ are the
generators of $G$, seen as elements of $G^\Z$ supported at $0$;
$\tilde t$ is the generator of $\Z$; and $\tilde f\in G^\Z$ takes
value $y_i$ at $2^i$. Then $G$ is uniformly SENS in $\widetilde G$,
since the $[y_i,y_j]$ are expressible as words of length
$2^{\Oh(i+j)}$ in $\tilde f,\tilde t$, and their inverses.
\end{example}
The following technical result will be used to prove that weakly
branched groups and Thompson's group $F$ are uniformly SENS.

\begin{proposition}\label{prop:wreathSENS}
  Let $G$ be a finitely generated group with the standard generating set $\Sigma$. Moreover, let
  $h_d$ $(d\in\N)$ be words over $\Sigma$ with
  $|h_d|\in2^{\Oh(d)}$ and such that given $1^d$ and a binary coded number $i$ with $\Oh(d)$ bits
  one can compute in \DLINTIME the $i$-th letter of $h_d$.
  Assume that $H = \langle h_0,h_1,\dots\rangle$ acts on a tree of words
  $X^*$ (where $X$ is not necessarily finite), and that $X$ contains pairwise distinct elements $v_{-1},v,v_1$ such that
  \begin{itemize}
  \item $h_d$ fixes all of $X^*\setminus v^d X^*$, and
  \item $(v^dv_{-1})^{h_d}=v^{d+1}$ and $(v^{d+1})^{h_d}=v^d v_1$.
  \end{itemize}
   Then $H$ is uniformly SENS in $G$, so in particular $G$ is uniformly SENS.
   Moreover, if $H$ is finitely generated and the $h_d$ are words over the generators of $H$, then $H$ is uniformly SENS.
\end{proposition}

\begin{proof}
For non-negative integers $d,q$ and $s\in\oneset{-1,1}$, consider
	the following elements $g_{d,s,q}$, defined inductively:
	\[g_{0,s,q}=h_q,\qquad g_{d,s,q} = [g_{d-1,-1,0}^s,\, g_{d-1,1,q+1}]\text{ if }d>0.\] 
We claim that $g_{d,1,0} \neq_G 1$. This implies the proposition:
By definition $g_{d,1,0}$ is a
$d$-fold nested commutator of words of the form $h_r^{\pm1}$
for various $r\le d$. It is easy to see that given $v \in\oneset{0,1}^d$, the index $r_v$ corresponding to the leaf of the commutator tree that is indexed by $v$ is computable in \DLINTIME and  by the hypothesis of the proposition $h_{r_v}$ is \DLINTIME-computable. 

Thus, it remains to show that $g_{d,1,0}$ is non-trivial. Indeed, we claim that, for $d>0$, the element $g_{d,s,q}$ acts only on the
  subtrees below $v^{d+q}$ and $v^{d-1}v_s$, and furthermore acts as
  $h_{d+q}$ on the subtree below $v^{d+q}$.

  We prove this claim by induction on $d$. Recall that for
  $g\in \Aut(X^*)$ and a node $w\in X^*$ we write
  $w*g$ for the element of $\Aut(X^*)$ that acts as $g$ on the subtree
  $w X^*$ and trivially elsewhere. Note that a conjugate $(w*g)^h$ with $h \in \Aut(X^*)$
  can be written as $(w*g)^h = w^h * g'$ for some $g' \in \Aut(X^*)$.
  With this notation, we may write
  $h_r=v^r*k_r$ for $k_r=h_r@v^r \in \Aut(X^*)$. Our claim becomes ($\Box$ represents an arbitrary element of $\Aut(X^*)$ that is not important)
  \[g_{d,s,q}=(v^{d+q}*k_{d+q})(v^{d-1}v_s* \Box).\]
  For $d=1$ we have
  \[
  g_{1,s,q}=[h_0^s,h_{1+q}] = \big(h_{1+q}^{h_0^s}\big)^{-1} h_{1+q} = 
  \big( (v^{1+q}*k_{1+q})^{h_0^s}\big)^{-1} (v^{1+q}*k_{1+q}) .
  \] 
  Moreover, the conjugate $(v^{1+q}*k_{1+q})^{h_0^s}$ is of the form
  $(v^{1+q})^{h_0^s} * \Box =v_s*\Box$ and  we get
  $g_{1,s,q}= (v_s*\Box)^{-1} (v^{1+q}*k_{1+q}) = (v^{1+q}*k_{1+q})(v_s*\Box)$.

  Consider now $d>1$. By induction,
  $g_{d-1,-1,0}=(v^{d-1}*k_{d-1})(v^{d-2}v_{-1}*\Box)$ and
  $g_{d-1,1,q+1}=(v^{d+q}*k_{d+q})(v^{d-2}v_1*\Box)$. Now
  $v^{d-2}v_{-1}*f$, $v^{d-1}*g$, and $v^{d-2}v_1*h$ commute for all
  $f,g,h\in \Aut(X^*)$ since they act non-trivially on disjoint subtrees. We get
  \[g_{d,s,q}=[g_{d-1,-1,0}^s,g_{d-1,1,q+1}]=[v^{d-1}*k_{d-1}^s,v^{d+q}*k_{d+q}]=(v^{d-1}v_s*\Box)(v^{d+q}*k_{d+q})\]
  using arguments as for the case $d=1$.
\end{proof}

\begin{theorem}\label{thm:wreathSENS}
	Let $G$ be a finitely generated group with $G \wr H \leq G$ for some non-trivial group $H$. Then $G$ is uniformly SENS.
\end{theorem}

\begin{proof}
	By possibly replacing $H$ with a cyclic subgroup, we can assume that $H=\Z$ or $H=\Z/p$ for some $p\in \Z$. Moreover, we can assume that $p\geq 3$: if $p=2$, we can use the associativity of the permutational wreath product: $G \wr (\Z/2 \wr \Z/2) = (G \wr \Z/2) \wr \Z/2 \leq G \wr \Z/2 \leq G$.
	Thus, since $\Z/2 \wr \Z/2$ contains an element of order $4$, we have  $G \wr \Z/4 \leq G $.
	Hence, we have $G \wr H \leq G$ for $H = \Z$ or $H = \Z/p$ with  $p \geq 3$. Let $t$ be a generator of $H$ and $\Sigma$ be a standard generating set for $G$. \Wlog we can assume that $t \in \Sigma$.
	
	Now, consider the endomorphism $\sigma: G\to G$ given by the embedding $G \wr H \leq G$. 
	After padding with the appropriate number of $1$'s, we can view $\sigma$ as a
	substitution $\sigma : \Sigma\to\Sigma^{2^\lambda}$ for some constant $\lambda$.
	We then define words $h_d
	=\sigma^d(t)$ for all $d\in\N$, and note that
	$|h_d| = 2^{\lambda d}$. It is straightforward to see that on input of $1^d$ and a binary coded number $i$ one can compute in \DLINTIME the $i$-th letter of $h_d$.
	Moreover, it follows that $\langle h_0, \dots, h_k\rangle$ is the $k$-fold iterated wreath product of cyclic groups and so 
	$\langle h_0,h_1,\dots\rangle\cong(\cdots\wr\Z)\wr \Z$ or $\langle h_0,h_1,\dots\rangle\cong(\cdots\wr(\Z/p))\wr (\Z/p)$, which acts on the rooted tree $X^*$ with $X  = H$ in the canonical way. We then apply \cref{prop:wreathSENS} with $(v_{-1},v,v_1)=(-1,0,1)$ (resp.\ $(v_{-1},v,v_1)=(p-1,0,1)$).
\end{proof}
As an immediate consequence of \cref{thm:wreathSENS} and \cref{lem:GS}, we obtain:

\begin{corollary}\label{cor:Thompson}
	Thompson's groups $F < T < V$ are uniformly SENS.
\end{corollary}
One can also show  \cref{cor:Thompson} directly without using  \cref{prop:wreathSENS}.
Consider the infinite presentation \eqref{eq:F}. From the relations $x_i^{-1} x_k x_i = x_{k+1}$ ($i < k$)
the reader can easily check that $g = x_3 x_2^{-1}$ satisfies the identity
\[ g = [g,g^{x_0^{-1}}]^{x_1} = [g^{x_1}, g^{x_0^{-1}x_1}]. \]
Nesting this identity $d$ times and pushing conjugations to the leaf level of the resulting tree yields
the words $g_{d,v}$. More precisely, let us define words $c_v$ ($v \in \{0,1\}^*$) by 
$c_{\eps} = \eps$, $c_{v0} = x_1 c_v$, and $c_{v1} = x_0^{-1} x_1 c_v$. We then define 
$g_{d,v} = g^{c_v}$ for $v \in \{0,1\}^{\le d}$ and immediately get $g_{d,v} = [g_{d,v0}, g_{d,v1}]$ in $F$. 
Clearly, the word $c_v$ can be computed in $\DTIME(\Oh(|v|))$. Hence, $g_{d,v}$ can be computed
in $\DTIME(\Oh(d))$.

\begin{theorem}\label{thm:weaklybranched}
  Let $G$ be a weakly branched self-similar group, and assume that it
  admits a finitely generated branching subgroup $K$. Then $K$ and hence $G$ are
  uniformly SENS.
\end{theorem}

\begin{proof}
  Let $K$ be a finitely generated branching subgroup of $G$ and let $X^*$ be the tree on which $G$ acts. 
  Let $\varphi$ as in \eqref{eq-phi}.
  First, we may find an element $k\in K$ and a vertex $v\in X^*$ such that
  $v, v_{-1}\coloneqq v^{k^{-1}}$, and $v_1 \coloneqq v^k$ are
  pairwise distinct. Indeed $K$ contains an element $k\neq 1$. If $k$
  has order $>2$ (possibly $\infty$), then there is a vertex $v$ on
  which it acts as a cycle of length $>2$. If $k^2=1$, then take a vertex $v$
  with $v^k \neq v$. Then the orbit of $vv$ under $k\cdot(v*k)$ has length four,
  so we only have to replace $k$ by $k\cdot(v*k)$ and $v$ by $vv$.
  After replacing $X$ by $X^{\abs v}$, we can assume that
  $v_{-1},v,v_1\in X$.

  Since $\phi(K)$ contains $K^X$, there exists an endomorphism
  $\sigma$ of $K$, given on generators of $K$ by
  $\sigma(g) = \phi^{-1}(1,\dots 1,g, 1,\dots,1)$ with the unique $g$ in
  position $v$. We fix a standard generating set $\Sigma$ for $K$ and
  express $\sigma$ as a substitution $\sigma \colon \Sigma\to \Sigma^*$. By padding its
  images with $1$'s, we may assume that $\sigma$ maps every generator
  to a word of length $2^\mu$ for some fixed $\mu$. Also without loss
  of generality, we may assume that the $k$ from the previous paragraph is a generator.
  In particular, the words $h_d = \sigma^d(k) \in \Sigma^*$ have length
  $2^{\mu d}$, and the letter at a given position of $h_d$ can be computed in
  $\DTIME(\Oh(d))$. We then apply \cref{prop:wreathSENS}.
\end{proof}
For the special case of the Grigorchuk group we give below an alternative
proof for the uniform SENS property. We show that there exist non-trivial
nested commutators of arbitrary depth with individual entries of
bounded (and not merely exponentially-growing) length and computable
in \DLINTIME:
\begin{proposition}\label{thm:grigorchuk_sens}
  Consider in the Grigorchuk group $G=\langle a,b,c,d\rangle$ the elements
  \[x = (abad)^2\quad \text{and} \quad y = x^b = babadabac.\]
  Define recursively elements $z_v\in\oneset{x,y,x^{-1},y^{-1}}$ for all
  $v\in\oneset{0,1}^*$ as follows:
  \begin{itemize}
  \item $z_\eps = x$;
  \item if $z_v$ is defined, then we define $z_{v0}$ and $z_{v1}$ according to the following table:
   \medskip
  \begin{center}
    \begin{tabular}{|c@{\hskip\tabcolsep\vrule width 1pt\hskip\tabcolsep}c|c|}
      $z_v$ & $z_{v0}$ & $z_{v1}$ \\ \thickhline 
            $x$ & $x^{-1}$   & $y^{-1}$  \\ \hline
   $x^{-1}$ & $y^{-1}$ & $x^{-1}$  \\ \hline
          $y$ & $y$ & $x$ \\ \hline
  $y^{-1}$ & $x$ & $y$  
   \end{tabular}
 \end{center}
 \medskip
  \end{itemize}
  For every $d\in\mathbb N$  and $v \in \{0,1\}^{\le d}$ let 
  $g_{d,v} = z_v$ for $\abs v = d$ and $g_{d,v} = [g_{v0},g_{v1}]$ if $\abs v<d$.
  We then have $g_{d,\eps} \neq 1$ in $G$.
  In particular, $G$ is uniformly SENS.
\end{proposition}
 
\begin{proof}
  That $x\neq1 \neq y$ is easy to check by computing their action on the
  third level of the tree.
  Now the following equations are easy to check in $G$:
  \begin{align*}
    [x,y] &= \pair<1,{\pair<1,y^{-1}>}>,  \\ 
    [x^{-1},y^{-1}] &= \pair<1,{\pair<1,x>}>,  \\ 
    [y,x] &= \pair<1,{\pair<1,y>}>,  \\ 
    [y^{-1},x^{-1}] &= \pair<1,{\pair<1,x^{-1}>}>.  
  \end{align*}
  In other words: $[z_{v0},z_{v1}] = \pair<1,{\pair<1,z_v>}>$.
  The checks are tedious to compute by hand, but easy in the GAP
  package FR (note that vertices are numbered from $1$ in GAP and from
  $0$ here):
\begin{verbatim}
gap> LoadPackage("fr");
gap> AssignGeneratorVariables(GrigorchukGroup);
gap> x := (a*b*a*d)^2; y := x^b;
gap> Assert(0,Comm(x,y) = VertexElement([2,2],y^-1));
gap> Assert(0,Comm(x^-1,y^-1) = VertexElement([2,2],x));
gap> Assert(0,Comm(y,x) = VertexElement([2,2],y));
gap> Assert(0,Comm(y^-1,x^-1) = VertexElement([2,2],x^-1));
\end{verbatim}
  We wish to prove that $g_{d,\eps} \neq 1$ in $G$. Now the equation  $[z_{v0},z_{v1}] = \pair<1,{\pair<1,z_v>}>$
  immediately implies that $g_{d,v}$ acts as $z_v$ on the subtree
  below vertex $1^{2(d-\abs v)}$
  and trivially elsewhere. In particular,
  $g_{d,\eps}$ acts as $z_\eps=x \neq 1$ on the subtree below vertex
  $1^{2d}$ and is non-trivial.
 	
  With this definition, the $g_{d,v}$ satisfy the definition of a SENS group. 
  Moreover, given some $v\in \oneset{0,1}^d$, $g_{d,v}$ can be
  computed in time $\Oh(d)$ by a deterministic finite state automaton with state set 
  $\oneset{x^{\pm1},y^{\pm1}}$.
\end{proof}

\section{Efficiently non-solvable groups have \Nc1-hard word problem}
We are ready to state and prove our generalization of Barrington's theorem, namely that SENS groups have
\Nc1-hard word problems, both in the non-uniform and uniform setting. We start
with the non-uniform setting.

\begin{theorem}\label{thm:NC1hard}
  Let $G$ be efficiently non-solvable and let $\Sigma$ be a finite standard
  generating set for $G$. Then every language in \Nc1 can be recognized by a family of 
  $(G,\Sigma)$-programs of polynomial length. In particular, $\WP(G)$
  is hard for \Nc{1} under projection reductions as well as
  \Ac0-many-one-reductions.
\end{theorem}
Note that for the second statement we need the padding letter $1$ in
the generating set for $G$; otherwise, we get a \TC-many-one
reduction.

The proof of \cref{thm:NC1hard} essentially follows Barrington's proof
that the word problem of finite non-solvable groups is \Nc1-hard
\cite{Barrington89}. The crucial observation here is that it suffices
to construct for every gate $v$ only one $G$-program (plus one for the
inverse) which evaluates to $g_{d,v}$ or to $1$ depending on the truth
value $v$ evaluates to, where $g_{d,v}$ is from \cref{def:SENS}.

Also note that Barrington uses conjugates of commutators in his proof and iterates this process. 
However, since $z^{-1}[x,y]z = [z^{-1}xz,z^{-1}yz]$ in every group, the conjugating elements can be pushed through to the inner-most level.

\begin{proof}
  By Lemma~\ref{lem:ENSimpliesSENS} it suffices to prove the statement for the case where
  $G$ is strongly efficiently
  non-solvable. 
	
  Given a language $L$ in \Nc1, we start by
  constructing a family of $G$-programs for $L$. For this let $(C_n)_{n \in \N}$
  be an \Nc1 circuit family for $L$. Let us fix an input length $n$ and write
  $C = C_n$. 
  Since \Nc1 is closed under complementation, we can assume that for every input word $x \in \{0,1\}^n$, we have $x \in L$ if and only if the output gate of the circuit $C$ evaluates to $0$ on input $x$.  
  By \cref{lem:layeredNAND} we may assume that $C$ is a
balanced  \nand-tree-circuit of depth $d \in \Oh(\log n)$ with each leaf
  labelled by a possibly negated input variable or constant via the input mapping $q_n\colon \oneset{0,1}^{d} \to \ccinterval{1}{n} \times \oneset{0,1} \times \oneset{0,1}$.
  All non-leaf gates  are  {\sf nand}-gates.
	
  For each gate $v \in \oneset{0,1}^{\leq d}$ let
  $g_v = g_{d,v}$ as in \cref{def:SENS}. We construct two $G$-programs $P_v$
  and $P_v^{-1}$ (both of input length $n$) such that for every input $x \in \{0,1\}^n$
  ($x$ is taken as the input for $C$,  $P_v$, and $P_v^{-1}$) we have
  \begin{align}\label{eq:pvdef}
    P_v[x] &=_G \begin{cases}
      g_v &\text{if } v \text{ evaluates to } 1,\\
      1 &\text{if } v \text{ evaluates to } 0, 		
    \end{cases} 
  \end{align}
  and $P_v^{-1}[x] = P_v[x]^{-1}$ in $G$.  Notice that we have
  $g_v P_v^{-1}[x] = g_v$ if $v$ evaluates to $0$ and $g_v P_v^{-1}[x] = 1$,
  otherwise. Thus, $g_v P_v^{-1}$ is a $G$-program for the ``negation''
  of $P_v$. Moreover, by \cref{eq:pvdef}, $P_\eps$ evaluates to $1$ on input $x$ if and only if the output gate evaluates to $0$ which by our assumption was the case if and only if  $x \in L$.
	
  The construction of the $P_v$ and $P_v^{-1}$ is straightforward:
  For an input gate $v \in \oneset{0,1}^{d}$ we simply define $P_v$ to be a $G$-program evaluating to $g_v$ or $1$~-- in which case it evaluates to which element depends on $q_n(v)$. 
  More precisely, write $g_v = a_1\cdots a_m$ with $a_i \in \Sigma$. If $q_n(v) = \langle i, a,b\rangle$ for $i \in \ccinterval{1}{n}$ and $a,b \in \oneset{0,1}$, we set 
  $P_v =  \langle i, a_1^a,a_1^b\rangle \cdots  \langle i, a_m^a,a_m^b\rangle$ and $P_v^{-1} =  \langle i, a^{-a}_m,a^{-b}_m\rangle \cdots  \langle i, a^{-a}_1,a^{-b}_1\rangle$.
  For a {\sf nand}-gate $v$ with
  inputs from $v0$ and $v1$, we define
  \begin{align*}
    P_v &= g_v [P_{v1}, P_{v0}]=g_v  P_{v1}^{-1}  P_{v0}^{-1} P_{v1} P_{v0}, \\
    P_v^{-1} &=  [P_{v0}, P_{v1}] g_v^{-1}= P_{v0}^{-1} P_{v1}^{-1}  P_{v1} P_{v1} g_v^{-1},
  \end{align*}
  where the $g_v$  and $g_v^{-1}$ represent constant
  $G$-programs evaluating to $g_v$ and $g_v^{-1}$, respectively, irrespective of the actual input
  (such constant $G$-programs consist of triples of the form $\langle 1, a,a\rangle$ for $a \in \Sigma$). These
  constant $G$-programs are defined via the commutator identities
  $g_v = \bigl[ g_{v0},\, g_{v1} \bigr]$ for $v \in \oneset{0,1}^{< d}$ in
  \cref{def:SENS}.

  Clearly, by induction we have $ P_v[x]^{-1} = P_v^{-1}[x]$ in $G$ (for every input $x$). Let us
  show that \cref{eq:pvdef} holds: For an input gate
  $v \in \oneset{0,1}^{d}$, \cref{eq:pvdef} holds by definition. Now, let
  $v \in \oneset{0,1}^{< d}$. Then, by induction, we have the following equalities in $G$:
  \begin{align*}
    P_v[x] = g_v  [P_{v1}[x], P_{v0}[x]]
    &= 	\begin{cases}
      g_v  & \text{if } v0 \text{ or }  v1 \text{ evaluates to } 0, \\
      g_v [ g_{v1},\,  g_{v0} ] & \text{if } v0 \text{ and }  v1 \text{ evaluate to } 1, \\
    \end{cases}\\
    &= 	\begin{cases}
      g_v  & \text{if } v \text{ evaluates to } 1, \\
      1 & \text{if } v \text{ evaluates to } 0.		
    \end{cases}
  \end{align*}
  Note that $[ g_{v1},\,  g_{v0}] = [ g_{v0},\,  g_{v1}]^{-1} = g_v^{-1}$ for the last equality.
  Thus, $P_v$ satisfies \cref{eq:pvdef}. For $P_v^{-1}$ the analogous
  statement can be shown with the same calculation. 
  For a leaf $v \in \{0,1\}^d$, we have $\abs{g_v} \in 2^{\mathcal{O}(d)} = n^{\mathcal{O}(1)}$
  by Condition~(\ref{SENSa}) from \cref{def:SENS} (recall that $d \in \mathcal{O}(\log n)$).
  Hence, $P_v^{-1}$ and $P_v$ have polynomial
  length in $n$. Finally, also $P_\eps$ has polynomial length in $n$ (with the same
  argument as for $g_\eps$; see the remark after \cref{def:SENS}).

  The fact that $\WP(G)$ is \Nc1-hard under projection reductions as well as
  \Ac0-many-one-reductions follows now form \cref{rem:Gprogramprojection}.
\end{proof}

\begin{remark}
The above construction also shows that from a given Boolean formula (i.e., a tree-like circuit that is given as an expression) 
$F$ with variables $x_1, \ldots, x_n$ one can compute in \L a $G$-program $P$ with input length $n$ such that
for every $x = a_1 \cdots a_n \in \{0,1\}^n$ we have $P[x] = 1$ if and only if $F$ evaluates to true
when variable $x_i$ receives the truth value $a_i$ for $1 \leq i \leq n$.
To show this, one first has to balance $F$ in the sense that $F$ is transformed into 
an equivalent Boolean formula of depth $\mathcal{O}(\log |F|)$. This can be done even in $\Tc{0}$ \cite{GanardiL19}.
This fact will be used in the forthcoming paper \cite{FGLZ20}.
\end{remark}

\begin{theorem}\label{thm:uniNC1hard}
  Let $G$ be uniformly (strongly) efficiently non-solvable and
  $\Sigma$ be a finite standard generating set of $G$. Then every language in
  \ALOGTIME can be recognized by a uniform family of polynomial length
  $(G,\Sigma)$-programs. In particular, $\WP(G)$ is hard for \ALOGTIME
  under uniform projection reductions (thus, also under \DLOGTIME-reductions). 
\end{theorem}
Notice that again for this theorem we need the padding letter $1$ in
$\Sigma$ and that all letters of $\Sigma$ are encoded using the same
number of bits; otherwise, we get a \TC-many-one reduction.

The proof of \cref{thm:uniNC1hard} is conceptually simple, but the
details are quite technical: We know that \ALOGTIME is the same as
\DLOGTIME-uniform \Nc{1}, so we apply the construction of \cref{thm:NC1hard}. By
a careful padding with trivial $G$-programs, we can ensure that from
the binary representation of some index $i$, we can read in \DLOGTIME
the input gate of the \Nc{1}-circuit on which the $i$-th instruction
in the $G$-program depends (this is the main technical part of the
proof). Then the theorem follows easily from the requirements of
being uniformly SENS and from the special type of \DLOGTIME-uniformity of
the circuit shown in \cref{lem:layeredNANDuniform}.

\begin{proof}
  We only prove the statement for the case where $G$ is uniformly SENS;
  the general case is more technical, but follows exactly the same outline.
  
  By \cref{thm:NC1hard}, we know that every language $L$ in \ALOGTIME can
  be recognized by a family of polynomial length
  $(G,\Sigma)$-programs. It remains to show that the construction of
  the $G$-programs is uniform.  In order to do so, we refine the
  construction of \cref{thm:NC1hard}.
	
 Fix a constant $\mu$ such that 
  for all $v \in \oneset{0,1}^d$ the word $g_v = g_{d,v}$ has
  length $2^{\mu d}$.
  We start with an \ALOGTIME-machine $M$. By
  \cref{lem:layeredNANDuniform}, we can assume that the balanced {\sf nand}-tree-circuit family $(C_n)_{n\in \N}$ in the proof of
  \cref{thm:NC1hard} is \DLOGTIME-uniform in the sense that the depth function $1^n \mapsto 1^{d(n)}$ as well as the input mapping $q$
  from \cref{nand-tree-circuit} can be strongly computed in \DLOGTIME.	
  Fix an input length $n$ and let $d = d(n)$ be the depth of the circuit $C = C_n$.
  From $1^n$ we can strongly compute $1^d$ in \DLOGTIME by the above assumptions.
  
  We now follow the recursive definition of the $G$-programs $P_v$ and
  $P_v^{-1}$ from the proof of \cref{thm:NC1hard}.  In order to have a nicer
  presentation, we wish that all $G$-programs corresponding to one
  layer of the circuit have the same length. To achieve this, we also
  define the constant $G$-programs $g_v$ and $g_{v}^{-1}$ precisely
  (which evaluate to the recursive commutators from
  \cref{def:SENS}). Moreover, for each $v \in \oneset{0,1}^{\leq d}$ we
  introduce a new constant $G$-program $1_v$ of the same length as $g_v$ which evaluates to
  $1$ in $G$. For $v \in \oneset{0,1}^d$ the program $1_v$ is the instruction $\langle 1, 1, 1\rangle$ repeated $2^{\mu d}$ times. The programs $1_v$ are
  only there for padding reasons and $1_u$ and $1_v$ are the same for
  $\abs{u} = \abs{v}$.
	
  Now the $G$-programs $P_v$, $P_v^{-1}$, $ g_v$, $g_v^{-1}$, and
  $1_v$ corresponding to a gate $v \in \oneset{0,1}^{< d}$ are defined as follows
  (note that each of these programs consists of 8 blocks):
  \begin{eqnarray}
  P_v &=& g_{v0}^{-1}g_{v1}^{-1}g_{v0}g_{v1}\; P_{v1}^{-1} P_{v0}^{-1} P_{v1} P_{v0} \label{eq-P1}\\
  P_v^{-1} &=& P_{v0}^{-1} P_{v1}^{-1} P_{v0}P_{v1}\;  g_{v1}^{-1}g_{v0}^{-1}g_{v1}g_{v0} \label{eq-P2}\\
  g_v &=& g_{v0}^{-1}g_{v1}^{-1}g_{v0}g_{v1}\; 1_{v0}1_{v0}1_{v0}1_{v0} \label{eq-P3}\\
  g_v^{-1} &=&  g_{v1}^{-1}g_{v0}^{-1}g_{v1}g_{v0}\; 1_{v0}1_{v0}1_{v0}1_{v0} \label{eq-P4}\\
 1_v &=& 1_{v0}1_{v0}1_{v0}1_{v0}\; 1_{v0}1_{v0}1_{v0}1_{v0} \label{eq-P5}.
  \end{eqnarray}
  Clearly, these $G$-programs all evaluate as described in the proof
  of \cref{thm:NC1hard} and all programs corresponding to one layer
  have the same length. Moreover, for $v \in \oneset{0,1}^{< d}$ with
  $\abs{v} = c$ the length of the $G$-program $g_v$ is exactly
  $2^{\mu d+3(d - c)}$ and, thus, also the length of $P_v$ and
  $P_v^{-1}$ is exactly $2^{\mu d + 3(d - c)}$. 
  
  For the $G$-program $P_{\eps}$ (which has length $2^{(\mu+3)d}$) we can prove the uniformity condition: Given the string $1^n$
  and a binary coded integer $i \in \cointerval{0}{2^{(\mu+3)d}}$ with $(\mu+3)d \in \Oh(\log n)$ bits, we
  want to compute in \DLOGTIME the $i$-th instruction in $P_\eps$, where
  $P_\eps$ is the $G$-program assigned to the $n$-input
  circuit. Note that \DLOGTIME means time $\Oh(\log n)$ (due to the input $1^n$).
  Since we have computed $1^d$ already in \DLOGTIME, we can check
  in \DLOGTIME whether $i$ has indeed $(\mu+3)d$ bits.
  
Next, given $i$  and $1^n$, the \DLOGTIME-machine goes over the first
$3d$ bits of $i$ and thereby computes an input gate
$v \in \oneset{0,1}^{d}$ of $C$ bit by bit together with one of the five symbols $\sigma \in \{ P_\ast, P_\ast^{-1}, g_\ast, g_\ast^{-1}, 1_\ast \}$.
The meaning of $v$ and $\sigma$ is that 
$\sigma[\ast \to v]$ (which is obtained by replacing $\ast$ by  $v \in \oneset{0,1}^{d}$ in $\sigma$)
is the $G$-program to which the $i$-th instruction in $P_\eps$ belongs to.
The approach is similar to~\cite[Theorem 4.52]{Vollmer99}. We basically run a deterministic finite state transducer
with states $P_\ast, P_\ast^{-1}, g_\ast, g_\ast^{-1}, 1_\ast$ that reads
three bits of $i$ and thereby outputs one bit of $v$. We start with $\sigma = P_\ast$. 
Note each of the $G$-programs $P_v$, $P_v^{-1}$, $g_v$, $g_v^{-1}$, $1_v$ for $|v|<d$ consists of $8 = 2^3$
blocks of equal length. The next three bits in $i$ determine to which block we have to descend. Moreover,
the block determines the next bit of $v$ and the next state. 
Let us give an example:  assume that the current state $\sigma$ is $P_\ast$ and $b \in \{0,1\}^3$ is the next 3-bit block of $i$.
Recall that $P_v = g_{v0}^{-1}g_{v1}^{-1}g_{v0}g_{v1}P_{v1}^{-1} P_{v0}^{-1} P_{v1} P_{v0}$ for $|v|<d$.
The following operations are done:
\begin{itemize}
	\item If $b = 000$, then print $0$ and set $\sigma :=g^{-1}_\ast$ (descend to block $g_{v0}^{-1}$).
	\item If $b = 001$, then print $1$ and set $\sigma :=g^{-1}_\ast$ (descend to block $g_{v1}^{-1}$).
	\item If $b = 010$, then print $0$ and set $\sigma :=g_\ast$ (descend to block $g_{v0}$).
	\item If $b = 011$, then print $1$ and set $\sigma :=g_\ast$ (descend to block $g_{v1}$).
	\item If $b = 100$, then print $1$ and set $\sigma :=P_\ast^{-1}$ (descend to block $P_{v1}^{-1}$).
	\item If $b = 101$, then print $0$ and set $\sigma :=P_\ast^{-1}$  (descend to block $P_{v0}^{-1}$).
	\item If $b = 110$, then print $1$ and set $\sigma :=P_\ast$  (descend to block $P_{v1}$).
	\item If $b = 111$, then print $0$ and set $\sigma :=P_\ast$  (descend to block $P_{v0}$).
\end{itemize}
For other values of $\sigma$ the behavior of the machine is similar and implements the definitions for $P_v^{-1}$, $g_v$, $g_v^{-1}$, and $1_v$ in 
\eqref{eq-P2}--\eqref{eq-P5}.

Assume now that the above $\DLOGTIME$-machine has computed $v\in \oneset{0,1}^{d}$ and $\sigma \in  \{ P_\ast, P_\ast^{-1}, g_\ast, g_\ast^{-1}, 1_\ast \}$.
If $\sigma = 1_\ast$, then the $i$-th instruction of $P_\eps$ is the padding instruction $\langle 1, 1, 1\rangle$. If $\sigma \in \{ P_\ast, P_\ast^{-1}, g_\ast, g_\ast^{-1} \}$,
then the machine reads the last $\mu d$ bits of the binary encoding of $i$.  These $\mu d$ bits are interpreted
as a binary coded position $j$ in $g_{d,v}$ or $g_{d,v}^{-1}$.  Assume that $\sigma \in \{ P_\ast, g_\ast\}$. The machine then computes the $j$-th symbol $a \in \Sigma$ of $g_{d,v}$ in $\DTIME(\Oh(d))$
according to \cref{def:SENS} (and, thus, in \DLOGTIME as $d \in \Oh(\log n)$ and $1^n$ is part of the input)
 and outputs the instruction $\langle 1, a, a \rangle$ in case $\sigma = g_\ast$. 
If $\sigma = P_\ast$, then $q(1^n,v)$ has to be computed, which can be done in $\DLOGTIME$ by \cref{lem:layeredNANDuniform}.
If $q(1^n,v) = \langle k, b,c\rangle$ with $k \in \ccinterval{1}{n}$ and $b,c \in \oneset{0,1}$, the machine then outputs the instruction $\langle k,a^b,a^c\rangle$.
If $\sigma = \{P_\ast^{-1}, g^{-1}_\ast\}$, then we proceed in a similar fashion. Instead of the $j$-th letter of $g_v$ we have to compute
the $j$-the letter of $g_v^{-1}$, which is the inverse of the $(2^{\mu d} - j +1)$-th letter of $g_v$. The binary coding of $2^{\mu d} - j +1$ can 
be computed in time $\Oh(\log n)$ (and hence \DLOGTIME) since subtraction can be done in linear time.
 Thus, we have obtained a \DLOGTIME-uniform  family of $G$-programs for $L$.

 The second part of the theorem (that $\WP(G)$ is hard for \ALOGTIME under uniform projection reductions) follows again from \cref{rem:Gprogramprojection}.
\end{proof}
Corollary~\ref{coro:A} from the introduction is a direct consequence of Corollaries~\cref{cor:Thompson}, \cref{thm:weaklybranched}, and \cref{thm:uniNC1hard}.

Here is another application of \cref{thm:uniNC1hard}: in \cite{KonigL15} it was shown
that for every f.g.~linear solvable group the word problem belongs to \DLOGTIME-uniform \TC.
It was also asked whether for every f.g.~linear group the word problem is in \DLOGTIME-uniform \TC or \ALOGTIME-hard (it might be the case
that \DLOGTIME-uniform \TC = \ALOGTIME).
We can confirm this. Recall that a group $G$ is called $\mathsf{C}_1$-by-$\mathsf{C}_2$ for group classes $\mathsf{C}_1$ and $\mathsf{C}_2$ 
if $G$ has a normal subgroup $K \in \mathsf{C}_1$ such that $G/K \in \mathsf{C}_2$.

\begin{theorem} \label{thm-linear-groups}
For every f.g.~linear group the word problem is  in \DLOGTIME-uniform \TC or \ALOGTIME-hard. More precisely: let $G$ be a f.g.~linear group.
\begin{itemize}
\item If $G$ is finite solvable, then $\WP(G)$ belongs to \DLOGTIME-uniform \ACC.
\item If $G$ is infinite solvable, then $\WP(G)$ is complete for \DLOGTIME-uniform \TC (via uniform \Ac0 Turing reductions).
\item If $G$ is solvable-by-(finite non-solvable), then $\WP(G)$ is complete for  \ALOGTIME (via \DLOGTIME or uniform projection reductions).
\item In all other cases, $\WP(G)$ is \ALOGTIME-hard and in \L.
\end{itemize}
\end{theorem}
Note that we can obtain a similar dichotomy for hyperbolic groups: they are either virtually abelian or contain a non-abelian free subgroup. In the first case, the word problem is in 
\DLOGTIME-uniform \TC, in the second case it is \ALOGTIME-hard. 

\begin{proof}
Let $G$ be f.g.~linear. First of all, by \cite{LiZa77,Sim79}, $\WP(G)$ belongs to \L.
By Tits alternative \cite{Tits72}, $G$ either contains a free subgroup of rank 2 or is virtually solvable.
In the former case, $\WP(G)$ is \ALOGTIME-hard by \cref{cor:freeSENS} and \cref{thm:uniNC1hard}.
Let us now assume that $G$ is virtually solvable. Let $K$ be a solvable subgroup of $G$ of finite index.
By taking the intersection of all conjugates of $K$ in $G$, we can assume that $K$ is a normal subgroup of $G$.
If also $G/K$ is solvable, then $G$ is solvable. Hence, $\WP(G)$ is in \DLOGTIME-uniform \ACC (if $G$ is finite) or, by \cite{KonigL15}, complete for 
 \DLOGTIME-uniform \TC (if $G$ is infinite). Finally, assume that the finite group $G/K$ is non-solvable (thus, $G$ is 
solvable-by-(finite non-solvable). By Lemmas~\ref{lem:subquotient} and~\ref{lem:finiteSENS}, $G$ is uniformly SENS, and 
\cref{thm:uniNC1hard} implies that $\WP(G)$ is \ALOGTIME-hard. Moreover, by \cite[Theorem~5.2]{Robinson93phd}, $\WP(G)$ is $\Ac{0}$-reducible 
to $\WP(K)$ and $\WP(G/K)$. The latter belongs to $\ALOGTIME$ and $\WP(K)$ belongs to \DLOGTIME-uniform \ACC if $K$ is finite and
to  \DLOGTIME-uniform \TC if $K$ is infinite (note that $K$ as a finite index subgroup of $G$ is f.g.~linear too).
In all cases, $\WP(G)$ belongs to \ALOGTIME.
\end{proof}

\section{Compressed words and the compressed word problem}

In the rest of the paper we deal with the compressed word problem, which is a succinct version of the word problem, where
the input word is given in a compressed form by a so-called straight-line program. In this section, we introduce  straight-line programs
and the compressed word problem and state a few simple facts. For more details on the compressed word problem see 
\cite{Lohrey14compressed}.

A {\em straight-line program} (SLP for short) over the alphabet $\Sigma$ is a triple $\mathcal{G} = (V,\rho,S)$, where 
$V$ is a finite set of variables such that $V \cap \Sigma = \emptyset$, $S \in V$ is the start variable, and $\rho: V \to (V \cup \Sigma)^*$ is a mapping
such that the relation $\{ (A,B) \in V \times V \colon B \text{ occurs in } \rho(A) \}$
is acyclic. For the reader familiar with context free grammars, it might be helpful to view the SLP
$\mathcal{G} = (V,\rho,S)$ as the context-free grammar $(V,\Sigma, P, S)$, where $P$ contains all productions $A \to \rho(A)$ for $A \in V$.
The definition of an SLP implies that this context-free grammar derives exactly on terminal word, which will be denoted by
$\val(\mathcal{G})$. Formally, one can extend $\rho$ to a morphism $\rho : (V \cup \Sigma)^* \to (V \cup \Sigma)^*$ by setting
$\rho(a)=a$ for all $a \in \Sigma$. The above acyclicity condition on $\rho$ implies that for $m = |V|$ we have $\rho^m(w) \in \Sigma^*$ for all
$w \in (V \cup \Sigma)^*$. We then define $\val_{\mathcal{G}}(w) = \rho^m(w)$ (the string derived from the sentential form $w$) and 
$\val(\mathcal{G}) = \val_{\mathcal{G}}(S)$.

The word $\rho(A)$ is also called the {\em right-hand side} of $A$. 
Quite often, it is convenient to assume that all right-hand sides are of the form $a \in \Sigma$ or $BC$ with $B,C \in V$.
This corresponds to the well-known Chomsky normal form for context-free grammars. There is a simple linear time 
algorithm that transforms an SLP $\mathcal{G}$ with $\val(\mathcal{G}) \neq \varepsilon$ into an SLP $\mathcal{G}'$ in Chomsky normal
form with $\val(\mathcal{G}) = \val(\mathcal{G}')$, see e.g. \cite[Proposition 3.8]{Lohrey14compressed}.

We define the size of the SLP $\mathcal{G} = (V,\rho,S)$ as the total length of all right-hand sides:  $|\mathcal{G}| = \sum_{A \in V} |\rho(A)|$.
SLPs offer a succinct representation of words that contain many repeated substrings. For instance, the word $(ab)^{2^n}$ can be produced
by the SLP $\mathcal{G} = (\{A_0, \ldots,A_n\},\rho,A_n)$ with $\rho(A_0) = ab$ and $\rho(A_{i+1}) = A_i A_i$ for $0 \leq i \leq n-1$.
Plandowski has shown that one can check in polynomial time whether two given SLPs produce the same string.
We need the following upper bound on the length of the word $\val(\mathcal{G})$:

\begin{lemma}[c.f.~\mbox{\cite{CLLLPPSS05}}] \label{lemma-SLP-upper-bound}
For every SLP $\mathcal{G}$ we have $|\val(\mathcal{G})| \leq  3^{|\mathcal{G}|/3}$.
\end{lemma}
We also need polynomial time algorithms for a few algorithmic problems for SLPs:

\begin{lemma}[\mbox{\cite[Chapter 3]{Lohrey14compressed}}] \label{lemma-basic}
The following problems can be solved in polynomial time, where $\mathcal{G}$ is an SLP over a terminal alphabet $\Sigma$, $a \in \Sigma$, and $p, q \in \mathbb{N}$
(the latter are given in binary notation):
\begin{itemize}
\item Given $\mathcal{G}$, compute the length $|\val(\mathcal{G})|$.
\item Given $\mathcal{G}$ and $a$, compute the number $|\val(\mathcal{G})|_a$ of occurrences of $a$.
\item Given $\mathcal{G}$ and $p$, compute the symbol $\val(\mathcal{G})[p] \in \Sigma$ (in case $0 \le p < |\val(\mathcal{G})|$ does not hold, the algorithm outputs a special symbol).
\item Given $\mathcal{G}$ and $p,q$, compute an SLP for the string $\val(\mathcal{G})[p:q]$ (in case $0 \le p \leq q < |\val(\mathcal{G})|$ does not hold, the algorithm
outputs a special symbol).
\end{itemize}
\end{lemma}

\begin{lemma}[c.f.~\mbox{\cite[Lemma 3.12]{Lohrey14compressed}}]  \label{lemma-SLP-iterated-morph}
Given a symbols $a_0 \in \Sigma$ and a sequence of morphisms $\phi_1, \ldots, \phi_n : \Sigma^* \to \Sigma^*$, where every $\phi_i$ is given by a list of the words
$\phi_i(a)$ for $a \in \Sigma$, one can compute in \L an SLP for the word
$\phi_1(\phi_2( \cdots \phi_n(a_0) \cdots))$.
\end{lemma}
The {\em compressed word problem} for a finitely generated group $G$ with the finite standard generating set $\Sigma$, $\CompWP(G,\Sigma)$ for short,
is the following decision problem:
\begin{description}
\item[Input:]  an SLP $\mathcal{G}$ over the alphabet $\Sigma$.
\item [Question:] does $\val(\mathcal{G})=1$ hold in $G$?
\end{description}
It is an easy observation that the computational complexity of the compressed word 
problem for $G$ does not depend on the chosen generating set $\Sigma$ in the sense
that if $\Sigma'$ is another finite standard generating set for $G$, then $\CompWP(G,\Sigma)$
is \L-reducible to $\CompWP(G,\Sigma')$ \cite[Lemma~4.2]{Lohrey14compressed}. Therefore we do not
have to specify the generating set and we just write $\CompWP(G)$.

The compressed word problem for $G$ is equivalent to the problem whether a given circuit over the group $G$ evaluates to $1$: Take an SLP $\mathcal{G} = (V,\rho,S)$ in
Chomsky normal form and built a circuit by taking $V$ is the set of gates. If $\rho(A)=a  \in \Sigma$ then $A$ is an input gate that is labelled
with the group generator $a$. If $\rho(A) = BC$ with $B,C \in V$ then $B$ is left input gate for $A$ and $C$ is the right input gate for $A$.
Such a circuit can be evaluated in the natural way (every internal gate computes the product of its input values) and the circuit output is the value
at gate $S$.

From a given SLP $\mathcal{G}$ a $\PSPACE$-transducer can compute the word $\val(\mathcal{G})$.
With Lemma~\ref{lemmaPSPACE} we get:

\begin{lemma} \label{lemma-cwp-PSPACE}
If $G$ is a finitely generated group such that $\WP(G)$ belongs to $\polyL$,
then $\CompWP(G)$ belongs to $\PSPACE$.
\end{lemma}

\section{Compressibly SENS groups}\label{sec:CSENS}

In this section, we present a variant of (uniformly) SENS property that allows to derive
\P-hardness of the compressed word problem.

\begin{definition}\label{def:compSENS}
  We call a group $G$ generated by a finite standard generating set
  $\Sigma$ \emph{compressibly strongly efficiently non-solvable}
  (compressibly SENS) if there is a polynomial $p$ and a collection of words
  $g_{d,i,j} \in \Sigma^*$ for $d \in \N$,
  $0\leq i \leq d$, and $1 \leq j \leq p(d)$ such that
  \begin{enumerate}[(a)]
  \item\label{compSENSa} for all $d \in \N$ and $1 \le j \le p(d)$ there is an SLP of size at most $p(d)$
    evaluating to $g_{d,d,j}$,
  \item\label{compSENSb} for all $d \in \N$, $0\leq i < d$ and $1 \leq j \leq p(d)$ there are
    $k,\ell \in \ccinterval{1}{p(d)}$ such that
    $g_{d,i,j} = \bigl[ g_{d,i+1,k},\, g_{d,i+1,\ell} \bigr]$,
  \item\label{compSENSc} $g_{d,0,1} \neq 1$ in $G$.
  \end{enumerate}
$G$ is called \emph{compressibly efficiently non-solvable}
  (compressibly ENS) if instead of (\ref{compSENSb}), we only require
  \begin{description}
  \item[{\em (b')}] there is some constant $M$ such that for
    all $ d \in N$, $0\leq i < d $ and $1 \leq j \leq p(d)$ there are
    $k_1,\ell_1, \dots, k_M, \ell_M$ such that
    $g_{d,i,j} = \bigl[ g_{d,i+1,k_1},\, g_{d,i+1,\ell_1} \bigr]\cdots
    \bigl[ g_{d,i+1,k_M},\, g_{d,i+1,\ell_M} \bigr]$.
  \end{description}
  If $d$ is clear from the context, then we write $g_{i,j}$ for $g_{d,i,j}$.
  
  $G$ is called \emph{\shortL-uniformly compressibly SENS} if,
  moreover,
  \begin{enumerate}[(a)]
    \setcounter{enumi}{3}
  \item\label{compSENSu1}  on input of the string
    $1^d$ and a binary encoded number $j$ one can compute in \L an SLP for $g_{d,d,j}$, and
  \item\label{compSENSu2} on input of the string
    $1^d$ and binary encoded numbers $i$ and $j$ one can compute in \L the binary
    representations of $k$ and $\ell$ such that
    $g_{d,i,j} = \bigl[ g_{d,i+1,k},\, g_{d,i+1,\ell} \bigr]$.
  \end{enumerate}
Analogously \emph{\shortL-uniformly compressibly ENS} is defined.
\end{definition}

\begin{remark}\label{rem:allpolySLP}
  Clearly, starting from the SLPs for $g_{d,d,j}$ and using the
  commutator identities (\ref{compSENSb}), we obtain SLPs of polynomial size for
  all $g_{d,i,j}$. Moreover, in the \shortL-uniform case, these SLPs  can be computed in \L from 
  $1^d$, $i$ and $j$ (the latter two given in binary representation).
  
\end{remark}
There is no evidence that a compressibly (S)ENS group is also (S)ENS. 
The point is that the length of the words $g_{d,d,j}$ can be only bounded by $2^{\Oh(p(d))}$ for the polynomial $p$
from \cref{def:compSENS}.

\begin{lemma}\label{lem:compSENS}
	The following properties of SENS also apply to (uniformly) compressibly (S)ENS:
	\begin{itemize}
		\item The property of being (uniformly) compressibly (S)ENS is independent of the
		choice of the standard generating set.
	\item If $Q$ is a finitely generated subquotient of a group $G$ and $Q$ is
	(uniformly) compressibly (S)ENS, then $G$ is also (uniformly) compressibly (S)ENS.
	\item If $G$ is a finite non-solvable group, then $G$ is uniformly compressibly SENS.
	\item If $F_n$ is a finitely generated free group of rank $n \geq 2$, then $F_n$ is uniformly compressibly SENS.
\end{itemize}
\end{lemma}
The proof of \cref{lem:compSENS} repeats verbatim the proofs of Lemmas~\ref{lem:indepgens}, \ref{lem:subquotient}, \ref{lem:finiteSENS}, and \cref{cor:freeSENS}.

Recall that \Ppoly (non-uniform polynomial time) is the class of languages that can be accepted by a family $(C_n)_{n \in \N}$ of boolean circuits
such that for some polynomial $s(n)$ the number of gates of $C_n$ is at most $s(n)$.

\begin{theorem}\label{thm:ppolyhard}
  Let $G$ be compressibly (S)ENS, then $\CompWP(G)$ is hard for \Ppoly under projection reductions.
\end{theorem}
    
\begin{proof} 
As before we only consider the case that $G$ is compressibly SENS.
Let $(C_n)_{n \in \N}$  be a family of polynomial size circuits. Fix an input length $n$ and consider the circuit $C = C_n$. 
 For simplicity, we assume that  all  non-input gates of $C$ are {\sf nand}-gates (this is by no means necessary for the proof, but that way we only need to deal with one type of gates). 
 Input gates are labelled with variables $x_1, \ldots, x_n$ or negated variables $\neg x_1, \ldots, \neg x_n$. This allows to assume that 
 $C$ is synchronous in the sense that
 for every gate $g$ all paths from an input gate to $g$ have the same length. 
 Let $d$ be the depth of $C$. Notice that we do not require $C$ to be a tree (indeed, this would lead to an exponential blow up since $d$ could be as large as the number of gates of $C$).
 
Let $g_{i,j} = g_{d,i,j} \in \Sigma^*$ for $0\leq i \leq d $ and $1 \leq j \leq p(d)$ be from \cref{def:compSENS}. 
We now construct an SLP $\mathcal{G}$ that contains for 
each gate $t$ of $C$ at distance $i$ from the output gate and each $1 \leq j \leq p(d)$ variables $A_{t,j}, A_{t,j}^{-1}$ 
over the terminal alphabet $\set{\langle k, a,b\rangle}{k \in \ccinterval{1}{n}, a,b \in \Sigma}$ of $G$-program instructions
such that for any input word $x \in \{0,1\}^n$ the following holds:
\begin{itemize}
\item  If gate $t$ evaluates to $0$ then
the $G$-programs $\val_{\mathcal{G}}(A_{t,j})$ and $\val_{\mathcal{G}}(A_{t,j}^{-1})$ evaluate to $1$ in $G$.
\item If gate $t$ evaluates to $1$ then
the $G$-programs $\val_{\mathcal{G}}(A_{t,j})$ and $\val_{\mathcal{G}}(A_{t,j}^{-1})$ evaluate to $g_{i,j}$ and $g^{-1}_{i,j}$, respectively, in $G$.
\end{itemize}
This is exactly as \cref{eq:pvdef} in the proof of \cref{thm:NC1hard}. 

For an input gate $t$ labelled with $x_i$ (respectively $\lnot x_i$) this is straightforward using the SLPs for $g_{d,j}$ for the different $j$ and replacing every terminal $a$ in the SLPs by the $G$-program instruction $\langle i, a,1\rangle$ (respectively $\langle i, 1 ,a\rangle$).
For an inner gate $t$ (which is a {\sf nand}-gate by assumption) at distance $i$ from the output gate with inputs from gates $r$ and $s$ (both having distance $i+1$ from the output gate), we set 
\begin{equation}
A_{t,j} \to g_{i,j}  A_{s,\ell}^{-1} A_{r,k}^{-1} A_{s,\ell} A_{r,k}    \label{eq:SLP}
\end{equation}
where $k$ and $\ell$ are as in \eqref{compSENSb} from \cref{def:compSENS} such that $g_{i,j} =	\bigl[ g_{i+1,k},\,   g_{i+1,\ell} \bigr]$. Here we write $g_{i,j}$ as shorthand for the SLP with constant $G$-program instructions evaluating to $g_{i,j}$ as in \cref{rem:allpolySLP}. The correctness follows as in the proof of \cref{thm:NC1hard}.

Thus, we have constructed an SLP of $G$-program instructions. The evaluation of the instructions is the desired projection reduction. 
\end{proof}

\begin{theorem}\label{thm:phard}
  Let $G$ be uniformly compressibly SENS. Then $\CompWP(G)$ is \P-hard under \L reductions.
\end{theorem}

\begin{proof}
  In \cite[A.1.6]{GreenlawHR95} the following variant of the circuit
  value problem is shown to be \P-complete: the input circuit is
  synchronous, monotone (only {\sf and}- and {\sf or}-gates), and alternating~--
  meaning that within one level all gates are of the same type and
  adjacent levels consist of gates of different types.

  Moreover, we can assume that the first layer after the inputs
  consists of {\sf and}-gates and that the output gate is an {\sf or}-gate
  (in particular, there is an even number of non-input layers). By replacing
  each {\sf and}- and {\sf or}-gate by a {\sf nand}-gate, we obtain a synchronous circuit
  computing the same function using only {\sf nand}-gates.

  Hence, we can apply the construction from the proof of
  \cref{thm:ppolyhard} and then evaluate the $G$-program instructions
  in the resulting SLP. The latter can clearly be done in
  \L. Moreover, the construction of the SLP can also be done in \L:
  For input gates the corresponding SLP can be computed in \L by assumption
  \eqref{compSENSu1}. For an inner gate $t$, one needs to compute the rules from
  \cref{eq:SLP}.  The indices $k$ and $\ell$ can be computed in \L by
  assumption \eqref{compSENSu2}. Notice here that $k$ and $\ell$ only need a
  logarithmic number of bits, so we can think of this computation as
  an oracle call with a logspace oracle. Since $\L^\L = \L$, the whole
  computation is \L even though we compute a non-constant number of
  SLP rules.
\end{proof}
\cref{thm:phard} implies that the compressed word for  finite non-solvable groups and finitely generated free groups of rank at least $2$
is \P-complete. This has been also shown in \cite{BeMcPeTh97} (for finite non-solvable groups) and \cite{Loh06siam} (for finitely generated free groups of rank at least $2$).

\begin{remark}\label{rem:wreathcompSENS}
  Consider the situation of \cref{prop:wreathSENS}, but now assume
  that for each $h_d$ there is an SLP of size $p(d)$ for some fixed
  polynomial $p$. Moreover, assume that the SLP for $h_d$ is
  computable from the string $1^d$ in \L.  Then the arguments from the proof of \cref{prop:wreathSENS}
  show that $G$ is uniformly compressibly SENS.
\end{remark}

\begin{corollary}
  The compressed word problem for every weakly branched group is \P-hard. 
\end{corollary}

\begin{proof}
By \cref{rem:wreathcompSENS}, we need to verify that we can compute SLPs for the $h_d$ as in \cref{prop:wreathSENS} in \L. However, this is straightforward because 
the $h_d$ in the proof of \cref{thm:weaklybranched} were defined by iterated application of some endomorphism. This yields the desired SLPs by \cref{lemma-SLP-iterated-morph}.
\end{proof}
In the same way it can be also shown that the compressed word problem for  Thompson's group $F$ is \P-hard.
But in the rest of the paper, we will show that the compressed word problem for  Thompson's group $F$ (as well as a large
class of weakly branched groups) is in fact \PSPACE-complete.

\section{Compressed word problems for wreath products} \label{sec-subsetsum-circuits}

In this section we consider regular wreath products of the form $G \wr \Z$.
The following result was shown in \cite{Lohrey14compressed} (for $G$ non-abelian) and \cite{KonigL18} (for $G$ abelian).

\begin{theorem}[c.f.~\cite{KonigL18,Lohrey14compressed}] \label{thm-brief}
If $G$ is a finitely generated group, then
\begin{itemize}
\item $\CompWP(G\wr\mathbb{Z})$ 
is $\coNP$-hard if $G$ is non-abelian and
\item $\CompWP(G\wr\mathbb{Z})$ belongs to 
$\mathsf{coRP}$ (complement of randomized polynomial time) if $G$ is abelian.
\end{itemize}
\end{theorem}
In this section, we prove the following result, which improves upon the first statement of Theorem~\ref{thm-brief}.

\begin{theorem} \label{thm-main-leaf2}
Let $G$ be a finitely generated non-trivial group.
\begin{itemize}
\item $\CompWP(G\wr\mathbb{Z})$  belongs to $\forall \LEAF(\WP(G))$.
\item $\CompWP(G\wr\mathbb{Z})$  is hard for the class $\forall \LEAF(\WP(G/Z(G)))$.
\end{itemize}
In particular, if $Z(G) = 1$, then $\CompWP(G\wr\mathbb{Z})$ 
is complete for $\forall \LEAF(\WP(G))$.
\end{theorem}
Be aware that in the case that $G$ is abelian, $\WP(G/Z(G))$ is the set of all words over the generators, and so $\forall \LEAF(\WP(G/Z(G)))$ consists of only the universal language. Therefore, for abelian $G$, the hardness statement in \cref{thm-main-leaf2} is trivial. 

The proof of the lower bound uses some of the techniques from the paper 
 \cite{Loh11IC}, where a connection between leaf strings and SLPs was established.
 In Sections~\ref{sec-subsetsum}--\ref{sec-subsetsum-SLP} we will introduce these techniques.
The proof of \cref{thm-main-leaf2} will be given in \cref{sec:leaf-wreath}.

\begin{remark}
	Let $G$ be a finite solvable group with composition series 
$1  = G_0 \leq G_1 \leq \cdots \leq G_r = G$
	meaning that $G_{i-1}$ is normal in $G_i$ and $G_i/G_{i-1}$ is cyclic of prime order $p_i$ for $i\in\oneset{1,\dots, r}$. 
	In this case, \cite[Satz 4.32]{Hertrampf94habil} implies that $\LEAF(\WP(G)) \sse \Mod{p_1} \cdots \Mod{p_r} \P$.
	Thus, using \cref{thm-main-leaf2} we obtain that $\CompWP(G\wr\mathbb{Z})$  belongs to $\forall \Mod{p_1} \cdots \Mod{p_r} \P$.	
	On the other hand, \cite[Theorem 2.2]{h00sigact} states that $\coMod{m} \P \sse \LEAF(\WP(G/Z(G))) $ for $m= \abs{G/Z(G)}$; thus, it follows that  $\CompWP(G\wr\mathbb{Z})$ is hard for $\forall \coMod{m} \P $. More\-over, by \cite[Theorem 2.6]{Hertrampf90tcs}, $\coMod{m} \P = \coMod{k} \P$ for $k= \prod_{p|m}p$ where the product runs over all prime divisors of $m$. As the next examples show there are the extreme cases that  $\CompWP(G\wr\mathbb{Z})$ actually belongs to $\forall \coMod{m} \P $ and also that it is hard for $\forall \Mod{p_1} \cdots \Mod{p_r} \P$ (at least, we give an example for $r=2$):
\begin{itemize}
	\item If $G$ is a finite, non-abelian $p$-group (\ie $p_i = p$ for all $i$), then  
	\[\LEAF(\WP(G)) \sse \Mod{p} \cdots \Mod{p} \P = \Mod{p} \P \sse \LEAF(\WP(G))\] by \cite[Theorem 6.7]{BeigelG92} and likewise 
	$\LEAF(\WP(G/Z(G)))  = \Mod{p} \P$.
	Hence, in this case $\CompWP(G\wr\mathbb{Z})$ is complete for  $\forall \Mod{p} \P$. More generally, for a finite non-abelian nilpotent group $G$ (\ie a direct product of $p$-groups) and  $m = \abs{G/Z(G)}$, it follows that $\CompWP(G\wr\mathbb{Z})$ is complete for $\forall \coMod{m} \P$.
	This is because by \cite[Lemma 2.4]{Hertrampf90tcs} a language $L$ is in $\coMod{m}\P$ if and only if it can be written as an intersection $\bigcap_{p|m}L_p$ for languages $L_p \in \Mod{p}\P$ for $p|m$.
	\item Finally, consider the symmetric group on three elements $S_3$. By \cite[Example 2.5]{h00sigact} we have $ \LEAF(\WP(S_3))= \Mod{3}\Mod{2} \P$ (also written as $\Mod{3}{\oplus} \P$). Since $S_3$ has trivial center, it follows that $\CompWP(S_3\wr\mathbb{Z})$ is complete for $\forall\Mod{3}{\oplus} \P$.
\end{itemize}	
\end{remark}

In the following, we will identify a bit string $\alpha = a_1 a_2 \cdots a_n$ ($a_1, \ldots, a_n \in \{0,1\}$)
with the vector $(a_1, a_2, \ldots, a_n)$. In particular, for another vector $\overline{s} = (s_1, s_2, \ldots, s_n) \in \mathbb{N}^n$
we will write $\alpha \cdot  \overline{s} = \sum_{i=1}^n a_i \cdot s_i$ for the scalar product. Moreover, we 
write $\sum \overline{s}$ for the sum $s_1 + s_2 + \cdots + s_n$.

\subsection{Subsetsum problems}  \label{sec-subsetsum}

A sequence $(s_1,\ldots,s_n)$ of natural numbers is
\emph{super-decreasing} if $s_i > s_{i+1}+ \cdots + s_n$ for all $i \in \ccinterval{1}{n}$. For example, $(s_1,\ldots,s_n)$
with $s_i=2^{n-i}$ is super-decreasing.
An instance of the \emph{subsetsum problem} is 
a tuple $(t,s_1,\ldots,s_k)$ of binary coded natural numbers.
It is a positive instance if there are $a_1, \ldots, a_k \in \{0,1\}$ such that
$t = a_1 s_1 + \cdots + a_k s_k$. 
Subsetsum is a classical $\NP$-complete problem, see e.g. \cite{GarJoh79}.
The \emph{super-decreasing subsetsum} problem is the restriction 
of subsetsum to instances $(t,s_1,\ldots,s_k)$, where $(s_1, \ldots, s_k)$
is super-decreasing. In \cite{KaRu89} it was shown that super-decreasing 
subsetsum is $\Ptime$-complete.\footnote{In fact, \cite{KaRu89} deals with
the {\em super-increasing} subsetsum problem. But this is only a nonessential
detail. For our purpose, super-decreasing sequences are more convenient.} 
We need a slightly generalized version of the construction showing $\Ptime$-hardness 
that we discuss in Section~\ref{sec-circuits-subsetsum}.

\subsection{From boolean circuits to super-decreasing subsetsum} \label{sec-circuits-subsetsum} 

For this section, we have to fix some more details on boolean circuits.
Let us consider a boolean circuit $C$ with input gates $x_1, \ldots, x_m$ 
and output gates $y_0, \ldots, y_{n-1}$.\footnote{It will be convenient for us to number the input gates from $1$ 
and the output gates from $0$.}
We view $C$ as a directed acyclic graph with multi-edges
(there can be two edges between two nodes); the nodes are the gates of the circuit.
The number of incoming edges of a gate is called its {\em fan-in} and the 
number of outgoing edges is the {\em fan-out}.
Every input gate $x_i$ has fan-in zero and every output gate $y_i$ has fan-out zero.
Besides the input gates there are two more gates $c_0$ and $c_1$ of fan-in zero, where
$c_i$ carries the constant truth value $i \in \{0,1\}$. Besides $x_1, \ldots, x_m, c_0, c_1$
every other gate has fan-in two and computes the \nand\ of its two input gates.
Moreover, we assume that every output gate $y_i$ is a {\sf nand}-gate.
For a bit string $\alpha = b_1 \cdots b_m$ ($b_1, \ldots, b_m \in \{0,1\}$) and $0 \leq i \leq n-1$
we denote with $C(\alpha)_i$ the value of the output gate $y_i$  when every input gate
$x_j$ ($1 \leq j \leq m$) is set to $b_j$. Thus, $C$ defines a map $\{0,1\}^m\to\{0,1\}^n$. 

We assume now that $C$ is a boolean circuit as above with the 
 following additional property that will be satisfied later: For all input bit strings $\alpha \in \{0,1\}^m$ there is exactly
one $i \in \cointerval{0}{n}$ such that $C(\alpha)_i=1$. Using a refinement of the construction from \cite{KaRu89}
we compute in \L $q_{0}, \ldots, q_{n-1} \in \mathbb{N}$ and two super-decreasing sequences 
$\overline{r} = (r_1, \ldots r_m)$ and $\overline{s} = (s_1, \ldots, s_k)$ for some $k$ (all numbers are represented in binary notation) with the following properties:
\begin{itemize}
\item The $r_1, \ldots, r_m$ are pairwise distinct powers of $4$. 
\item For all $0 \leq i \leq n-1$ and all $\alpha \in \{0,1\}^m$:
$C(\alpha)_i = 1$  if and only if there exists
$\delta \in \{0,1\}^k$ such that $\delta \cdot \overline{s} = q_i + \alpha \cdot \overline{r}$. 
\end{itemize}
Let us first add for every input gate $x_i$ two new {\sf nand}-gates
${\bar x}_i$ and ${\bar{\bar x}}_i$, where ${\bar{\bar x}}_i$ has the same outgoing edges as $x_i$. Moreover we remove the old outgoing edges of $x_i$
and replace them by the edges $(x_i, {\bar x}_i)$, $(c_1, {\bar x}_i)$ and two edges from ${\bar x}_i$ to ${\bar{\bar x}}_i$.
This has the effect that every input gate $x_i$ has a unique outgoing edge. Clearly, the new circuit computes the same
boolean function (basically, we introduce two negation gates for every input gate).
Let $g_1, \ldots, g_p$ be the {\sf nand}-gates of the circuit enumerated in reverse topological
order, i.e., if there is an edge from gate $g_i$ to gate $g_j$ then $i>j$. We denote the two edges entering gate
$g_i$ with $e_{2i+n-2}$ and $e_{2i+n-1}$. Moreover, we write $e_i$ ($0 \leq i \leq n-1$) for an imaginary edge
that leaves the output gate $y_{i}$ and whose target gate is unspecified. Thus, the edges of the circuit are 
$e_0,\ldots, e_{2p+n-1}$.
We now define the natural numbers $q_{0}, \ldots, q_{n-1}, r_1, \ldots r_m,s_1, \ldots, s_k$ with $k = 3p$:
\begin{itemize}
\item
Let $I = \{ j \mid e_j \text{ is an outgoing edge of the constant gate  $c_1$ or a {\sf nand}-gate} \}$.
For $0 \leq i \leq n-1$ we define the number $q_{i}$ as 
$$
q_{i} = \sum_{j \in I \setminus \{i\}} 4^j .
$$
Recall that $e_i$ is the unique outgoing edge of the output gate $y_i$.
\item If $e_j$ is the unique outgoing edge of the input gate $x_i$ then we set $r_i = 4^j$. We can choose the reverse
topological sorting of the {\sf nand}-gates in such a way that $r_1 > r_2 > \cdots > r_m$ (we only have to ensure that the target
gates $\overline{x}_1, \ldots, \overline{x}_m$ of the input gates appear in the order $\overline{x}_m, \ldots, \overline{x}_1$ in the 
reverse topological sorting of the {\sf nand}-gates).
\item To define the numbers $s_1, \ldots, s_k$ we first define for every {\sf nand}-gate $g_i$
three numbers $t_{3i}$, $t_{3i-1}$ and $t_{3i-2}$ as follows, where 
$I_i = \{ j \mid e_j \text{ is an outgoing edge of gate } g_i \}$:
\begin{eqnarray*}
t_{3i}       &=& 4^{2i+n-1} + 4^{2i+n-2} + \sum_{j \in I_i} 4^j \\
t_{3i-1} &=& 4^{2i+n-1} - 4^{2i+n-2} = 3 \cdot 4^{2i+n-2} \\
t_{3i-2}    &=& 4^{2i+n-2}   
\end{eqnarray*}
Then, the tuple $(s_1,\ldots, s_k)$ is $(t_{3p}, t_{3p-1}, t_{3p-2}, \ldots, t_3, t_2, t_1)$, which is indeed
super-decreasing (see also \cite{KaRu89}). In fact, we have $s_i - (s_{i+1} + \cdots + s_k) \geq 4^{n-1}$ for all $i \in \ccinterval{1}{k}$.
To see this, note that the sets $I_{i+1},\ldots, I_k$ are pairwise disjoint.  This implies that the $n-1$ low-order digits in the base-4 expansion of 
$s_{i+1} + \cdots + s_k$ are zero or one.

\end{itemize}
In order to understand this construction, one should think of the edges of the circuit carrying truth values.
Recall that there are $2p+n$ edges in the circuit (including the imaginary outgoing edges of the output gates
$y_0, \ldots, y_{n-1}$). 
A number in base-4 representation with $2p+n$ digits that are either $0$ or $1$ represents a truth assignment to the $2p+n$ edges,
where a $1$-digit represents the truth value $1$ and a $0$-digit represents the truth value $0$.
Consider  an input string $\alpha = b_1 \cdots b_m \in \{0,1\}^m$  and 
consider an output gate $y_i$, $i \in \cointerval{0}{n}$. Then the number
$N := 4^i + q_{i} + b_1 r_1 + \cdots + b_m r_m$ encodes the truth assignment for the circuit edges, where:
\begin{itemize}
\item all outgoing edges of the constant gate $c_1$ carry the truth value $1$, 
\item all outgoing edges of the constant gate $c_0$ carry the truth value $0$, 
\item the unique outgoing edge of an input gate $x_i$ carries the truth value $b_i$,
\item all outgoing edges of {\sf nand}-gates 
carry the truth value $1$.
\end{itemize}
We have to show that $C(\alpha)_i = 1$  if and only if there exists
$\delta \in \{0,1\}^k$ such that $\delta \cdot \overline{s} = N-4^i$.
For this we apply the canonical algorithm for super-decreasing subsetsum with input $(N,s_1,\ldots,s_k)$.
This algorithm initializes a counter $A$ to $N$ and then goes over the sequence $s_1,\ldots,s_k$ in that order.
In the $j$-th step ($1 \leq j \leq k$) we set $A$ to $A-s_j$ if $A \geq s_j$. If $A < s_j$ then we do not modify $A$.
After that we proceed with $s_{j+1}$. The point is that 
this process simulates the evaluation of the circuit on the input values $b_1, \ldots, b_m$. Thereby
the {\sf nand}-gates are evaluated in the topological order $g_p, g_{p-1}, \ldots, g_1$. Assume that $g_j$ is the gate that we 
want to evaluate next. In the above algorithm for super-decreasing subsetsum the evaluation of $g_j$ is simulated 
by the three numbers $t_{3j}$, $t_{3j-1}$, and $t_{3j-2}$. At the point where the algorithm checks whether $t_{3j}$ can 
be subtracted from $A$, the base-4 digits at positions $2j+n,\ldots, 2p+n-1$ in the counter value $A$ have been already set to zero. If the digits at the next two high-order 
positions $2j+n-1$ and $2j+n-2$ are still $1$ (i.e., the input edges $e_{2j+n-2}$ and $e_{2j+n-1}$ for gate
$g_j$ carry the truth value $1$), then we can subtract $t_{3j}$ from $A$. Thereby we subtract all powers
$4^{2j+n-1}$, $4^{2j+n-2}$ and
$4^h$, where $e_h$ is an outgoing edge for gate $g_j$. Since gate $g_j$ evaluates to zero (both input edges carry $1$), this subtraction
correctly simulates the evaluation of gate $g_j$: all outgoing edges $e_h$ of $g_j$ (that were initially set to the truth value $1$) are set to the 
truth value $0$. On the other hand, if one of the two digits at
positions $2j+n-1$ and $2j+n-2$ in $A$ is $0$ (which means that gate $g_j$ evaluates to $1$), then we 
cannot subtract $t_{3j}$ from $A$. If both digits at positions $2j+n-1$ and $2j+n-2$ in $A$ are $0$, then
also $t_{3j-1}$ and $t_{3j-2}$ cannot be subtracted. On the other hand, if exactly one of the two digits
at positions $2j+n-1$ and $2j+n-2$ is $1$, then with
$t_{3j-1}$ and $t_{3j-2}$ we can set  these two digits to $0$ (thereby digits at positions $< 2j+n-2$ are not modified).

Assume now that $y_j$ ($j \in \cointerval{0}{n}$) is the unique output gate that evaluates to $1$, i.e., all
output gates $y_{j'}$ with $j' \neq j$ evaluate to zero. Then after processing all weights
$s_1,\ldots,s_k$ we have $A = 4^j$ (we will never subtract $4^j$). We have shown that there exists
$\delta \in \{0,1\}^k$ such that $\delta \cdot \overline{s} + 4^j = N$. Hence, if $i=j$ (i.e., $C(\alpha)_i = 1$)
then  $\delta \cdot \overline{s} = N-4^i$. Now assume that $i \neq j$. In order to get a contradiction assume 
that there is $\delta' \in \{0,1\}^k$ such that $\delta' \cdot \overline{s} + 4^i = N$. We have $\delta \neq \delta'$ and 
$\delta \cdot \overline{s} + 4^j  = \delta' \cdot \overline{s} + 4^i$, i.e, $\delta \cdot \overline{s} - \delta' \cdot \overline{s} = 4^i - 4^j$.
Since $i,j \in \cointerval{0}{n}$ we get $|\delta \cdot \overline{s} - \delta' \cdot \overline{s}| < 4^{n-1}$. But 
$s_i - (s_{i+1} + \cdots s_k) \geq 4^{n-1}$ for all $i \in \ccinterval{1}{k}$ implies that 
$|\delta \cdot \overline{s} - \delta' \cdot \overline{s}| \ge 4^{n-1}$.

\subsection{From super-decreasing subsetsum to straight-line programs} \label{sec-subsetsum-SLP}

In \cite{LiLo06} a super-decreasing sequence $\overline{t} = (t_1,\ldots,t_k)$ of natural numbers is encoded by the string 
$S(\overline{t}) \in \{0,1\}^*$ of length $\sum \overline{t} +1$
such that for all $0 \leq p \leq \sum \overline{t}$:
\begin{equation}\label{main string}
S(\overline{t})[p] = \begin{cases} 1 & \text{if $p = \alpha \cdot \overline{t}$ for some $\alpha  \in \{0,1\}^k$}, \\
0 & \text{otherwise}.
\end{cases}
\end{equation}
Note that in the first case, $\alpha$ is unique.
Since $\overline{t}$ is a super-decreasing sequence, 
the number of $1$'s in the string $S(\overline{t})$ is $2^k$.
Also note that $S(\overline{t})$ starts and ends with $1$.
In \cite{LiLo06} it was shown that from a super-decreasing sequence $\overline{t}$
of binary encoded numbers one can construct in \L an SLP for 
the word $S(\overline{t})$.

\subsection{Proof of Theorem~\ref{thm-main-leaf2}} \label{sec:leaf-wreath}

Let us fix a regular wreath product of the form $G \wr \mathbb{Z}$ for a finitely generated
group $G$. Such groups are also known as generalized lamplighter groups (the lamplighter group arises
for $G = \mathbb{Z}_2$). Throughout this section, we fix a set of standard generators $\Sigma$ for $G$ and let $\tau = 1$ be the generator for $\mathbb{Z}$. 
Then $\Sigma \cup \{\tau,\tau^{-1}\}$ is a standard generating set for the wreath product $G \wr \mathbb{Z}$.
In $G \wr \mathbb{Z}$ the $G$-generator $a \in \Sigma$ represents the mapping $f_{a} \in G^{(\Z)}$ with $f_{a}(0) = a$ and 
$f_{a}(z) = 1$ for $z \neq 0$. 
For a word $w \in (\Sigma \cup \{\tau,\tau^{-1}\})^*$ we define $\eta(w) :=  |w|_\tau-|w|_{\tau^{-1}}$. Thus, the element of $G \wr \mathbb{Z}$ represented by
$w$ is of the form $f  \tau^{\eta(w)}$ for some $f \in G^{(\Z)}$.
Recall the definition of the left action of $\Z$ on $G^{(\Z)}$ from \cref{sec-wreath} (where we take $H = Y = \Z$).
For better readability, we write $c \circ f$ for ${}^c \! f$ ($c \in \Z$, $f \in G^{(\Z)}$).
Hence, we have $(c \circ f)(z) = f(z+c)$. If one thinks of $f$ as a bi-infinite word over the alphabet $G$,
then $c \circ f$ is the same word but shifted by $-c$.

The following intuition might be helpful: Consider a word $w \in ( \Sigma \cup \{\tau, \tau^{-1} \})^*$. In 
$G \wr \mathbb{Z}$ we can simplify $w$ to a word of the form
$\tau^{z_0} a_{1} \tau^{z_1} a_{2} \cdots \tau^{z_{k-1}} a_{k} \tau^{z_k}$ (with $z_j \in \Z$, $a_j \in \Sigma$),
which in $G \wr \mathbb{Z}$ can be rewritten
as
\[
\tau^{z_0} a_{1} \tau^{z_1} a_{2} \cdots \tau^{z_{k-1}} a_{k} \tau^{z_k} = \big(\prod_{j=1}^k \tau^{z_0 + \cdots+ z_{j-1}} a_{j} \tau^{-(z_0 + \cdots +z_{j-1})}\big) \; \tau^{z_0 + \cdots +z_k}.
\]  
Hence, the word $w$ represents the group element 
\[
\big(\prod_{j=1}^k (z_0 + \cdots+ z_{j-1}) \circ f_{a_j} \big)\, \tau^{z_0 + \cdots +z_k} .
\]
This gives the following intuition for evaluating $\tau^{z_0} a_{1} \tau^{z_1} a_{2} \cdots \tau^{z_{k-1}} a_{k} \tau^{z_k}$:
In the beginning, every $\Z$-position carries the $G$-value $1$. First, 
go to the $\Z$-position $-z_0$ and multiply the $G$-element at this position with $a_1$ (on the right), then 
go to the $\Z$-position $-z_0-z_1$ and multiply the $G$-element at this position with $a_2$, and so on.  

\begin{proof}[Proof of Theorem~\ref{thm-main-leaf2}]
The easy part is to show that the compressed word problem for $G \wr \mathbb{Z}$ belongs to
$\forall\LEAF(\WP(G))$. In the following, we make use
of the statements from Lemma~\ref{lemma-basic}.
Let $\mathcal{G}$ be an SLP over the alphabet $\Sigma \cup \{\tau,\tau^{-1}\}$ and let $f \tau^{\eta(\val(\mathcal{G}))} \in G \wr \mathbb{Z}$
be the group element represented by $\val(\mathcal{G})$. By Lemma~\ref{lemma-basic} we can compute $\eta(\val(\mathcal{G}))$ in polynomial time. 
If $\eta(\val(\mathcal{G})) \neq 0$ then the Turing-machine rejects by printing a non-trivial generator of $G$ (here we need the assumption that
$G$ is non-trivial).
So, let us assume that $\eta(\val(\mathcal{G}))=0$. 
We can also compute in polynomial time two integers $b,c \in \mathbb{Z}$ such that $\supp(f) \subseteq \ccinterval{b}{c}$. We can
take for instance $b = -|\val(\mathcal{G})|$ and $c = |\val(\mathcal{G})|$. It suffices to check whether for all
$ x \in \ccinterval{b}{c}$ we have $ f(x)=1$. For this, the Turing-machine branches universally to all binary coded integers $x \in \ccinterval{b}{c}$
(this yields the $\forall$-part in $\forall \LEAF(\WP(G))$). Consider a specific branch that leads to the integer 
$x \in \ccinterval{b}{c}$. From $x$ and the input SLP $\mathcal{G}$ the Turing-machine then produces a leaf string over the standard
generating set $\Sigma$ of $G$ such that this leaf string represents the group element $f(x) \in G$. For this, the machine
branches to all positions $p \in \cointerval{0}{|\val(\mathcal{G})|}$ (if $p<q<|\val(\mathcal{G})|$ then the branch for $p$ is to the left 
of the branch for $q$). For a specific position $p$, the machine computes in polynomial time the symbol
$a = \val(\mathcal{G})[p]$. If $a$ is $\tau$ or $\tau^{-1}$ then the machine prints $1 \in \Sigma$. On the other hand, if 
$a \in \Sigma$ then the machine computes in polynomial time $d = \eta(\val(\mathcal{G})[:p])$. This is possible by first
computing an SLP for the prefix $\val(\mathcal{G})[:p]$. If $d=-x$ then the machine prints the symbol $a$, otherwise
the machine prints the trivial generator $1$. It is easy to observe that the leaf string produced in this way represents the 
group element $f(x)$.

We now show the hardness statement from Theorem~\ref{thm-main-leaf2}.
By Lemma~\ref{lemma-JMT} it suffices to show that $\CompWP(G \wr \mathbb{Z})$ 
is hard for $\forall\bLEAF(\WP(G/Z(G)))$ with respect to \L-reductions. Let $a_0, \ldots, a_{n-1}$ be an arbitrary enumeration of the standard generators in $\Sigma$.
Fix a language $L \in \forall\bLEAF(\WP(G/Z(G)))$. From the definition of the class $\forall\bLEAF(\WP(G/Z(G)))$ it follows that
there exist two polynomials $p_1$ and $p_2$ and a balanced polynomial time NTM $M$ running in time $p_1+p_2$
that outputs a symbol from $\Sigma$ after termination and such that the following holds: Consider an input word
$z$ and let $T(z)$ be the corresponding computation tree of $M$. 
Let $m_1 = p_1(|z|)$, $m_2 = p_2(|z|)$, and $m = m_1+m_2$. Note that the nodes of $T(z)$ are the bit strings of length at most $m$.
For every leaf $\alpha \in \{0,1\}^m$ let us denote with $\lambda(\alpha)$ the symbol from $\Sigma$ that $M$ prints 
when reaching the leaf $\alpha$. Then $z \in L$ if and only if for all
$\beta \in \{0,1\}^{m_1}$ the string 
\begin{equation} \label{eq-alpha-u}
\lambda_\beta := \prod_{\gamma \in \{0,1\}^{m_2}}\lambda(\beta\gamma)
\end{equation}
represents a group element from the center $Z(G)$. Here (and in the following), the product in the right-hand side of  \eqref{eq-alpha-u} goes over all bit strings of length $m_2$
in lexicographic order.
Our construction consists of five steps:

\medskip
\noindent
{\em Step 1.}
Note that given a bit string $\alpha \in \{0,1\}^m$, we can 
compute in polynomial time the symbol $\lambda(\alpha) \in \Sigma$ by following the computation path specified by $\alpha$. 
Using the classical Cook-Levin construction 
(see e.g. \cite{AroBar09}), we can compute from the input $z$ and $a \in \Sigma$  in \L
a boolean circuit $C_{z,a}$ with $m$ input gates $x_1, \ldots, x_m$ and a single output gate $y_0$
such that for all $\alpha \in \{0,1\}^m$:
$C_{z,a}(\alpha)_0 = 1$ if and only if $\lambda(\alpha) = a$. 
By taking the disjoint union of these circuits and merging the input gates,
we can build a single circuit $C_z$ with $m$ input gates $x_1, \ldots, x_m$
and $n = |\Sigma|$ output gates $y_0, \ldots, y_{n-1}$. For every $\alpha \in \{0,1\}^m$ and every
$0 \leq i \leq n-1$ the following holds: $C_z(\alpha)_i = 1$ if and only if $\lambda(\alpha) = a_i$.

\medskip
\noindent
{\em Step 2.}
Using the construction from Section~\ref{sec-circuits-subsetsum} we can compute from the  
circuit $C_z$ in \L numbers $q_{0}, \ldots, q_{n-1} \in \mathbb{N}$ and two super-decreasing sequences 
$\overline{r} = (r_1, \ldots, r_m)$ and $\overline{s} =(s_1, \ldots, s_k)$ with the following properties:
\begin{itemize}
\item The $r_1, \ldots, r_m$ are pairwise distinct powers of $4$. 
\item For all $0 \leq i \leq n-1$ and all $\alpha \in \{0,1\}^m$ we have:
$\lambda(\alpha) = a_i$ if and only if
$C_z(\alpha)_i = 1$  if and only if there is $\delta \in \{0,1\}^k$ such that $\delta \cdot \overline{s} = q_{i} + \alpha \cdot \overline{r}$. 
\end{itemize}
Note that for all $\alpha \in \{0,1\}^m$ there is a unique $i$ such that $C_z(\alpha)_i = 1$.
Hence, for all  $\alpha \in \{0,1\}^m$ there is a unique $i$ such that $q_{i} + \alpha \cdot \overline{r}$ is of the form 
$\delta \cdot \overline{s}$ for some $\delta \in \{0,1\}^k$. For this unique $i$ we have $\lambda(\alpha) = a_i$.

We split the super-decreasing sequence $\overline{r} = (r_1, \ldots, r_m)$ into the two
sequences $\overline{r}_1 = (r_1, \ldots, r_{m_1})$ and $\overline{r}_2 = (r_{m_1+1}, \ldots, r_m)$.
For the following consideration, we need the following numbers:
\begin{eqnarray}
\ell & = & \textstyle \max\big\{ \sum \overline{r}_1 + \max\{q_0, \ldots, q_{n-1}\}+1, \; \sum\overline{s} - \sum\overline{r}_2 - \min \{q_0, \ldots, q_{n-1}\}+1\big\} \label{eq:l} \\
\pi &= & \textstyle \ell + \sum \overline{r}_2 
\end{eqnarray}
The binary codings of these numbers can be computed in \L (since iterated addition, $\max$, and $\min$ can be computed in \L).
The precise value of $\ell$ will be only relevant at the end of step 4.

\medskip
\noindent
{\em Step 3.} 
By the result from \cite{LiLo06} (see Section~\ref{sec-subsetsum-SLP})
we can construct in \L from the three super-decreasing sequences
$\overline{r}_1, \overline{r}_2$ and $\overline{s}$ three SLPs $\mathcal{G}_1$,   $\mathcal{G}_2$ and $\mathcal{H}$ over the alphabet
$\{0,1\}$ such that $\val(\mathcal{G}_1) = S(\overline{r}_1)$, $\val(\mathcal{G}_2) = S(\overline{r}_2)$
and $\val(\mathcal{H}) = S(\overline{s})$ (see (\ref{main string})).
For all positions $p \geq 0$ (in the suitable range) we have:
\begin{eqnarray*}
\val(\mathcal{G}_1)[p] = 1  & \Longleftrightarrow & 
\exists \beta \in  \{0,1\}^{m_1}: 
p = \beta \cdot \overline{r}_1 \\ 
\val(\mathcal{G}_2)[p] = 1  & \Longleftrightarrow & 
\exists  \gamma \in  \{0,1\}^{m_2}: 
p = \gamma \cdot \overline{r}_2 \\ 
\val(\mathcal{H})[p] = 1  & \Longleftrightarrow & 
\exists \delta \in  \{0,1\}^k: 
p = \delta \cdot \overline{s}
\end{eqnarray*}
Note that $|\val(\mathcal{G}_1)| = \sum\overline{r}_1+1$, $|\val(\mathcal{G}_2)| = \sum\overline{r}_2+1$, and
$|\val(\mathcal{H})| = \sum\overline{s}+1$.

\medskip
\noindent
{\em Step 4.}
We build in \L for every $i \in \ccinterval{0}{n-1}$ an
SLP $\mathcal{H}_i$ 
from the SLP $\mathcal{H}$ by replacing in every right-hand side of $\mathcal{H}$ 
every occurrence of $0$  by $\tau^{-1}$ and
every occurrence of $1$  by $a_i \tau^{-1}$.
Let $T_i$ be the start variable of $\mathcal{H}_i$, let $S_1$ be the start variable of $\mathcal{G}_1$, and let
$S_2$ be the start variable of $\mathcal{G}_2$.
We can assume that the variable sets of the SLPs $\mathcal{G}_1, \mathcal{G}_2, \mathcal{H}_0, \ldots, \mathcal{H}_{n-1}$
are pairwise disjoint.  We next combine these SLPs into a single SLP $\mathcal{I}$.
The variables of $\mathcal{I}$ are the variables of 
the SLPs $\mathcal{G}_1, \mathcal{G}_2,  \mathcal{H}_0, \ldots, \mathcal{H}_{n-1}$ plus a fresh variable $S$ which is the start
variable of $\mathcal{I}$. The right-hand sides for the variables are defined below. In the right-hand sides 
we write powers $\tau^p$ for integers $p$
whose binary codings can be computed in \L. Such powers can be produced by small subSLPs that can be constructed in \L too.
\begin{itemize}
\item In all right-hand sides of $\mathcal{G}_1$  and $\mathcal{G}_2$ 
we replace all occurrences of the terminal symbol $0$ by the $\mathbb{Z}$-generator $\tau$. 
\item We replace every
occurrence of the terminal symbol $1$ in a right-hand side of $\mathcal{G}_1$ by $S_2 \tau^\ell$,
where $\ell$ is from \eqref{eq:l}.
\item We  replace every
occurrence of the terminal symbol $1$ in a right-hand side of $\mathcal{G}_2$ by
$\sigma \tau$, where
\begin{equation} \label{eq-sigma}
\sigma = \tau^{q_{0}} T_0 \tau^{h-q_{0}} \tau^{q_{1}} T_1 \tau^{h-q_{1}} \cdots \tau^{q_{n-1}} T_{n-1} \tau^{h-q_{n-1}}
\end{equation}
and $h = \sum \overline{s}+1$ is the length of the word
$\val(\mathcal{H})$ (which is $-\eta(\val_{\mathcal{I}}(T_i))$ for every  $i \in\cointerval{0}{n}$).
Note that $\eta(\val_{\mathcal{I}}(\sigma)) = 0$.
\item Finally, 
the right-hand side of the start variable $S$ is  $S_1 \tau^{-d}$ where $d := \sum \overline{r}_1 + 1 + 2^{m_1} \cdot \pi$.
(note that $d = \eta(\val_{\mathcal{I}}(S_1))$).
\end{itemize}
Before we explain this construction, let us first introduce some notations.
\begin{itemize}
\item Let $u := \val_{\mathcal{I}}(S_2)$. We have $\eta(u) = |\val(\mathcal{G}_2)|$. Hence,
the group element represented by $u$ can be written as
$f_u \tau^{|\val(\mathcal{G}_2)|}$ for a mapping $f_u \in G^{(\Z)}$. 
\item Let $v := \val_{\mathcal{I}}(\sigma)$ where $\sigma$ is from \eqref{eq-sigma}. Note that $\eta(v)=0$.
Hence, the group element represented by $v$ is a mapping $f_v \in G^{(\Z)}$. Its support is a subset of 
the interval from position $-\max\{ q_0, \ldots, q_{n-1}\}$ 
to position $\sum \overline{s}-\min\{ q_0, \ldots, q_{n-1}\}$.
\item For $\beta \in \{0,1\}^{m_1}$ let
$\mathsf{bin}(\beta)$ be the number represented by $\beta$ in binary notation (thus, $\mathsf{bin}(0^{m_1}) = 0$, $\mathsf{bin}(0^{m_1-1}1) = 1$, \ldots, 
$\mathsf{bin}(1^{m_1}) = 2^{m_1}-1$). Moreover, let 
\[ p_\beta := -\mathsf{bin}(\beta) \cdot \pi.\]
\end{itemize}
First, note that $\eta(\val(\mathcal{I})) = 0$. This is due to the factor $\tau^{-d}$ in the right-hand side of the start
variable $S$ of $\mathcal{I}$. Hence, the group element represented by $\val(\mathcal{I})$ is a mapping $f \in G^{(\Z)}$.
The crucial claim is the following:

\medskip
\noindent
\begin{claim*}
	For every $\beta \in \{0,1\}^{m_1}$, $f(p_{\beta})$ is the group element represented by the leaf string
	$\lambda_{\beta}$ from \eqref{eq-alpha-u}.
\end{claim*}

\medskip
\noindent
{\em Proof of the claim.} In the following, we compute in the restricted direct product $G^{(\mathbb{Z})}$.
Recall that the multiplication in this group is defined by the pointwise multiplication of mappings.

Since we replaced in $\mathcal{G}_1$ every $1$ in a right-hand side by 
$S_2 \tau^\ell$, which produces $u \tau^\ell$ in $\mathcal{I}$ (which evaluates to $f_u \tau^{\pi+1}$)
the mapping $f$ is a product (in the restricted direct product $G^{(\mathbb{Z})}$) of shifted copies of $f_u$.
More precisely, for every $\beta' \in \{0,1\}^{m_1}$ we get the shifted copy 
\begin{equation} \label{eq-shift-dist}
 \big(\beta' \cdot \overline{r}_1 + \mathsf{bin}(\beta') \cdot \pi \big) \circ f_u 
\end{equation}
of $f_u$. The shift distance $\beta' \cdot \overline{r}_1 + \mathsf{bin}(\beta') \cdot \pi$ 
can be explained as follows: The $1$ in $\val(\mathcal{G}_1)$ that corresponds to $\beta' \in \{0,1\}^{m_1}$
occurs at position $\beta' \cdot \overline{r}_1$ (the first position is $0$) and to the left of this position we find $\mathsf{bin}(\beta')$ many
$1$'s and $\beta' \cdot \overline{r}_1 - \mathsf{bin}(\beta')$ many $0$'s in $\val(\mathcal{G}_1)$.
Moreover, every $0$ in $\val(\mathcal{G}_1)$ was replaced
by $\tau$ (shift by $1$) and every $1$ in $\val(\mathcal{G}_1)$ was replaced by $u \tau^\ell$ 
(shift by $\ell + |\val(\mathcal{G}_2)| = \pi + 1$). Hence, the total shift distance
is indeed  \eqref{eq-shift-dist}. Also note that
if $\beta' \in \{0,1\}^{m_1}$ is lexicographically smaller than $\beta'' \in \{0,1\}^{m_1}$ then 
$\beta' \cdot \overline{r}_1 < \beta'' \cdot \overline{r}_1$. 
This implies that 
\begin{eqnarray*}
f = \prod_{\beta' \in \{0,1\}^{m_1}}  \big(\beta' \cdot \overline{r}_1 + \mathsf{bin}(\beta') \cdot \pi \big) \circ f_u  = 
\prod_{\beta' \in \{0,1\}^{m_1}}  \big(\beta' \cdot \overline{r}_1 -p_{\beta'}  \big) \circ f_u .
\end{eqnarray*}
Let us now compute the mapping $f_u$. Recall that we replaced in $\mathcal{G}_2$ every occurrence of $1$ by  $\sigma \tau$, where $\sigma$ is from
\eqref{eq-sigma} and derives to $v$. The $1$'s in $\val(\mathcal{G}_2)$ occur at positions of the form $\gamma \cdot \overline{r}_2$
for $\gamma \in \{0,1\}^{m_2}$ and if $\gamma \in \{0,1\}^{m_2}$ is lexicographically smaller than $\gamma' \in \{0,1\}^{m_2}$ then 
$\gamma \cdot \overline{r}_2 < \gamma' \cdot \overline{r}_2$.
We therefore get
\[
f_u = \prod_{\gamma \in \{0,1\}^{m_2}}  (\gamma \cdot \overline{r}_2) \circ f_v . 
\]
We obtain 
\begin{eqnarray*}
f &=&  \prod_{\beta' \in \{0,1\}^{m_1}}  \big(\beta' \cdot \overline{r}_1 - p_{\beta'}\big) \circ f_u  \\
& = & \prod_{\beta' \in \{0,1\}^{m_1}}   \big(\beta' \cdot \overline{r}_1 - p_{\beta'} \big) \circ \prod_{\gamma \in \{0,1\}^{m_2}} ( \gamma \cdot \overline{r}_2 \circ f_v)   \\
& = & \prod_{\beta' \in \{0,1\}^{m_1}}  \prod_{\gamma \in \{0,1\}^{m_2}}   \big(\beta' \cdot \overline{r}_1 + \gamma \cdot \overline{r}_2 - p_{\beta'} \big) \circ f_v
\end{eqnarray*}
and hence
\begin{eqnarray*}
f(p_\beta) 
&=& \prod_{\beta' \in \{0,1\}^{m_1}}  \prod_{\gamma \in \{0,1\}^{m_2}} f_v(p_{\beta} - p_{\beta'} + \beta' \cdot \overline{r}_1 + \gamma \cdot \overline{r}_2) .
\end{eqnarray*}
We claim that for all $\beta \neq \beta'$ and all $\gamma\in \{0,1\}^{m_2}$ we have
\begin{equation} \label{eq-f_v}
f_v(p_{\beta} - p_{\beta'} + \beta' \cdot \overline{r}_1 + \gamma \cdot \overline{r}_2) = 1 .
\end{equation}
Let us postpone the proof of this for a moment. From \eqref{eq-f_v} we get
\[
f(p_\beta) = \prod_{\gamma \in \{0,1\}^{m_2}} f_v(\beta \cdot \overline{r}_1 + \gamma \cdot \overline{r}_2) .
\]
Consider a specific $\gamma \in \{0,1\}^{m_2}$ and let $\alpha = \beta \gamma$ and
$p = \beta \cdot \overline{r}_1 + \gamma \cdot \overline{r}_2 = \alpha \cdot \overline{r}$.
From the definition of $v = \val_{\mathcal{I}}(\sigma)$ it follows that for all $x \in \Z$, $f_v(x)$ is a product
of those group generators $a_i$ such that $x =  -q_i + \delta \cdot \overline{s}$ for some $\delta \in \{0,1\}^k$.
For the position $p$ this means that $q_i + \alpha \cdot \overline{r} = \delta \cdot \overline{s}$. By our previous remarks,
there is a unique such $i \in\cointerval{0}{n}$ and for this $i$ we have $\lambda(\alpha) = a_i$.
Hence, we obtain $f_v(p) = \lambda(\alpha)= \lambda(\beta\gamma)$ and thus
\[
f(p_\beta) = \prod_{\gamma \in \{0,1\}^{m_2}} \lambda(\beta\gamma) = \lambda_\beta .
\]
It remains to show \eqref{eq-f_v}. To get this identity, we need the precise value of $\ell$ from \eqref{eq:l}
(so far, the  value of $\ell$ was not relevant).
Assume now that $\beta \neq \beta'$, which implies
\[
|p_{\beta} - p_{\beta'}| \geq \pi =  \ell +  \sum \overline{r}_2 .
\]
Hence, we  either have 
\begin{eqnarray*}
p_{\beta} - p_{\beta'} +\beta' \cdot \overline{r}_1 + \gamma \cdot \overline{r}_2 & \geq & \ell +  \sum \overline{r}_2 + \beta' \cdot \overline{r}_1 + \gamma \cdot \overline{r}_2 \\
& \geq & \ell + \sum \overline{r}_2 \\
& > & \sum \overline{s} -  \min\{ q_0, \ldots, q_{n-1}\}
\end{eqnarray*}
or
\begin{eqnarray*}
p_{\beta} - p_{\beta'} +\beta' \cdot \overline{r}_1 + \gamma \cdot \overline{r}_2 & \leq & -\ell - \sum \overline{r}_2 + \beta' \cdot \overline{r}_1 + \gamma \cdot \overline{r}_2 \\
& \leq & -\ell + \sum \overline{r}_1  \\
& < &  - \max\{ q_0, \ldots, q_{n-1}\},
\end{eqnarray*}
where the strict inequalities follow from our choice of $\ell$.
Recall that the support of the mapping $f_v$ is contained  in $\ccinterval{-\max\{ q_0, \ldots, q_{n-1}\}}{\sum \overline{s}-\min\{ q_0, \ldots, q_{n-1}\}}$.
This shows \eqref{eq-f_v} and hence the claim.

\medskip
\noindent
{\em Step 5.} By the above claim, we have $f(p_{\beta}) \in Z(G)$ for all $\beta \in \{0,1\}^{m_1}$ if and only if 
$\lambda_{\beta} \in Z(G)$ for all $\beta \in \{0,1\}^{m_1}$, which is equivalent to $z \in L$. 
The only remaining problem is that the word $\val(\mathcal{I})$ produces some ``garbage'' group elements $f(x)$ on 
positions $x$ that are not of the form $p_\beta$. Note that for every $g \in G \setminus Z(G)$, there is a generator $a_i \in \Sigma$ such that the commutator 
$[g,a_i]$ is non-trivial. We now produce from $\mathcal{I}$ an SLP $\mathcal{I}^{-1}$ such that $\val(\mathcal{I}^{-1})$ represents
the inverse element of $f \in G^{(\Z)}$, which is the mapping $g$ with $g(x) = f(x)^{-1}$ for all $x \in \mathbb{Z}$. To construct $\mathcal{I}^{-1}$,
we have to reverse every right-hand side of $\mathcal{I}$ and replace every occurrence of a symbol $a_0, \ldots, a_{n-1}, \tau, \tau^{-1}$ by
its inverse. 

It is easy to compute in \L for every $i \in \cointerval{0}{n}$ an SLP for the word 
\[ 
w_i := \big(a_i \tau^\pi \big)^{2^{m_1}} \tau^{-2^{m_1} \cdot \pi} .
\]
Then the group element represented by $w_i$ is the mapping $f_i \in G^{(\Z)}$  whose  support 
is the set of positions $p_\beta$ for $\beta \in \{0,1\}^{m_1}$ and 
$f_i(p_\beta) = a_i$ for all $\beta \in \{0,1\}^{m_1}$.
We can also compute in \L an SLP for the word $w_i^{-1}$.
We then built in \L SLPs $\mathcal{J}_0, \ldots, \mathcal{J}_{n-1}$ such that $\val(\mathcal{J}_i) = \val(\mathcal{I}^{-1}) w_i^{-1} \val(\mathcal{I}) w_i$.
Hence, the word $\val(\mathcal{J}_i)$ represents the group element $g_i \in G^{(\Z)}$, where $g_i(x) = 1$ for all 
$x \in \mathbb{Z}\setminus \{ p_\beta \mid \beta \in \{0,1\}^{m_1} \}$ and 
$g_i(p_\beta) = f(p_\beta)^{-1} a_i^{-1} f(p_\beta) a_i  = [f(p_\beta),a_i]$. 

Finally, we construct in \L an SLP $\mathcal{J}$ such that 
\[ 
\val(\mathcal{J}) = \val(\mathcal{J}_0)\, \tau \,\val(\mathcal{J}_1) \,\tau\, \val(\mathcal{J}_2) \cdots \tau \, \val(\mathcal{J}_{n-1}) \,\tau^{-n+1} .
\]
We can assume that $n \leq \ell+\sum \overline{r}_2 = \pi$ ($n$ is a constant and we can always make $\ell$ bigger).
Then $\val(\mathcal{J})$ evaluates to the group element $g \in G^{(\Z)}$ with $g(x) = 1$ for 
$x \in \mathbb{Z} \setminus \{ p_\beta - i \mid \beta \in \{0,1\}^{m_1}, 0 \leq i \leq n-1 \}$ and
$g(p_\beta - i) = g_i(p_\beta) = [f(p_\beta),a_i]$ for $0 \leq i \leq n-1$.
Hence, if $f(p_{\beta}) \in Z(G)$ for all $\beta \in \{0,1\}^{m_1}$ then 
$\val(\mathcal{J})=1$ in $G \wr \Z$. On the other hand, if there is a $\beta \in \{0,1\}^{m_1}$ such that
$f(p_{\beta}) \in G \setminus Z(G)$ then there is an $a_i$ such that 
$[f(p_\beta),a_i] \neq 1$. Hence $g(p_\beta - i) \neq 1$ and 
$\val(\mathcal{J})\neq 1$ in $G \wr \Z$. 
This proves the theorem.
\end{proof}
The following remark will be needed in the next section.

\begin{remark} \label{rem:Z/t}
Consider the SLP $\val(\mathcal{J})$ computed in the previous proof from the machine input $z$.
We showed that $z \in L$ if and only if $\val(\mathcal{J}) = 1$ in $G \wr \Z$. 
Let $s = |\val(\mathcal{J})|$; it is 
a number that grows exponentially with $|z|$.
The binary expansion of $s$
can be computed from $z$ in \L using simple arithmetics.  
Let $t$ be any positive integer with $t \geq 2s+1$. Then 
$\val(\mathcal{J}) = 1$ in $G \wr \Z$ if and only if $\val(\mathcal{J}) = 1$ in $G \wr (\Z/t)$
where in the latter equality $\tau$ is taken for the generator of $\Z/t$. To see this, note that
during the evaluation of $\val(\mathcal{J})$ in $G \wr \Z$ only the $G$-elements at positions in the interval $\ccinterval{-s}{s}$
(whose size is at most $t$) can be multiplied with a generator of $G$.
Intuitively, $\val(\mathcal{J})$ evaluates in $G \wr \Z$ in the same way as in $G \wr (\Z/t)$.
\end{remark}

\section{PSPACE-complete compressed word problems} \label{sec:PSPACE-compl-CWP}

In this section, we will use Theorem~\ref{thm-main-leaf2} (and \cref{rem:Z/t}) to show $\PSPACE$-completeness
of the compressed word problem for several groups.
For upper upper bounds, we will make use of the following simple lemma:

\begin{lemma} \label{lemma-CWP-wreath}
If $\WP(G)$ belongs to $\polyL$, 
then $\CompWP(G\wr\mathbb{Z})$ belongs to $\PSPACE$.
\end{lemma}

\begin{proof}
We use a result of Waack \cite{Waa90} according  to which  
the word problem for a wreath product $G_1 \wr G_2$ is $\NC^1$-reducible (and hence \L-reducible) to the word problems
for $G_1$ and $G_2$. Since  $\WP(G)$ belongs to $\polyL$ 
and $\WP(\mathbb{Z})$ belongs to \L, it follows that $\WP(G \wr \mathbb{Z})$ belongs to $\polyL$ ($\polyL$ is closed under \L-reductions).
Hence, by Lemma~\ref{lemma-cwp-PSPACE} the compressed word problem for $G \wr \mathbb{Z}$ belongs to 
$\PSPACE$.
\end{proof}
The following lemma generalizes the inclusion $\PSPACE \subseteq \LEAF(\WP(G))$ for $G$ finite non-solvable (where in fact equality holds)
from \cite{HLSVW93}. It can be proved directly using the same idea based on commutators as \cref{thm:uniNC1hard}. Here we follow a different approach and derive it by a padding argument from \cref{thm:uniNC1hard}.

\begin{lemma} \label{lemma-PSPACE-SENS}
If the finitely generated group $G$ is uniformly SENS, then $\PSPACE \subseteq \LEAF(\WP(G/Z(G)))$. 
\end{lemma}

\begin{proof}
    Let $L \subseteq \Gamma^*$ belong to $\PSPACE$.       
	Recall that $\PSPACE = \APTIME$. Hence, there is an ATM for $L$ with running time bounded by a polynomial $p(n)$. We can assume that $p(n) \ge n$ for all $n$.
	Now, consider the language 
	\[ \Pad_{2^{p(n)}}(L) = \set{v\$^{ 2^{p(\abs{w})} - \abs{v} }}{v \in L},\] where $\$$ is some fresh letter. Then $\Pad_{2^{p(n)}}(L)$ is in \ALOGTIME: 
	Let $w$ be the input word and let $n = |w|$ be the input length.
	First, we check whether $w \in \Gamma^* \$^*$ (the latter regular language even belongs to  uniform \Ac0). If not, we reject, otherwise we 
	can write $w =v\$^k$ for some $k\in \N$ and $v\in \Gamma^*$. Let $m = n-k = |v|$. We next have to verify that $n = 2^{p(m)}$.
	Using binary search, we compute in \DLOGTIME the binary representation of the input length $n$. If $n$ is not a power of two 
	(which is easy to check from the binary representation of $n$), then we reject. Otherwise, let $l = \log_2 n$. The unary representations of $l$
	can be obtained from the binary representation of $n$. It remains to check $l = p(m)$. Using $1^l$ we can check whether $|v| = m \leq l$. If not, we reject.
	Otherwise, we can produce $1^m$. Since polynomials are time constructible we can 
	simply run a clock for $p(m)$ steps, and stop if the number of steps exceeds $l$. Finally, we check whether $v \in L$ (by assumption this can be done in $\ATIME(p(|v|))$, which is contained in \ALOGTIME because of the increased input length). Thus, $\Pad_{2^{p(n)}}(L)$ is in \ALOGTIME.
			
	Since we aim for applying \cref{thm:uniNC1hard}, we have to encode every symbol $c \in \Gamma \cup \oneset{\$ }$ by a bit string $\gamma(c)$ of length $2^\mu$ for some fixed constant $\mu$.
       Hence, we consider the language $\gamma(\Pad_{2^{p(n)}}(L))$, which belongs to $\ALOGTIME$ as well.
	Observe that by \cref{lem:centerSENS}, also $G/Z(G)$ is uniformly SENS. Thus, we can apply \cref{thm:uniNC1hard}, which states that there is a uniform family
	$(P_n)_{n \in \N}$ of $(G/Z(G),\Sigma)$-programs of polynomial length recognizing $\gamma(\Pad_{2^{p(n)}}(L))$. Be aware, however, that ``polynomial'' here means 
	polynomial in the input length for $\gamma(\Pad_{2^{p(n)}}(L))$.
         Let $Q_n = P_{2^{p(n)+\mu}}$, which has length $2^{d(n)}$ for some function $d(n) \in \Oh(p(n))$.
	By the uniformity of $(P_n)_{n \in \N}$ we can compute $1^{d(n)}$ from $1^{2^{p(n)+\mu}}$ in $\DTIME(\Oh(\log(2^{p(n)+\mu}))) = \DTIME(\Oh(p(n)))$.
	Here we do not have to construct the unary representation of $2^{p(n)+\mu}$: recall that we have a random access Turing machine for the computation. One can easily
	check whether the content of the address tape (a binary coded number) is at most $2^{p(n)+\mu}$.
	
	Now, we construct an adequate NTM $M$ with $L =\LEAF(M,\WP(G/Z(G)))$: 
	on input $z \in \Gamma^*$ of length $n$ the machine $M$ produces a full binary tree of depth $d(n)$. In the  
         $i$-th leaf ($i \in \cointerval{0}{2^{d(n)}}$) it computes 
	the $i$-th instruction of $Q_n$. By the uniformity of 
	$(P_n)_{n \in \N}$  this can be done in  $\DTIME(\Oh(p(n)))$, so $M$ respects a polynomial time bound. Let $\langle j,a,b\rangle$ be the computed instruction.
	Here $j \in \ccinterval{1}{2^{p(n)+\mu}}$ is a position in $\gamma(z \$^{2^{p(n)} - n})$.
	Depending on the input bit at position $j$  in $\gamma(z \$^{2^{p(n)} - n})$ (which can be easily computed from $z$ and $j$ in polynomial time),
	the machine then outputs either $a$ or $b$. We then have $\leaf(M,z) = Q_n[\gamma(z \$^{2^{p(n)} - n})]$.
	  Thus, $z \in L$ iff $\gamma(z \$^{2^{p(n)} - n}) \in \gamma(\Pad_{2^{p(n)}}(L))$ iff $Q_n[\gamma(z \$^{2^{p(n)} - n})] \in \WP(G/Z(G))$ 
	  iff $\leaf(M,z) \in \WP(G/Z(G))$. 
\end{proof}
From \cref{thm-main-leaf2} and \cref{lemma-PSPACE-SENS} we get:

\begin{corollary} \label{thm-SENSE-PSPACE}
If $G$ is uniformly SENS, then $\CompWP(G\wr\mathbb{Z})$ is $\PSPACE$-hard. 
\end{corollary}
Since finite non-solvable groups and finitely generated free group of rank at least two are uniformly SENS 
and their word problems can be solved in \L (see \cite{LiZa77} for the free group case), we obtain the following
from \cref{lemma-CWP-wreath} and \cref{thm-SENSE-PSPACE}:

\begin{corollary} \label{coro-SENS-wreath-PSPACE}
If $G$ is a finite non-solvable group or a finitely generated free group of rank at least two,
then $\CompWP(G\wr\mathbb{Z})$  is $\PSPACE$-complete.
\end{corollary}
For Thompson's group $F$ we have $F \wr \mathbb{Z} \leq F$ (Lemma~\ref{lem:GS}). 
Moreover, $F$ is uniformly SENS (Corollary~\ref{cor:Thompson}).
Finally, Lehnert and Schweitzer have shown that $F$ is co-context-free, i.e., the complement 
of the word problem of $F$ (with respect to any finite generating set) 
is a context-free language \cite{LehSchw07}.
This implies that the word problem for $F$ belongs to the complexity class $\LogCFL$
(the closure of the context-free languages under \L-reductions). It is known that 
$\LogCFL \subseteq \DSPACE(\log^2 n)$ \cite{LewisSH65}.
If we put all this into Theorem~\ref{thm-SENSE-PSPACE}, we get:

\begin{corollary} \label{coro-Thompson-PSPACE}
The compressed word problem for Thompson's group $F$ is $\PSPACE$-complete.
\end{corollary}
In rest of  the section we prove that the compressed word problem for some weakly branched groups
(including the Grigorchuk group and the Gupta-Sidki groups) is $\PSPACE$-complete as well.
We need the following lemma.

\begin{lemma} \label{lemma-embedding-b^n}
Let $G$ be a finitely generated group with the standard generating set $\Sigma$ 
such that $G \wr (\Z/p) \leq G$ for some
$p \geq 2$.
Let $\tau_n$ be a generator for the cyclic group $\Z/p^n$ for $n \geq 1$.
Then $G \wr (\Z/p^n) \leq G$ for every $n \geq 1$, and 
given $n$ in unary encoding  and $a \in \Sigma \cup \{ \tau_n, \tau_n^{-1} \}$ one can compute in \L an SLP $\mathcal{G}_{n,a}$ over the 
terminal alphabet $\Sigma$
such that the mapping $a \mapsto \val(\mathcal{G}_{n,a})$ $(a \in \Sigma \cup \{ \tau_n, \tau_n^{-1}  \})$ induces an embedding of 
$G \wr (\Z/p^n)$ into $G$. 
\end{lemma}

\begin{proof}
We fix an embedding $\phi_1 : G \wr (\Z/p) \to G$.
We prove the lemma by induction on $n$. The case $n=1$ is clear.
Consider $n \geq 2$ and assume that we have the embedding $\phi_{n-1} : G \wr (\Z/p^{n-1}) \to G$.
We show that 
\[ G \wr (\Z/p^n) = G \wr \langle \tau_n \rangle 
\leq (G \wr \langle \tau_{n-1} \rangle) \wr \langle \tau_1 \rangle = 
(G \wr (\Z/p^{n-1})) \wr (\Z/p)  
\] 
via an embedding $\psi_n$.
For this we define $\psi_n(g) = g \in G \leq G \wr (\Z/p^{n-1})$ for $g \in G$  and 
$\psi_n(\tau_n) = \tau_{n-1} \tau_1$. It is easy to see that this defines indeed an embedding.
The element $\tau_{n-1} \tau_1$ generates a copy of $\Z/p^{n}$ by cycling through 
$p$ copies of $\Z/p^{n-1}$ and incrementing mod $p^{n-1}$ the  current
$\Z/p^{n-1}$-value.

We extend the embedding $\phi_{n-1} : G \wr (\Z/p^{n-1}) \to G$ to an embedding 
\[ \phi_{n-1} : (G \wr (\Z/p^{n-1})) \wr (\Z/p) \to G \wr (\Z/p)\]
by letting  $\phi_{n-1}$ operate as the identity mapping on the right factor $\Z/p$.
Finally, we can define $\phi_n : G \wr (\mathbb{Z}/p^{n}) \to G$ by
$\phi_n = \psi_n \circ \phi_{n-1} \circ \phi_1$, where composition is executed from left to right.
We get 
\[ 
\phi_n(\tau_n) = \phi_1(\phi_{n-1}(\psi_n(\tau_n))) = \phi_1(\phi_{n-1}(\tau_{n-1} \tau_1)) = \phi_1(\phi_{n-1}(\tau_{n-1})) \phi_1(\tau_1).
\]
and $\phi_n(g) = \phi_1(\phi_{n-1}(\psi_n(g))) = \phi_1(\phi_{n-1}(g))$.
By induction on $n$ we get
\[ 
\phi_n(\tau_n) = \phi_1^{n}(\tau_1) \phi_1^{n-1}(\tau_1) \cdots \phi_1^{2}(\tau_1) \phi_1(\tau_1) .
\]
and $\phi_n(g) = \phi^n_1(g)$ for $g \in G$. 
Lemma~\ref{lemma-SLP-iterated-morph}  implies that given $n$ in unary encoding 
we can compute in \L SLPs for $\phi_n(\tau_n)$ and all $\phi_n(g)$ ($g \in G$).
\end{proof}
Using Lemma~\ref{lemma-embedding-b^n} we can show the following
variant of Theorem~\ref{thm-main-leaf2}.


\begin{theorem} \label{thm-PSPACE-Z_p}
	Let $G$ be a finitely generated group such that $G \wr (\Z/p) \leq G$ for some $p \geq 2$. Then $\CompWP(G)$ 
	is $\PSPACE$-hard. 
\end{theorem}

\begin{proof}
	By \cref{thm:wreathSENS}, every group with $G \wr (\Z/p) \leq G$ is SENS. Hence, by \cref{lemma-PSPACE-SENS}, $\PSPACE \subseteq \LEAF(\WP(G/Z(G)))$ and it suffices to show that $\CompWP(G)$ is hard for the complexity class $\forall \LEAF(\WP(G/Z(G)))$.
	
Consider a language $L \in \forall \LEAF(\WP(G/Z(G)))$ and an input word $z$ of length $n$.
Let  $\mathcal{J}$ be the SLP that we computed in the proof of Theorem~\ref{thm-main-leaf2} in \L from $z$.
We showed that $z \in L$ if and only if $\val(\mathcal{J}) = 1$ in $G \wr \mathbb{Z}$. Let $s = |\val(\mathcal{J})|$; it is 
a number in $2^{n^{\Oh(1)}}$. Hence, we can choose a fixed polynomial $q$ such that $p^{q(n)} \geq 2s+1$ for all input lengths $n$.
Let $m = q(n)$. By Remark~\ref{rem:Z/t} we have $z \in L$ if and only if $\val(\mathcal{J}) = 1$ in $G \wr (\mathbb{Z}/p^m)$.

From $1^m = 1^{q(n)}$ (which can be constructed in \L) we can compute by Lemma~\ref{lemma-embedding-b^n} 
for every $a \in \Sigma \cup \{ \tau_m, \tau_m^{-1}  \}$ an SLP $\mathcal{G}_{m,a}$ 
over the terminal alphabet $\Sigma$ such that the mapping $a \mapsto \val(\mathcal{G}_{m,a})$ 
$(a \in \Sigma \cup \{ \tau_m, \tau_m^{-1}  \})$ induces an embedding of the wreath product 
$G \wr (\Z/p^m)$ into $G$. Note that $\log m \in \mathcal{O}(\log n)$. Hence, the space needed for the construction
of the $\mathcal{G}_{m,a}$ is also logarithmic in the input length $n$.
We can assume that the variable sets of the SLPs $\mathcal{G}_{m,a}$ ($a \in \Sigma \cup \{ \tau_m, \tau_m^{-1}  \}$) and $\mathcal{J}$ are pairwise disjoint. Let
$S_{m,a}$ be the start variable of $\mathcal{G}_{m,a}$. We construct an SLP $\mathcal{G}$ by taking the union of the SLPs 
$\mathcal{G}_{m,a}$ ($a \in \Sigma \cup \{ \tau_m, \tau_m^{-1} \}$) and $\mathcal{J}$
and replacing in every right-hand side of $\mathcal{J}$ every occurrence of a terminal symbol $a$ by $S_{m,a}$. 
We have $\val(\mathcal{G})=1$ in $G$ if and only if $\val(\mathcal{J}) = 1$ in $G \wr (\Z/p^{m})$ if and only if 
$z \in L$.
\end{proof}


Let us now come to weakly branched groups. We restrict ourselves to weakly branched groups $G$ whose branching subgroup $K$ is not torsion-free.
\begin{lemma}\label{lem:wbtorsion=>wreath}
  Let $G$ be a weakly branched group whose branching subgroup $K$ contains elements of finite order. Then $K$ contains $K\wr (\Z/p)$ for some $p\ge2$.
\end{lemma}
\begin{proof}
  Let $k\in K$ be an element of finite order. Up to replacing $k$ by a
  power of itself, we may assume $k$ has prime order $p$. In
  particular, there exists a vertex $v\in X^*$ whose orbit under $k$
  has size $p$. Then $\langle v*K,k\rangle\cong K\wr (\Z/p)$ is
  the desired subgroup.
\end{proof}
The following result applies in particular to the Grigorchuk group and the Gupta-Sidki groups:
\begin{corollary} \label{coro-weakly-trosion-PSPACE-CWP}
  Let $G$ be a weakly branched torsion group whose
  branching subgroup is finitely generated.
  \begin{itemize}
\item $\CompWP(G)$ is $\PSPACE$-hard.
\item If $G$ is also contracting, then $\CompWP(G)$ is $\PSPACE$-complete.
\end{itemize}
\end{corollary}

\begin{proof}
By Lemma~\ref{lem:wbtorsion=>wreath} the branching subgroup $K$ of $G$
  satisfies the hypotheses of Theorem~\ref{thm-PSPACE-Z_p}, so the
  compressed word problem for $K$ (and hence $G$) is
  $\PSPACE$-hard.

If $G$ is also contracting, then the word problem of $G$ is in $\L$ by
  Proposition~\ref{prop:contractingL}, so
  Lemma~\ref{lemma-cwp-PSPACE} implies that $\CompWP(G)$ belongs to $\PSPACE$.
\end{proof}
Corollary~\ref{coro-weakly-trosion-PSPACE-CWP} gives new (and natural) examples for groups where the compressed
word problem is provably more difficult than the word problem (since $\L$ is a proper subset of $\PSPACE$).
The first example for such a group was provided in \cite{WaeWei19}: it is an automaton group where the word problem
is $\PSPACE$-complete and the compressed word problem is $\EXPSPACE$-complete.
Let us also remark, that the Grigorchuk group is an example of a group where the compressed word problem 
is even more difficult than the power word problem. For the power word problem \cite{LoWe19} the input consists 
of a word $w_1^{z_1} w_2^{z_2}\cdots w_n^{z_n}$, where the exponents $z_i$ are given in binary representation
and the $w_i$ are explicitly given words over the group generators. In terms of complexity, the power word problem lies
between the word problem and the compressed word problem. It is shown in \cite{LoWe19} that the power word
problem for the Grigorchuk group belongs to $\L$, whereas by Corollary~\ref{coro-weakly-trosion-PSPACE-CWP}
the compressed word problem is $\PSPACE$-complete.

\section{Conclusion and open problems}

We have added an algorithmic constraint (uniformly SENS) to the algebraic notion of being a non-solvable group, which implies that the word problem is \Nc1-hard (resp. \ALOGTIME-hard). 
Using this, we produced several new examples of non-solvable groups with an \ALOGTIME-hard word problem.
However, the question remains open whether all non-solvable groups have \ALOGTIME-hard word problem, even if they are not ENS. 
We showed that for every contracting self-similar group the word problem belongs \L. Here, the question remains whether there exists
a contracting self-similar group with a \L-complete word problem.  In particular, is the word problem for the Grigorchuk group
\L-complete? (we proved that it is \ALOGTIME-hard). Also the precise complexity of the word problem for Thompson's group $F$ is
open. It is \ALOGTIME-hard and belongs to \LOGCFL; the latter follows from \cite{LehSchw07}. In fact, from the proof in \cite{LehSchw07} one 
can deduce that the word problem for $F$ belongs to \LOGDCFL (the closure of the deterministic context-free languages with respect to \L-reductions).

\def\cprime{$'$} \def\cprime{$'$}

\end{document}